\documentclass[10pt]{article}

\usepackage[dvipdf]{graphicx}
\usepackage[latin1]{inputenc}
\usepackage{latexsym}
\usepackage{amsmath,amssymb,amsfonts,amsthm}
\usepackage{xcolor}
\usepackage{mathrsfs}
\usepackage{bbm}

\usepackage{enumerate}
\numberwithin{equation}{section}

\usepackage[numbers,comma,sort]{natbib}
\usepackage{comment}
\usepackage{stmaryrd}
\definecolor{db}{RGB}{0, 0, 130}
\usepackage[colorlinks=true,citecolor=red,linkcolor=db,urlcolor=blue,pdfstartview=FitH]{hyperref}
\usepackage{dsfont}
\usepackage{pdfsync}

\newcommand{\R}{\mathbb{R}}

\newcommand{\N}{\mathbb{N}}
\newcommand{\EE}{\mathbb{E}}

\newcommand{\bqn}{\begin{equation}}
\newcommand{\eqn}{\end{equation}}
\newcommand{\bqne}{\begin{equation*}}
\newcommand{\eqne}{\end{equation*}}

\newcommand{\argmin}{\operatorname{argmin}}

\newcommand{\hatB}{\widehat{B}}

\DeclareMathOperator{\Id}{Id}
\DeclareMathOperator{\Var}{Var}

\makeatletter
\newcommand{\customlabel}[2]{%
   \protected@write \@auxout {}{\string\newlabel {#1}{{#2}{\thepage}{#2}{#1}{}}}%
   \hypertarget{#1}{#2\hspace{-0.14cm}}
}
\def\namedlabel#1#2{\begingroup
    #2%
    \def\@currentlabel{#2}%
    \phantomsection\label{#1}\endgroup
}
\makeatother

\newtheorem{definition}{Definition}[section]
\newtheorem{theorem}[definition]{Theorem}
\newtheorem{prop}[definition]{Proposition}

\newtheorem{lemma}[definition]{Lemma}

\newtheorem{proposition}[definition]{Proposition}
\newtheorem{remark}[definition]{Remark}

\author{El Mehdi Haress\footnote{Universit\'e Paris-Saclay, CentraleSup\'elec, MICS and CNRS FR-3487. \texttt{el-mehdi.haress@centralesupelec.fr}} \and 
Alexandre Richard\footnote{Universit\'e Paris-Saclay, CentraleSup\'elec, MICS and CNRS FR-3487. \texttt{alexandre.richard@centralesupelec.fr}
\newline 
 E.H. acknowledges the support of the Labex Math\'ematiques Hadamard. This work is supported by the SIMALIN project ANR-19-CE40-0016 from the French National Research Agency.
}}

\title{ \Large{\textbf{Estimation of several parameters in discretely-observed Stochastic Differential Equations with additive fractional noise}}}

\begin{document}

\maketitle

\begin{abstract}
We investigate the problem of joint statistical estimation of several parameters for a stochastic differential equation driven by an additive fractional Brownian motion. Based on discrete-time observations of the model, we construct an estimator of the Hurst parameter, the diffusion parameter and the drift, which lies in a parametrised family of coercive drift coefficients. Our procedure is based on the assumption that the stationary distribution of the SDE and of its increments permits to identify the parameters of the model. Under this assumption, we prove consistency results and derive a rate of convergence for the estimator. Finally, we show that the identifiability assumption is satisfied in the case of a family of fractional Ornstein-Uhlenbeck processes and illustrate our results with some numerical experiments.
\end{abstract}

\section{Introduction}\label{sec:intro}
Consider the following $\mathbb{R}^d$-valued stochastic differential equation
\begin{align}\label{eq:fsde0}
   Y_t = Y_0 + \int_0^t b_{\xi_0}(Y_s) d s + \sigma_0 B_t,
\end{align}
where $B$ is an $\mathbb{R}^d$-fractional Brownian motion (fBm) with Hurst parameter $H_0 \in (0,1)$. The goal in this work is to estimate simultaneously the parameter $\xi_{0}$, the diffusion coefficient $\sigma_{0}$ and the Hurst parameter $H_0$ from discrete observations of the process $Y$. We will assume that the drift parameter $\xi_0$ lies in a set $\Xi$ of $\mathbb{R}^m$ and $\{b_\xi(\cdot), \xi \in \Xi \}$ is a parametrised family of drift coefficients with $b_\xi(\cdot): \mathbb{R}^d \rightarrow \mathbb{R}^d$, and $\sigma_0$ is an invertible $\mathbb{R}^{d \times d}$ matrix. The unknown parameters are denoted by $\theta_0 = (\xi_{0},\sigma_{0},H_{0}) \in \mathbb{R}^{q+1}$, where $q=m+d^2$. 

In the framework of SDEs driven by fBm, many recent works have focused on the parametric estimation of the drift, mostly assuming that the process $Y$ is observed continuously and that the parameters $H$ and $\sigma$ are known (see e.g \cite{belfadli2011parameter,
hu2010parameter,rao2011statistical,tudor2007statistical,
hu2019parameter}). These works propose estimators of $\xi_{0}$ which are strongly consistent, providing a rate of convergence towards $\xi_{0}$ and even sometimes a central limit theorem \cite{hu2010parameter,hu2019parameter}. In these works, the drift function is of the form $b_\xi(y)=-\xi y$, i.e. a family of Ornstein-Uhlenbeck (OU) processes, or of the form $b_\xi(y)=\xi b(y)$ as in \cite{tudor2007statistical}. In addition, the process $Y$ is observed in continuous time. In practical situations though, we only have access to discrete-time observations. Taking into account this constraint, two recent papers \cite{panloup2020general,hu2013parameter} constructed estimators of $\xi_{0}$ which were proven to be strongly consistent. Their rate of convergence is studied and a central limit theorem is also proven in \cite{hu2013parameter}: while \cite{hu2013parameter} considers the fractional OU case, \cite{panloup2020general} treats general drift functions which satisfy a coercivity assumption.

The diffusion coefficient $\sigma_{0}$ is usually estimated using the quadratic variations of $Y$, which is possible only when the process is either observed continuously or the step-size goes to zero (i.e high frequency data), see \cite{xiao2011parameter} and \cite{berzin2015variance}. The Hurst parameter $H_{0}$ is also estimated using quadratic variations, see e.g.  \cite{kubilius2012rate}, or by a direct access to discrete observations of a fractional Brownian motion path with a step-size that goes to zero as in \cite{gloter2007estimation}.

When it comes to estimating all the parameters $(\xi_{0},\sigma_{0},H_{0})$, we refer to \cite{brouste2013parameter} where the observations are assumed to be made continuously, and \cite{haress2020estimation} which is, to the best of our knowledge, the only work which estimates all the parameters of a fractional Ornstein-Uhlenbeck process in a discrete-time setting. 

~

In this paper, we consider an ergodic setting that allows for \eqref{eq:fsde0} to have a stationary distribution for any $\theta_0 \in \Theta$.  We work with the assumption that \emph{the stationary distribution of $Y$ identifies the parameters}, as initiated in \cite{panloup2020general}. However, as illustrated by the authors of \cite{haress2020estimation}, in the simple case of a one-dimensional fractional OU process, this claim is false for more than one parameter to estimate. In fact, the stationary distribution of $Y$ is Gaussian and therefore distinguished by its mean (which does not depend on the parameters) and its variance. In this case, the variance itself cannot identify the three parameters. In \cite{haress2020estimation}, this issue is circumvented by considering the increments of $Y$; the increments of the stationary solution are also Gaussian but have different variances. Thus, adding two increments, the authors have access to three functions and show that these functions are sufficient to estimate the parameters. We propose here to generalise the approach presented in \cite{haress2020estimation}. We add $q$ linear transformations of the original process and assume that they are enough to identify the parameters. Therefore, our assumption (which is detailed later) will be that \emph{the stationary distribution of $Y$ and its increments identify the parameters $(\xi,\sigma,H)$.}
 
 Assume for simplicity that the observations are of the form $(Y_{kh}^{\theta_{0}})_{k=0,\dots,n+q}$ and consider $q$ linear transformations $\{ \ell^i (Y^{\theta_{0}}_{kh}, \dots, Y^{\theta_{0}}_{kh+ih} ) \}_{k=0,\dots,n}$ where $i \in \llbracket 1,q \rrbracket$. Hence, we now have access to $q+1$ paths, which we use to define the path of a higher-dimensional process $X^{\theta_{0}}$ that we call the  augmented process associated to the SDE \eqref{eq:fsde0}. With access to a path of $X^{\theta_{0}}$, we construct the estimator of $\theta_{0}$ by
 \begin{align}\label{eq:thetan-0}
 \hat{\theta}_n= \underset{\theta \in \Theta}{\argmin} \ d\left(\frac{1}{n} \sum_{k=0}^{n-1} \delta_{X_{kh}^{\theta_{0}}}, \mu_{\theta}\right),
 \end{align}
 where $\mu_{\theta}$ is the stationary distribution of $X^\theta$. We prove that $\hat{\theta}_n$ is a strongly consistent estimator of $\theta_0$ and obtain a rate of convergence. 

~

In \cite{haress2020estimation}, the authors provided numerical evidence of the identifiability assumption (i.e the fact that the stationary distribution of $Y$ and its increments identify the parameters). We prove here that in the setting of \cite{haress2020estimation}, i.e. of a fractional OU process, the aforementioned identifiability assumption holds. Also, as in \cite{panloup2020general}, we consider two variations of this assumption, a weak one which we will just call the identifiability assumption and a strong one. Moreover, to construct an estimator of the drift parameter $\xi$,  the authors of \cite{panloup2020general} proved beforehand results on the regularity of $Y$ with respect to $\xi$. This is a natural procedure, since the estimation method relies on minimizing a certain functional of $Y$, by showing that it has enough regularity so that its minimum is attained at the true parameter $\xi_0$. Here, in view of estimating all the parameters, we will will study the regularity of $Y$ with respect to $\xi, \sigma$ and $H$.

~

Since we are interested in ergodic estimators, we need the regularity of $Y$ in all the parameters to be uniform in time. In particular we need the regularity in $H$ to hold uniformly in $t\geq 0$. To achieve this, the drift will be assumed to be contractive. 

Let us mention that the sensitivity in the Hurst parameter has been studied in various situations and is an important topic in modeling. The fBm is known to be infinitely differentiable w.r.t its Hurst parameter (see \cite{Neuenkirch}). In addition, other functionals of the fBm were considered. In \cite{jolis2010continuity,jolis2007continuity}, the law of the integral w.r.t the fBm is proven to be continuous in $H$;  
in \cite{richard2015fractional}, the H\"older continuity in is obtained for generalised fractional Brownian fields; and in \cite{GiordanoJolisQS}, the law of stochastic heat and wave equations with additive fractional noise is proven to be continuous in $H$.
Let us also mention that in \cite{richard2016h}, the law of functionals of fractional SDEs is proven to be Lipschitz continuous around its Markovian counterpart ($H=\frac{1}{2}$), including irregular functionals such as the law of the first hitting time (see also \cite{RT2017} for a numerical approach and applications, in particular in neuroscience).

In this work, new results on the Hurst regularity of fractional models were needed, and they have been gathered in a separate paper \cite{HRarxiv}. \\

In the formula \eqref{eq:thetan-0}, the stationary distribution $\mu_{\theta}$ is generally unknown, except in some simple cases like for Ornstein-Uhlenbeck processes. This means that the estimator cannot be implemented. This problem can be solved by considering a numerical approximation of $\mu_{\theta}$ via an Euler scheme $Y^{\theta, \gamma}$ of time-step $\gamma >0$. Given $N+q$ simulated points of the form $(Y^{\theta,\gamma}_{k \gamma} )_{k=0,\dots, N+q}$, we consider as before $q$ linear transformations $\{ \ell^i ( Y^{\theta,\gamma}_{k \gamma}, \dots, Y^{\theta,\gamma}_{k \gamma + i \gamma}) \}_{k=0, \dots, N}$, which we use to define a higher-dimensional process $X^{\theta, \gamma}$. We then define the estimator of $\theta_0$ by

\begin{align}\label{eq:thetan-02}
\hat{\theta}_{N,n,\gamma} = \underset{\theta \in \Theta}{\argmin} \ d\left(\frac{1}{n} \sum_{k=0}^{n-1} \delta_{X_{kh}^{\theta_{0}}}, \frac{1}{N} \sum_{k=0}^{N-1} \delta_{X^{\theta,\gamma}_{k \gamma}}  \right).
\end{align}
We prove that $\hat{\theta}_{N,n,\gamma} $ is a strongly consistent estimator of $\theta_0$ and obtain a rate of convergence. 

\paragraph{Organisation of the paper.} In Section \ref{sec:proced}, we first detail the notations and some assumptions in Section \ref{subsec:not-assump}. Then in Section \ref{subsec:inv_est} we explain how to approximate the invariant measure via an Euler scheme in order to implement the estimator \eqref{eq:thetan-02}. We present the main results for the estimators \eqref{eq:thetan-0} and \eqref{eq:thetan-02} in Section \ref{subsec:results}. In Section \ref{sec:app-OU}, we show that our results can be applied to the fractional Ornstein-Uhlenbeck process and to small perturbations of this process. In Section \ref{sec:consist}, we prove the strong consistency of the estimator \eqref{eq:thetan-0} and its rate of convergence. In Section \ref{sec:practical}, we prove the strong consistency of the estimator \eqref{eq:thetan-02} and its rate of convergence. In Section \ref{subsec:simplif} and \ref{sec:Hknown-or-xiknown}, we prove that the identifiability assumption holds in the case of a fractional Ornstein-Uhlenbeck process for the estimation of two parameters, and in Section \ref{sec:OU-perturbation}, we exhibit a more general family of SDEs that verifies a stronger identifiability assumption for the estimation of one parameter. We also implement our method and run numerical simulations in Section \ref{sec:numerical}. In the Appendix~\ref{app:regH}, we recall some results from our companion paper \cite{HRarxiv}. In Appendix~\ref{appendix-L}, we prove continuity and tightness results on $Y$ and the solution of the Euler scheme associated to \eqref{eq:fsde0}. Finally Appendix~\ref{subsec:proofof4.1} is dedicated to the proof of Proposition~\ref{prop:dconv-euler}.

\section{A general procedure}\label{sec:proced}

We first give some general notations. Then we state the assumptions on the coefficients of \eqref{eq:fsde0} and define the estimator. At the end of this section, we give an almost sure convergence for this estimator result as well as a convergence rate.

\subsection{Notation and assumptions}\label{subsec:not-assump}
\paragraph{Notations.} Let $\mathcal{M}_1(\mathbb{R}^d)$ denote the set of probability measures on $\mathbb{R}^d$. For any given $p$, we will consider the $p$-Wasserstein distance, which is defined for every $\mu,\nu$ in $\mathcal{M}_1(\mathbb{R}^d) $ as follows:
\begin{align*}%
    \mathcal{W}_p (\mu, \nu) = \inf\{\left(\EE| X - Y |^p\right)^{\frac{1}{p}}; \mathcal{L}(X)= \mu, ~ \mathcal{L}(Y) = \nu \} .
\end{align*}
We denote by $\mathcal{D}_p$ the set of distances dominated by the $p$-Wasserstein distance. As in \cite{panloup2020general}, we will also work with the distance $d_{CF,p} \in \mathcal{D}_1$ defined for $p>( \frac{d}{2} \vee 1)$ as
\begin{align}\label{eq:d_cf}
    d_{CF,p}(\mathcal{L}(X),\mathcal{L}(Y)) = \left(\int_{\R^d} (\EE[e^{i\langle \chi, X \rangle}] - \EE[e^{i\langle \chi, Y \rangle}])^2 g_p(\chi) d \chi\right)^{1/2}  ,
\end{align}
where $g_p$ is the integrable kernel given by
\begin{align}\label{eq:gp}
g_p(\chi)=c_p (1+| \chi |^2)^{-p}  ,
\end{align} 
and $c_p = (\int_{\mathbb{R}^d} (1+| \chi |^2)^{-p} d\chi )^{-1}$ is a normalizing constant. 

We denote by $\N^*$ the set $\N \backslash \{0 \}$ and by $C$ a constant that can change from line to line and that does not depend on time and the parameters $\xi,\sigma,H$. When we want to make the dependence of $C$ on some other parameter $a$ explicit, we will write $C_a$.

The $\R^d$-fBm will be denoted by $B$, or by $B^H$ if we need to emphasize on the Hurst parameter $H$ of the process. Whenever we compare, on the same probability space, two fBm with different Hurst parameters $B^{H_{1}}$ and $B^{H_{2}}$, it is assumed that they are built from the same Brownian motion $W$ by the Mandelbrot-Van Ness formula:
\begin{equation}\label{eq:MVN}
B^{H_{i}}_{t} = \frac{1}{\Gamma(H_{i}+\frac{1}{2})} \int_{\R} \left((t-s)_{+}^{H_{i}-\frac{1}{2}} - (-s)_{+}^{H_{i}-\frac{1}{2}}\right) dW_{s}, \quad t\geq 0,\ i=1,2.
\end{equation}

\paragraph{Assumptions.} First, we assume that the number of unknown parameters $q+1$ is such that $q \ge 1$ (we have at least two unknowns), which is decomposed into $m$ parameters for the drift $b_{\xi_0}$, $ \xi_0 \in \Xi \subset \mathbb{R}^m$, $d^2$ parameters for $\sigma \in \mathbb{R}^{d \times d}$ and the last one which is the Hurst parameter. The next assumption states the compactness of the spaces where the parameters lie.

\begin{itemize}
\item[\namedlabel{asmp:compact}{{$\mathbf{A_0}$}}.] $\Xi$ is compactly embedded in $\mathbb{R}^m$ for a given $m \ge 1$. $H_0$ belongs to $\mathcal{H}$, a compact subset of $(0,1)$. The diffusion matrix $\sigma_0$ belongs to $\Sigma$ a compact set of $d\times d$-invertible matrices.
\end{itemize}
Therefore, we have that $\Theta = \Xi \times \Sigma \times \mathcal{H}$ is a compact subset of $\mathbb{R}^{q+1}$. We will also assume a coercivity assumption on the drift $b$.

\begin{itemize}
\item[\namedlabel{asmp:drift}{{$\mathbf{A_1}$}}.] $b \in \mathcal{C}^{1,1}(\mathbb{R}^d \times \Xi;\mathbb{R}^d)$ and there exist constants $\beta,K, c > 0$ and $r \in \mathbb{N}$ such that\\
(i) For every $x,y \in \mathbb{R}^d$ and $\xi \in \Xi$, we have
\begin{align}\label{eq:drift-coerciv}
    \langle b_\xi(x) - b_\xi(y),  x-y \rangle \leq -\beta | x - y |^2 ~\text{and}~ | b_\xi(x) - b_\xi(y) | \leq K | x-y |.
\end{align} 
(ii) For every $x \in \mathbb{R}^d$ and $\xi_1, \xi_2 \in \Xi$, the following growth bound is satisfied:
\begin{align}\label{eq:drift-growth}
    |  b_{\xi_1}(x)- b_{\xi_2}(x) | \leq c (1+| x |^r).
\end{align}
\end{itemize}
For $\theta = (\xi,\sigma,H) \in \Theta$, we denote by $Y^{\theta}$ the unique solution of the following equation
\begin{align}\label{eq:fsde}
        Y^\theta_t = Y_0 + \int_0^t b_{\xi}(Y^\theta_s) d s + \sigma B_t ,
\end{align}
where $Y_{0}\in \R^d$ and $B$ is an fBm of Hurst parameter $H$. Under \ref{asmp:drift}, \cite{Hairer} (see also \cite[Remark 2.4]{panloup2020general} and the references therein) gives the existence and uniqueness of the invariant measure to \eqref{eq:fsde}. 
We denote by $\bar{Y}^{\theta}$ the unique stationary solution and by $\nu_{\theta}$ its marginal distribution. For each $i \in \llbracket 1,q \rrbracket$, let $\ell^i$ be a linear transformation from $\left( \mathbb{R}^{d} \right)^{i+1}$ to $\mathbb{R}^d$.

Let us define the following processes for all $i \in \llbracket 1, q \rrbracket $:
\begin{align}%
    Z^{i,\theta}_. & = \ell^i(Y^{\theta}_\cdot,\dots,Y^{\theta}_{\cdot+ih}  ) \nonumber \\
    \bar{Z}^{i,\theta}_\cdot & = \ell^i(\bar{Y}^{\theta}_\cdot,\dots,\bar{Y}^{\theta}_{\cdot+ih}) \nonumber \\
    \label{eq:defX}
    X^{\theta}_\cdot & = (Y^{\theta}_\cdot, Z^{1,\theta}_\cdot,\dots,Z^{q,\theta}_\cdot)\\
    \label{eq:defXbar}
    \bar{X}^{\theta}_.\cdot& = (\bar{Y}^{\theta}_\cdot, \bar{Z}^{1,\theta}_\cdot,\dots,\bar{Z}^{q,\theta}_\cdot) .
\end{align}
Typical linear transformations considered in applications (see the discussion in Section~\ref{sec:app-OU}) will be the simple increments
\begin{align}\label{eq:simpleinc}
\ell^i(U^{\theta_0}_\cdot,\dots,U_{.+ih}^{\theta_0}) = U_{.+ih}^{\theta_0}-U_\cdot^{\theta_0}, \ \ \ i=1,\dots,q.
\end{align}

Observe that for all $\theta \in \Theta$ and $i \in \llbracket 1, q \rrbracket$, the processes $\bar{Z}^{i,\theta}$ and $\bar{X}^{\theta}$ are stationary. Denote by $\mu_{\theta}$ the law of $\bar{X}^{\theta}$. For simplicity, we  will not write the parameter $\theta$ on the processes when $\theta$  is the true parameter $\theta_0$. The triangle inequality yields the following inequalities for all $\theta, \theta' \in \Theta$ and $p>0$,
\begin{align}\label{eq:boundsXtoY}
| X^\theta_\cdot |^p & \leq C_{p,q} \Biggl( \sum_{i=0}^{q} |Y^\theta_{\cdot+ih} |^p \Biggr)\nonumber \\
| X^\theta_\cdot - X^{\theta'}_\cdot|^p & \leq C_{p,q} \Biggl( \sum_{i=0}^{q} |Y^\theta_{\cdot+ih} - Y^{\theta'}_{\cdot+ih} |^p \Biggr) \\
| X^\theta_\cdot - \bar{X}^{\theta}_\cdot|^p & \leq C_{p,q} \Biggl( \sum_{i=0}^{q} |Y^\theta_{\cdot+ih} - \bar{Y}^{\theta}_{\cdot+ih} |^p \Biggr)  , \nonumber
\end{align}
where $C_{p,q}$ is a constant that  do not depend on $\theta$ or $\theta'$. More precisely, $C_{q,p} \sim 2^p q^p L^p$, where $L$ is the maximum of the Lipschitz constants of the applications $\ell^i$ for $i \in \llbracket 1,q \rrbracket$. This means that upper bounds on $X$ will be obtained by bounding $Y$, and the regularity of the process $X$ will be studied through the regularity of the process $Y$.

As was highlighted previously in the introduction, the estimators are defined by assuming that $\mu_{\theta}$ characterizes $\theta$. This weak identifiability hypothesis reads as follows:
\begin{itemize}
\item[\namedlabel{asmp:identif}{{$\mathbf{I_w}$}}.] For any $\theta$ in $\Theta$,
\begin{align}\label{eq:identif-cond}
    \mu_{\theta}=\mu_{\theta_0} \Longleftrightarrow \theta=\theta_0,
\end{align}
where we recall that $\mu_\theta$ is the stationary distribution of $\bar{X}^\theta$.
\end{itemize}
\begin{remark}
A similar assumption is considered in \cite{panloup2020general} based on the stationary distribution of $\bar{Y}$: assume that $\nu_{\theta}=\nu_{\theta_0} \text{ iff } \theta=\theta_0$. Assumption~\ref{asmp:identif} is weaker, in the sense that it is satisfied in situations where the assumption from \cite{panloup2020general} is not (because we consider the process $X^\theta$ instead of $Y^\theta$). Indeed, let $\theta,\theta_0 \in \Theta$ such that $\mu_{\theta}=\mu_{\theta_0}$. This implies $d_{CF,p}(\mu_\theta,\mu_{\theta_0})=0$. Using the definition of $d_{CF,p}$, we have
\begin{align*}
    \text{for almost all}~ \chi_{q} \in \mathbb{R}^{(q+1)d}, \ \mathbb{E}\left[e^{i\langle \chi_{q}, \bar{X}^{\theta}_{t} \rangle}\right] = \mathbb{E}\left[e^{i\langle \chi_{q}, \bar{X}^{\theta_0}_{t} \rangle}\right],
\end{align*}
which implies that
\begin{align*}
    \text{for almost all}~ \chi \in \mathbb{R}^{d}, \ \mathbb{E}\left[e^{i\langle \chi, \bar{Y}^{\theta}_{t} \rangle}\right] = \mathbb{E}\left[e^{i\langle \chi, \bar{Y}^{\theta_0}_{t} \rangle}\right].
\end{align*}
Hence, if the assumption from \cite{panloup2020general} holds, one gets $d_{CF,p}(\nu_\theta,\nu_{\theta_0})=0$, i.e. $\nu_\theta=\nu_{\theta_0}$, which then implies that $\theta=\theta_0$. To see that \ref{asmp:identif} is strictly weaker, we refer the reader to the example of the fractional OU process detailed in Section~\ref{sec:app-OU}: in dimension $1$, the stationary measure is centred Gaussian and the variance, which depends on the drift coefficient, the diffusion coefficient and $H$, is not sufficient to identify all $3$ parameters. However, considering the stationary measure of $Y^\theta$ and its increments permit to retrieve identifiability, see Proposition~\ref{prop:identifOU}.
\end{remark}

\subsection{Approximation of the invariant measure $\mu_\theta$}\label{subsec:inv_est}
To approximate $\mu_{\theta}$, we consider the Euler scheme of the stochastic process $Y^{\theta}$, solution to \eqref{eq:fsde}. For a time-step $\gamma > 0$, the Euler scheme $Y^{\theta,\gamma}$ is then  defined by $Y^{\theta,\gamma}_0 = y_0 \in \mathbb{R}^d$ and
\begin{equation}\label{eq:euler-scheme}
\begin{split}
   Y^{\theta,\gamma}_{(k+1) \gamma } & = Y^{\theta,\gamma}_{k \gamma } + \gamma b_{\xi}( Y^{\theta,\gamma}_{k \gamma }) + \sigma (\hatB_{(k+1)\gamma} - \hatB_{k \gamma} )  \\
   Y_t^{\theta,\gamma} & = Y^{\theta,\gamma}_{t_\gamma} = Y^{\theta,\gamma}_{k \gamma} ~ \text{for}~ t \in [k \gamma, (k+1) \gamma ) ,
\end{split}
\end{equation}
where $t_\gamma= \gamma \lfloor t/\gamma \rfloor$ and $\hatB$ is a simulated fractional Brownian motion, which is \emph{a priori} different from the process $B$ in \eqref{eq:fsde}, since $B$ is unobserved. In practice, this means that we will not be able to compare pathwise the observed process and the simulated one. When necessary, to mark the dependence of $Y^{\theta,\gamma}_\cdot$ on $\hatB$, we write $Y^{\theta,\gamma}_\cdot(\hatB)$. 
We will say that $\left(\bar{Y}^{\theta,\gamma}_t\right)_{t \geq 0}$ is a discrete stationary solution to \eqref{eq:euler-scheme} if it is a solution of \eqref{eq:euler-scheme} satisfying
$$
\left(\bar{Y}^{\theta,\gamma}_{t_1+k \gamma}, \ldots, \bar{Y}^{\theta,\gamma}_{t_n+k \gamma}\right) \stackrel{\mathcal{L}}{=}\left(\bar{Y}^{\theta,\gamma}_{t_1}, \ldots, \bar{Y}^{\theta,\gamma}_{t_n}\right) \quad \forall\, 0<t_1<\cdots<t_n, \, \forall n, k \in \mathbb{N} .
$$
By \cite[Proposition 3.4]{panloup2020general}, there exists $\gamma_0 > 0$ such that for any $\gamma \in(0, \gamma_0]$ and $\theta \in \Theta$, \eqref{eq:euler-scheme} admits a unique stationary solution $\bar{Y}^{\theta,\gamma}$. As in Section \ref{subsec:not-assump}, we define the augmented Euler scheme $X^{\theta,\gamma}$ by
\begin{align}\label{eq:Xthetagamma}
    X^{\theta,\gamma}_\cdot = \left(Y^{\theta,\gamma}_\cdot, \ell^{1} (Y^{\theta,\gamma}_\cdot, Y^{\theta,\gamma}_{\cdot+h}),\dots,\ell^q( Y^{\theta,\gamma}_\cdot, \dots, Y^{\theta,\gamma}_{\cdot+qh})\right) .
\end{align}
Similarly, we write $X^{\theta,\gamma}_\cdot(\hatB)$ to insist on the dependence on $\hatB$ when necessary. We also define the stationary augmented Euler scheme $\bar{X}^{\theta,\gamma}$ and denote its distribution by $\mu_{\theta}^\gamma$

\subsection{Main results}\label{subsec:results}

Assume that the solution $Y$ is discretely observed at times $\{kh; k=1,\dots,n+q\}$ for a fixed time step $h>0$. Under Assumption \ref{asmp:drift}, we have the following lemma (the proof is postponed to Section \ref{subsec:contrast_conv}):
\begin{lemma}\label{lem:contrast-conv}
For any $d \in \mathcal{D}_2$ and any $\theta \in \Theta$, we have
\begin{align*}
d\left( \frac{1}{t} \int_{0}^{t} \delta_{X_{s}^\theta} ds ,\mu_{\theta} \right) \underset{t \rightarrow+\infty}{\longrightarrow} 0 \quad a . s. \, ,
\end{align*}
and
\begin{align*}
d\left( \frac{1}{n} \sum_{k=0}^{n-1} \delta_{X_{kh}^\theta}, \mu_{\theta}\right) \underset{n \rightarrow+\infty}{\longrightarrow} 0 \quad a . s .
\end{align*}
\end{lemma}

\begin{remark}
The integral $\int_0^t \delta_{X_s} ds$ is to be understood as the probability measure which associates to each Borel set $A$ the value $\int_0^t \delta_{X_s}(A) ds$.
\end{remark}

Hence, for some observations $X_0^{\theta_0}, \dots, X_{(n-1)h}^{\theta_0}$ and under the identifiability assumption \ref{asmp:identif}, the previous lemma justifies to use the estimator $\hat{\theta}_n$ defined in \eqref{eq:thetan-0}.

In practice, we want to implement the estimator $\hat{\theta}_{n,N,\gamma}$ defined in \eqref{eq:thetan-02}. The following result, coupled with Proposition \ref{prop:conv_contrast}$(i)$ justifies the use of this estimator.

\begin{prop}\label{prop:dconv-euler}
Let $( X^{\theta,\gamma}_{k \gamma} )_{k \geq 0} $ be the augmented Euler scheme with time-step $\gamma$. Assume that \ref{asmp:compact} and \ref{asmp:drift} hold. Then for any distance $d \in \mathcal{D}_2$, there exists $\gamma_0>0$ such that for all $\theta \in \Theta$ and $\gamma \in\left(0, \gamma_0\right]$, we have
\begin{align*}
\lim _{N \rightarrow \infty} d\left( \frac{1}{N} \sum_{k=0}^{N-1} \delta_{X^{\theta,\gamma}_{k \gamma}}, \mu_{\theta}^\gamma \right) = 0 .
\end{align*}
\end{prop}
The proof is postponed to Appendix \ref{subsec:proofof4.1}. 


The first result (Theorem \ref{thm:strong_consist}) states the strong consistency of the estimator \eqref{eq:thetan-0} under the assumptions \ref{asmp:compact}, \ref{asmp:drift}, \ref{asmp:identif} (see Section \ref{subsec:consist} for the proof).

\begin{theorem}\label{thm:strong_consist}
Assume that \ref{asmp:compact}, \ref{asmp:drift}, \ref{asmp:identif} hold. Consider a distance $d$ on $\mathcal{M}_1(\mathbb{R}^d)$ which belongs to $\mathcal{D}_2$. Then $( \hat{\theta}_n )_{n \in \mathbb{N}}$ defined in \eqref{eq:thetan-0} is a strongly consistent estimator of $\theta_0$ in the following sense:
\begin{align*}%
    \lim_{n\rightarrow \infty} \hat{\theta}_n = \theta_0 \ \text{a.s.}
\end{align*}
\end{theorem}

We also have strong consistency of the estimator \eqref{eq:thetan-02} under the same assumptions.

\begin{theorem}\label{thm:strong_consist2}
Assume that \ref{asmp:compact}, \ref{asmp:drift}, \ref{asmp:identif} hold. Consider a distance $d$ on $\mathcal{M}_1(\mathbb{R}^d)$ which belongs to $\mathcal{D}_2$. Assume that the exponent $r$ in the sub-linear growth of $b_\xi$ in \eqref{eq:drift-growth} satisfies $r \leq 1$. Then the family $\{\hat{\theta}_{n,N, \gamma}, \ (n, N,\gamma) \in \mathbb{N}^2 \times \mathbb{R}_+ \}$ is a strong consistent estimator of $\theta_0$ in the following sense:
\begin{align*}%
    \lim_{\substack{n\rightarrow \infty \\  N \rightarrow \infty \\ \gamma \rightarrow 0}} \ \hat{\theta}_{n,N, \gamma} = \theta_0 \ \text{a.s.}
\end{align*}
\end{theorem}

We will also establish a rate of convergence of the estimators when $d=d_{CF,p}$ for some $p \in \mathbb{N^*}$, under the strong identifiability assumption:
\begin{itemize}
\item[\namedlabel{asmp:strong_identif}{{$\mathbf{I_s}$}}.] There exists a constant $c_1> 0$ and $\alpha \ge 2$, such that for every $\theta$ in $\Theta$,
\begin{align*}%
    d_{CF,p} (\mu_{\theta},\mu_{\theta_0})^{\alpha} \ge c_1 | \theta - \theta_0 |^2  .
\end{align*}
\end{itemize}
Under this assumption, we obtain a rate of convergence for $\hat{\theta}_n$ and $\hat{\theta}_{n,N,\gamma}$, which will be proved in Section \ref{sec:rateofconv} and Section \ref{sec:bounds-rate}.
\begin{theorem}\label{thm:rate}
Assume that \ref{asmp:compact} and \ref{asmp:drift} hold, and that \ref{asmp:strong_identif} holds for $p>\frac{\alpha+d(q+1)}{2}$. Then $ \lim_{n\rightarrow \infty} \hat{\theta}_n = \theta_0 \ \text{a.s.}$
and there exists a constant $C>0$ such that for any $n \in \mathbb{N}^*$,
\begin{align*}%
    \mathbb{E}| \hat{\theta}_n - \theta_0 |^2 \leq C n^{-\alpha(1-(\max(\mathcal{H}) \vee \frac{1}{2}) )}  .
\end{align*}
\end{theorem}

\begin{theorem}\label{thm:rate-2}
Assume that \ref{asmp:compact} and \ref{asmp:drift} hold and that \ref{asmp:strong_identif} holds for $p > \frac{\alpha+d(q+1)}{2}$. Assume that the exponent $r$ in the sub-linear growth of $b_\xi$ in \eqref{eq:drift-growth} satisfies $r \leq 1$. Then
$\lim_{n\rightarrow \infty,  N \rightarrow \infty, \gamma \rightarrow 0}  \hat{\theta}_{n,N, \gamma} = \theta_0$ \text{a.s.}
 Moreover, for any $\varepsilon \in (0,\min(\mathcal{H}))$ and $\varpi \in (0,1)$, there exists positive constants $C,\gamma_0$ such that for any $\gamma \in (0, \gamma_0]$ and $n,N \in \mathbb{N}$ satisfying $N \gamma \ge 1$, we have 
\begin{align*}
\mathbb{E}\left|\hat{\theta}_{n,N, \gamma}-\theta_{0}\right|^{2}&  \leq C\Big(n^{\alpha(-1+ \max(\mathcal{H}) \vee \frac{1}{2} )}+N^{\alpha(-1+ \max(\mathcal{H}) \vee \frac{1}{2} )}+ \gamma^{\alpha(\min(\mathcal{H})-\varepsilon)} \\ 
& \quad +(N \gamma)^{-\frac{\varpi \alpha^2}{2(\varpi \alpha+2d)} (2-2 \max(\mathcal{H}) \vee 1) }\Big) .
\end{align*}
\end{theorem}

\begin{remark}
We discuss here whether the above rate of convergence can be optimal. \\
Assume first that $H \leq \frac{1}{2}$, $d=1$ and $\alpha=2$ (that is the case for the fractional Ornstein-Uhlenbeck (OU) process for example, see \cite[Lemma 6.2]{panloup2020general}). Then by Theorem \ref{thm:rate-2}, using that $N \gamma \leq N$, we have for $\varepsilon \in (0,\frac{1}{2})$
\begin{align*}
\left( \mathbb{E}\left|\hat{\theta}_{n,N, \gamma}-\theta_{0}\right|^{2} \right)^{\frac{1}{2}} &  \leq  C \left( n^{-\frac{1}{2}} + \gamma^{\min(\mathcal{H})-\varepsilon} + (N\gamma)^{-\frac{1}{4}+\varepsilon} \right) .
\end{align*}
The term $n^{-\frac{1}{2}}$ corresponds to the convergence with respect to the sample size. It matches the CLT rate and generalises \cite[Theorem 4.5]{haress2020estimation} where, for the fractional OU process, the authors construct estimators of $\xi,\sigma$ and $H$ based on the invariant measure and prove a CLT with respect to the sample size.
 Moreover, taking $N = \gamma^{-(4\min(\mathcal{H})+1)}$, there is 
\begin{align*}
\left( \mathbb{E}\left|\hat{\theta}_{n,N, \gamma}-\theta_{0}\right|^{2} \right)^{\frac{1}{2}} &  \leq  C \left( n^{-\frac{1}{2}} + \gamma^{\min(\mathcal{H})-\varepsilon} \right) .
\end{align*}
The term $\gamma^{\min(\mathcal{H})- \varepsilon}$ corresponds to the convergence of the Euler scheme defined in \eqref{eq:euler-scheme} and is known to be optimal optimal for strong errors. So the rate cannot be improved in this situation.\\
 Beyond the fractional OU process, we generalise the rate obtained in \cite[Theorem 2.13]{panloup2020general} to an estimation of all the parameters. 
Finally, the CLT obtained in \cite{haress2020estimation} holds for $H \in (0, \frac{3}{4})$ while here the rate we obtained is slower when $H > \frac{1}{2}$: for $\alpha=2$, the first $L^2$ error term in the bound of Theorem \ref{thm:rate-2} reads $n^{-1+\max(\mathcal{H})} \gg n^{-\frac{1}{2}}$.
\end{remark}

\begin{remark}
The proofs of Theorem \ref{thm:rate} and Theorem \ref{thm:rate-2} rely on an upper bound of $\mathbb{E}[d(\mu_{\widehat{\theta}_n}, \mu_{\theta_0})^\alpha]$. This involves bounding the quantities in the left-hand side of \eqref{eq:boundsXtoY} with $p=\alpha$. In \eqref{eq:boundsXtoY} the constant $C_{\alpha,q}$ is of order $q^{\alpha} L^\alpha$, where $L$ is the biggest Lipschitz constant of the linear transformations. Hence the constant $C$ in Theorem \ref{thm:rate} depends on $(\ell^{i})_{i=1,\dots,q}$ as $C \sim q^{\alpha+1} L^\alpha$. This can be useful in practice when choosing $(\ell^{i})_{i=1,\dots,q}$.
\end{remark}

\subsection{Application to fractional Ornstein-Uhlenbeck-type processes}\label{sec:app-OU}

We first discuss  the identifiability assumption for the fractional Ornstein-Uhlenbeck (OU) process, then for a family of small perturbations of the fractional OU process. 

\paragraph{Identifiability assumption.} Consider the family of one-dimensional fractional OU processes given by
\begin{align}\label{eq:fO-U}
\begin{array}{ll}
    dU^{\theta} = -\xi U^{\theta} dt + \sigma dB ,\quad U^{\theta}_0 = 0 .
\end{array}    
\end{align}
It is known from the proof of \cite[Proposition 3.12]{Hairer}
  that the stationary measure of $U^\theta$ follows the Gaussian distribution
\begin{align}\label{eq:fOUinvariant}
    \mathcal{N}(0, \sigma^2 H \Gamma(2H) \xi^{-2H}).
\end{align}
In particular, this distribution alone does not permit to identify simultaneously $\xi$, $\sigma$ and $H$. Hence the need to consider increments.\\
Here, $\xi$ and $\sigma$ are in compact subsets of $(0, \infty)$ and the linear transformation $\ell^1$ takes the form of an increment
\begin{align*}
\ell^1(U^{\theta}_\cdot,U^{\theta}_{.+h}) = U^{\theta}_{.+h} - U^{\theta}_\cdot  .
\end{align*}
We suppose here that $\theta$ is of dimension 2, i.e. only two of the three parameters $(\xi,  \sigma,  H)$ are unknown. In the following result, we establish that \ref{asmp:identif} is verified.

\begin{prop}\label{prop:identifOU}
Consider the fractional Ornstein-Uhlenbeck model defined by equation \eqref{eq:fO-U} and assume that one of the parameters $\xi$, $\sigma$ or $H$ is known. Let $p > 1$ and let $\mu_{\theta}$ denote the stationary measure of $(U^\theta_{\cdot}, U^\theta_{\cdot+h}-U^\theta_{\cdot})$. Then there exists $h_0>0$ such that for all $h \in (0,h_0)$, we have
\begin{align*}
\forall \theta_1, \theta_2 \in \Theta, \quad  d_{CF,p} (\mu_{\theta_1}, \mu_{\theta_2}) =0 \ \ \ \text{iff} \ \ \ \theta_1 = \theta_2 .
\end{align*}
\end{prop}
The proof is given in Section \ref{subsec:simplif}.

\paragraph{Strong identifiability assumption.} 
We show that Assumption \ref{asmp:strong_identif} holds for some specific examples of \eqref{eq:fsde} and for the distance $d=d_{CF,p}$. Specifically, we consider a family $U^{\lambda, \theta}$ of real-valued processes defined by
\begin{align}\label{eq:O-U-perturbation}
d U_t^{\lambda, \theta} = \left( - \xi U_t^{\lambda, \theta} + \lambda b_{\xi}(U_t^{\lambda, \theta}) \right) dt + \sigma d B_t .
\end{align}
Under the assumption that the coefficient $b_\xi$ is bounded altogether with its derivatives with respect to $\xi$ and $y$, one can check that the drift term $b(\cdot) = -\xi \cdot  + \lambda b_\xi(\cdot)$ satisfies \ref{asmp:drift} for $\lambda$ small enough. Therefore, the equation has a unique invariant measure, which is denoted by $\mu_{\theta}^\lambda$. The process $U^{\lambda, \theta}$ can be seen as a small perturbation of the fractional Ornstein-Uhlenbeck process, since $U^{0,\theta} = U^\theta$, where $U^\theta$ is the fractional OU process defined in \eqref{eq:fO-U}. For the fractional OU process, we simply write $\mu_\theta$ for the invariant measure. 
We make the following assumption on the parameters:
\begin{itemize}
\item[\namedlabel{asmp:drift2}{{$\mathbf{\widetilde{A_0}}$}}.] Assume that $\xi$, $\sigma$ and $H$ are one-dimensional parameters and that
\begin{align*}
\begin{split}
\xi \in [m_\Xi, M_\Xi ], & ~ \text{with} ~ 0 < m_\Xi < M_\Xi < \infty \\
\sigma \in  [m_\Sigma, M_\Sigma ], & ~ \text{with} ~ 0 < m_\Sigma < M_\Sigma < \infty \\
H \in [m_\mathcal{H}, M_\mathcal{H} ], & ~ \text{with} ~  0 < m_\mathcal{H} < M_\mathcal{H} < 1 .
\end{split}
\end{align*}
\end{itemize}
We shall prove that $U^{\lambda, \theta}$ satisfies assumption \ref{asmp:strong_identif} when only one parameter is unknown (so either $\theta=\xi$, $\theta=\sigma$ or $\theta=H$). When referring to $\theta$, we will write our assumption above as $\theta \in [m_\Theta, M_\Theta]$.

\smallskip

The first lemma below states that \ref{asmp:strong_identif} is satisfied for $U^{\theta}:=U^{0, \theta}$.
\begin{lemma}\label{lem:strongidentif-O-U}
Let $\theta$ represent either $\xi$, $\sigma$ or $H$. Assume that \ref{asmp:drift2} holds and if $\theta = H$, assume further that
\begin{align}\label{eq:xiawayfrom0}
\xi > \sup_{H \in [m_{\mathcal{H}}, M_{\mathcal{H}}]} \exp\left({\frac{\Gamma(2H)+2H\Gamma'(2H)}{2H \Gamma(2H)}}\right) ~\text{or}~ \xi < \inf_{H \in [m_{\mathcal{H}}, M_{\mathcal{H}}]} \exp\left({\frac{\Gamma(2H)+2H\Gamma'(2H)}{2H \Gamma(2H)}}\right) .
\end{align}
Let $p \ge 1$, then for all $\theta_1, \theta_2 \in [m_\Theta,M_\Theta]$,
\begin{align*}%
d_{CF,p} (\mu_{\theta_1},\mu_{\theta_2}) \ge c | \theta_1- \theta_2 | ,
\end{align*}
where $c$ is a constant that  does not depend on $\theta_1$ or $\theta_2$. 
\end{lemma}

\smallskip

The previous lemma extends to the solution of Equation \eqref{eq:O-U-perturbation}. 
\begin{prop}\label{lem:O-U-perturbation}
Let $U^{\lambda, \theta}$ be the process defined by \eqref{eq:O-U-perturbation} where $\theta$ is either $\xi$, $\sigma$ or $H$, and let $p>3/2$. Assume that \ref{asmp:drift2} holds and that $b_\xi$, $\partial_y b_\xi$, $\partial_\xi b_\xi$ are bounded. Moreover, if $\theta=\xi$, assume that $|\partial^2_{y, \xi} b_\xi| \leq 1$  and if $\theta = H$, assume that
\eqref{eq:xiawayfrom0} holds. \\
Then there exists $\lambda_0 = \lambda_0(m_\Theta,M_\Theta,p)>0$ and  $c_{m_\Theta,M_\Theta,p}>0$ such that for any $\lambda \in (0, \lambda_0)$ and any $\theta_1, \theta_2 \in [m_\Theta, M_\Theta]$,
\begin{align*}%
d_{CF,p}(\mu_{\theta_1}^\lambda, \mu_{\theta_2}^\lambda) \ge c_{m_\Theta,M_\Theta,p} | \theta_1-\theta_2 | .
\end{align*}
\end{prop}
The proofs of Lemma~\ref{lem:strongidentif-O-U} and Proposition~\ref{lem:O-U-perturbation} are given in Section \ref{sec:OU-perturbation}.

\section{Strong consistency and rate of convergence of the estimator $\hat{\theta}_n$}\label{sec:consist}

To prove the almost sure convergence, we will use \cite[Proposition 4.3]{panloup2020general} that we recall in Proposition \ref{thm:convergence} below for the reader's convenience. It concerns the limiting property of a collection of real-valued processes $\{L_{v}(\theta)\}_{v}$ indexed by a generic $v$ which lies in a topological space and converges to a generic $v_0$.  In this Section, we always have $v \equiv n \in \mathbb{N}$, and so $\lim_{v \rightarrow v_0}$ is to be understood as $\lim_{n \rightarrow \infty}$. In Section \ref{sec:practical}, we will take $v \equiv (\gamma,n,N)$ with $\gamma \rightarrow 0$ and $n,N \rightarrow \infty$, and therefore $\lim_{v \rightarrow v_0}$ will be understood as $\lim_{n \rightarrow
 \infty, N \rightarrow \infty, \gamma \rightarrow 0}$.
\begin{prop}[{\cite[Proposition 4.3]{panloup2020general}}]
\label{thm:convergence}
Let $\Theta$ be a compact set and $\{\theta \in \Theta \mapsto L_v(\theta) \}_{v}$ a family of non-negative stochastic processes. Assume that 
\begin{itemize}
    \item[(i)] Almost surely, $\lim_{v \rightarrow v_0} L_v(\theta) = L(\theta)$ uniformly in $\theta$.
    \item[(ii)] $\theta \mapsto L(\theta)$ is deterministic and continuous in $\theta$.
    \item[(iii)] For any $v$, the set $\argmin\{ L_v(\theta), \theta \in \Theta \} $ is non-empty.
\end{itemize}
Let $\theta_v \in \argmin_{\theta \in \Theta} L_v(\theta)$. If $A$ is a limit point of $\{\theta_v\}_{v}$, then $A \in \argmin_{\theta \in \Theta} L(\theta) $.    
\end{prop}

In this Section, we always have $L_v(\theta) = d(\frac{1}{n} \sum_{k=0}^{n-1} \delta_{X_{kh}}, \mu_{\theta})$, with $v \equiv n$ and $v_0 \equiv \infty$.

\subsection{Continuity of $\theta \mapsto d(\mu_\theta,\mu_{\theta_0})$}%

First, we prove two lemmas that state the $L^p(\Omega)$-continuity with respect to $\theta$ of the solution to \eqref{eq:fsde}, and the exponential convergence of the law of $X^\theta$ (defined in \eqref{eq:defX}) towards its stationary distribution $\mu_{\theta}$. Then we deduce the continuity of the mapping $\theta \mapsto d(\mu_\theta,\mu_{\theta_0})$ in Proposition~\ref{prop:continuity-distance-measures}.

\begin{lemma}\label{prop:continuity}
Assume \ref{asmp:compact} and \ref{asmp:drift} are satisfied. Let $T>0$ and $p>0$. Let $W$ be an $\R^d$-Brownian motion and for any $H\in (0,1)$, denote by $B^H$ the fBm with underlying noise $W$ (i.e. as in \eqref{eq:MVN}).
There exists a constant $C_{T,p}>0$ such that for any $\theta_1, \theta_2 \in \Theta$,
\begin{align*}%
\| Y_T^{\theta_1} - Y_T^{\theta_2} \|_{L^p}  & \leq  C_{T,p} | \theta_1 - \theta_2  | ,
\end{align*}
where $Y^{\theta_{1}}$ (resp. $Y^{\theta_{2}}$) is the solution to \eqref{eq:fsde} with parameter $\theta_{1}$ (resp. $\theta_{2}$) and driving fBm $B^{H_{1}}$ (resp. $B^{H_{2}}$),  and both $Y^{\theta_{1}}$ and $Y^{\theta_{2}}$ start from the same initial condition.
\end{lemma}
\begin{proof}
Without any loss of generality, we assume $p \ge 2$. Up to introducing pivot terms, we can consider three different cases:
\begin{align*}
\text{1)}~ \theta_1 = (\xi, \sigma, H_1) & ~\text{and}~ \theta_2 = (\xi, \sigma, H_2) \\ 
~\text{2)}~ \theta_1 = (\xi, \sigma_1, H) & ~ \text{and}~ \theta_2 = (\xi, \sigma_2, H) \\
  ~\text{3)}~ \theta_1 = (\xi_1, \sigma, H) & ~\text{and}~ \theta_2 = (\xi_2, \sigma, H) . 
\end{align*}
In the first case,  where only $H$ changes, we get from the definition of $Y_t^{\theta_1}$ and  $Y_t^{\theta_2}$ that for any $t \in [0,T]$,
\begin{align*}
Y_t^{\theta_1} - Y_t^{\theta_2} = \int_0^t (b_{\xi} (Y_t^{\theta_1} ) -b_{\xi} (Y_t^{\theta_2})) ds + \sigma (B_t^{H_1}- B_t^{H_2}) .
\end{align*}
Since $b$ is $K$-Lipschitz, we get
\begin{align*}
| Y_t^{\theta_1} - Y_t^{\theta_2} |^2 & \leq 2  \left( \int_0^t K | Y_t^{\theta_1} - Y_t^{\theta_2} |  d s \right)^2 + 2 | \sigma |^2 | B_t^{H_1}- B_t^{H_2}|^{2} .
\end{align*}
By Jensen's inequality, we have
\begin{align*}
| Y_t^{\theta_1} - Y_t^{\theta_2} |^2 & \leq 2 K^2 t \int_0^t  | Y_t^{\theta_1} - Y_t^{\theta_2} |^2  d s+ 2 | \sigma |^2 | B_t^{H_1}- B_t^{H_2}|^2 
\end{align*}
By Gr\"onwall's lemma, we deduce that
\begin{align*}
| Y_t^{\theta_1} - Y_t^{\theta_2} |^2 \leq 2 K^2 T \int_0^t  | \sigma |^2 | B_s^{H_1}- B_s^{H_2} |^2 e^{2 K^2 T (t-s)} d s+ 2 | \sigma |^2 | B_t^{H_1}- B_t^{H_2}|^2 .
\end{align*}
By Jensen's inequality, there exists a constant $C_p$ such that
\begin{align*}
| Y_t^{\theta_1} - Y_t^{\theta_2} |^p \leq C_p \left( 2^{p/2} K^{p} T^{p-1} \int_0^t  | \sigma |^p | B_s^{H_1}- B_s^{H_2} |^p e^{K^2 T p (t-s)} d s+ | \sigma |^p | B_t^{H_1}- B_t^{H_2}|^p \right) .
\end{align*}Since $B_t^{H_1}-B_t^{H_2}$ is a Gaussian random variable, $ \EE | B_t^{H_1}- B_t^{H_2}|^p $ is proportional to $( \EE | B_t^{H_1}- B_t^{H_2}|^2 )^{p/2}$. Using \cite[Proposition 2.1]{HRarxiv}, the fractional Brownian motion verifies
\begin{align*}
\EE | B_t^{H_1} - B_{t}^{H_2} |^p & \leq  C\, \big(t^{pH_1}\vee t^{pH_2}\big) \, (\log^2(t)+1)^{p/2} \,  |H_1-H_2|^p  .
\end{align*}
Therefore,
\begin{align*}
\EE| Y_t^{\theta_1} - Y_t^{\theta_2} |^p \leq C_p | \sigma |^p \left( 2^{p/2} K^{p} T^{p}  e^{K^2 T^2 p} + 1 \right) \left( T^{pH_1} \vee T^{pH_2} \right)  (\log^2(T)+1)^{p/2}  | H_1-H_2 |^p .
\end{align*}
Since $\sigma \in \Sigma$, we conclude that
\begin{align*}
\| Y_t^{\theta_1} - Y_t^{\theta_2} \|_{L^p} 
& \leq C_{p,\sigma,K} ( T  e^{K^2 T^2} + 1 ) ( 1+ T^{\max(\mathcal{H})} )  (\log^2(T)+1)^{1/2} | H_1-H_2 | . 
\end{align*}
In the second case, since $b$ is $K$-Lipschitz, using Jensen's inequality, we have
\begin{align*}
| Y_t^{\theta_1} - Y_t^{\theta_2} |^2 & =\left( \int_0^t [b_{\xi} (Y_t^{\theta_1} ) -b_{\xi} (Y_t^{\theta_2} )]ds + (\sigma_1-\sigma_2) B_t \right)^2 \\
& \leq 2 K^2  T \int_0^t  | Y_s^{\theta_1} - Y_t^{\theta_s}  |^2 ds + 2 | \sigma_1 - \sigma_2 |^2 | B_t |^2 .
\end{align*}
By Gr\"onwall's lemma, we get
\begin{align*}
| Y_t^{\theta_1} - Y_t^{\theta_2} |^2 \leq | \sigma_1 - \sigma_2 |^2 \left(  | B_t |^2 + 2 K^2 T \int_0^t | B_s |^2 e^{2 K^2 T (t-s)} ds \right)  .
\end{align*}
Therefore, by Jensen's inequality, there exists a constant $C_p$ such that
\begin{align*}
| Y_t^{\theta_1} - Y_t^{\theta_2} |^p \leq C_p | \sigma_1 - \sigma_2 |^p \left(  | B_t |^p + 2^{p/2} K^{p} T^{p-1} \int_0^t | B_s |^p e^{ K^2 Tp (t-s)} ds \right)  .
\end{align*}
It follows that
\begin{align*}
\| Y_t^{\theta_1} - Y_t^{\theta_2} \|_{L^p} & \leq C_p | \sigma_1 - \sigma_2 | ( T^{H} + T^{1+H} e^{K^2 T^2} )  \\
& \leq C_p | \sigma_1 - \sigma_2 | (1+T^{\max(\mathcal{H})}) ( T e^{K^2 T^2}+1 ) .
\end{align*}
Finally, in the third case, we have by \cite[Proposition 3.5]{panloup2020general} that $\| Y_t^{\theta_1} - Y_t^{\theta_2} \|_{L^p} \leq C_{T,p} | \xi_1-\xi_2 |$, 
where it appears from the proof of \cite[Proposition 3.5]{panloup2020general} that $C_{T,p}$ does not depend on $H$ or $\sigma$.
\end{proof}

\begin{lemma}\label{prop:general-results}
Assume \ref{asmp:compact} and \ref{asmp:drift} hold. Let $d$ be a distance in $\mathcal{D}_p$. Then there exists a constant $C>0$ such that for all $\theta \in \Theta$ and for all $t\ge 0$, we have
\begin{align}\label{eq:general-result}
d(\mathcal{L}(X_t^\theta),\mu_\theta)  \leq C e^{- \frac{1}{C} t}   .
\end{align}
\end{lemma}
\begin{proof}
Since $d \in  \mathcal{D}_p$, it comes:
\begin{align*}
d(\mathcal{L}(X_t^\theta),\mu_\theta) & \leq \mathbb{E}(| X_t^\theta -  \bar{X}_t^\theta|^p)^\frac{1}{p} 
 \leq C \Biggl(  \sum_{i=0}^{q}  \mathbb{E}(|  Y_{t+ih}^\theta -  \bar{Y}_{t+ih}^\theta|^p)^\frac{1}{p} \Biggr) .
\end{align*}
Using \ref{asmp:drift}, 
\begin{align*}
\frac{d}{dt} | Y_t^\theta - \bar{Y}_t^\theta |^2 & = 2 \langle Y_t^\theta - \bar{Y}_t^\theta, b_\xi (Y_t^\theta)-b_\xi(\bar{Y}_t^\theta) \rangle 
 \leq -2 \beta | Y_t^\theta - \bar{Y}_t^\theta |^2 .
\end{align*}
It follows that $|Y_t^\theta - \bar{Y}_t^\theta |^2 \leq | Y_0^\theta - \bar{Y}_0^\theta |^2 e^{-2\beta t}$.
Hence for $p \ge 2$,
\begin{align}\label{eq:expBoundYYbar}
\|  Y_t^\theta -  \bar{Y}_t^\theta \|_{L^p} & \leq \| Y_0^\theta - \bar{Y}_0^\theta \|_{L^p}  e^{-\beta t}
 \leq ( \| Y_0^\theta \|_{L^p} + \| \bar{Y}_0^\theta \|_{L^p} )\,  e^{-\beta t} .
\end{align}
Moreover, by stationarity and Proposition \ref{prop:uniform-Lq-bounds}$(i)$, we have 
$$ \| \bar{Y}_0^\theta \|_{L^p} = \lim_{t \rightarrow \infty} \| Y_t^\theta \|_{L^p}  \leq \sup_{t \ge 1} \sup_{\theta \in \Theta} \| Y_t^\theta \|_{L^p} < \infty .$$
This concludes the proof.
\end{proof}

We can now state the main continuity result of this section.
\begin{proposition}\label{prop:continuity-distance-measures}
Assume \ref{asmp:compact} and \ref{asmp:drift} hold and let $d$ be a distance in $\mathcal{D}_p$. Then the mapping $\theta \mapsto d(\mu_\theta,\mu_{\theta_0})$ is continuous on $\Theta$.
\end{proposition}

\begin{proof}
Let now $\theta_1,\theta_2\in \Theta$. Then for $t \ge 0$,
\begin{align*}
  d(\mu_{\theta_1},\mu_{\theta_2}) \leq C \mathcal{W}_p(\mu_{\theta_1},\mu_{\theta_2}) &\leq C \mathcal{W}_p(\mu_{\theta_1},\mathcal{L}(X_t^{\theta_1})) + C \mathcal{W}_p(\mu_{\theta_2},\mathcal{L}(X_t^{\theta_2})) + C \| X_t^{\theta_1}-X_t^{\theta_2} \|_{L^p} \\
    &\leq 2 C \sup_{\theta \in \Theta}  \mathcal{W}_p(\mathcal{L}(X_t^{\theta}),\mu_{\theta}) +  C \| X_t^{\theta_1} - X_t^{\theta_2} \|_{L^p} .
\end{align*}
Let $\epsilon>0$. By Lemma \ref{prop:general-results} there exists $t_0$ such that
\begin{align*}
2 C \sup_{\theta \in \Theta} \mathcal{W}(\mathcal{L}(X_{t_0}^{\theta}),\mu_{\theta}) \leq  \frac{\epsilon}{2} .
\end{align*}
Now in view of \eqref{eq:boundsXtoY} and Lemma \ref{prop:continuity}, there exists a constant $C_{t_0,p}$ such that $\| X_{t_0}^{\theta_1} - X_{t_0}^{\theta_2} \|_{L^p}  \leq C_{t_{0},p} | \theta_1 - \theta_2|$. Let $\delta>0$ be such that $C_{t_0,p} \delta \leq \epsilon/2$. Then for $| \theta_1 - \theta_2 | \leq \delta$, we have
\begin{align*}
   d(\mu_{\theta_1},\mu_{\theta_2}) \leq \epsilon ,
\end{align*}
and this proves the continuity of $\theta \mapsto d(\mu_\theta, \mu_{\theta_0})$. 
\end{proof}

\subsection{Convergence of the contrast: proof of Lemma \ref{lem:contrast-conv}} \label{subsec:contrast_conv}
Let $\theta = (\xi,\sigma,H) \in \Theta$. Recall that the Prokhorov distance is defined for any $\mu, \nu \in \mathcal{M}_1(\mathbb{R}^d)$ as
\begin{align*}
d_{\text{P}}(\mu,\nu) = \inf \{ \varepsilon > 0,~ \mu(A) \leq \nu(A^\varepsilon)+\varepsilon ~\text{for any Borel set}~A \},
\end{align*}
where $A^\varepsilon$ is the $\varepsilon$-neighbourhood of $A$. Convergence in law is equivalent to convergence with respect to the Prokhorov distance.

We will first prove that almost surely, the random measure $\frac{1}{t} \int_0^t \delta_{X_s^\theta} ds$ converges in law to $\mu_{\theta}$. This implies that $\frac{1}{t} \int_0^t \delta_{X_s^\theta} ds$ converges to $\mu_{\theta}$ in the Prokhorov distance. To extend this result to distances $d$ in $\mathcal{D}_2$ (i.e dominated by the $2$-Wasserstein distance), we use the fact that the $2$-Wasserstein distance is dominated by the Prokorov distance $d_P$ as follows (see \cite[Theorem 2]{gibbs2002choosing}):
\begin{equation}\label{eq:boundProkhorov}
\begin{split}
& d\left(\frac{1}{t} \int_0^t \delta_{X_s^\theta} ds, \mu_{\theta}\right)  \\ &  \leq C_p \sup_{t \ge 0} \left( \max\left(\frac{1}{t} \int_0^t |X_s^\theta|^2 ds , \, \mathbb{E} |\bar{X}_t^\theta|^2\right) + 1 \right) \, d_P\left(\frac{1}{t}\int_0^t \delta_{X_s^\theta} ds, \mu_{\theta}\right) .
\end{split}
\end{equation}
By definition of the process $X^\theta$, we have that
\begin{align}\label{eq:comp-X-Y}
\max \left( \frac{1}{t} \int_0^t |X_s^\theta|^2 ds ,  \,\mathbb{E} |\bar{X}_t^\theta|^2 \right) \leq C_{q} \sum_{i=0}^{q}  \max \left( \frac{1}{t} \int_0^t |Y_{s+ih}^\theta|^2 ds , \, \mathbb{E} |\bar{Y}_{t+ih}^\theta|^2 \right) .
\end{align}
Therefore, we conclude thanks to Proposition \ref{prop:uniform-Lq-bounds} that in the present case, the convergence in law (i.e. in Prokhorov distance) implies the convergence for the $2$-Wasserstein distance. Let us now prove the convergence in law. The proof of the convergence in law follows the same steps as \cite[Proposition 3.3]{panloup2020general} and relies on a tightness argument. While we can show that the family $\{ \frac{1}{t} \int_0^t \delta_{X_s^\theta} ds \}_{t \ge 0}$ is tight, it is not easy to identify the limit points. That is why we consider a family of probability measures on the set of continuous functions for which the identification of the limit is easier, namely $\{ \pi_t^{\theta} = \frac{1}{t} \int_0^t \delta_{X_{s+.}^{\theta}} d s \}_{t \ge 0}$. The criterion from \cite[Corollary p.83]{Billingsley} ensures that $\{\pi_t^{\theta} ; t \geq 0 \}$ is a.s. tight if for every positive $T, \eta$ and $\varepsilon$, there exists $\delta>0$ such that for all $t_0 \in[0, T]$,
\begin{align*}
\limsup _{t \rightarrow+\infty} \frac{1}{t} \int_0^t \frac{1}{\delta} \mathds{1}_{\left\{\sup _{u \in\left[t_0, t_0+\delta\right]}\left|X^\theta_{s+u}-X^\theta_{s+t_0} \right| \geq \varepsilon\right\}} d s \leq \eta \quad ~\text {a.s.}
\end{align*}
Moreover, the above inequality holds true if there exist some positive $r$ and $\rho$ such that
\begin{align}\label{eq:tightness}
 \forall T>0, \ \exists \delta >0,~ & r>0,~ \rho>0 ~ \text{ s.t. } ~ \forall t_0 \in [0,T],  \nonumber \\
 & \limsup_{t\rightarrow  \infty} \frac{1}{t} \int_0^t \sup_{u \in [t_0, t_0 + \delta]} | X_{s+u}^{\theta} -X_{s+t_0}^{\theta} |^r d s \leq C_{r,T} \delta^{1+\rho}  \quad ~\text{a.s} .
\end{align}
For $T,r, \delta>0$, by definition of $X^\theta$ and \eqref{eq:boundsXtoY}, there is
\begin{align*}
& \limsup_{t\rightarrow  \infty} \frac{1}{t} \int_0^t \sup_{u \in [t_0, t_0 + \delta]} | X_{s+u}^{\theta} -X_{s+t_0}^{\theta} |^r d s \\ 
& \leq 
\sum_{i=0}^q \limsup_{t\rightarrow  \infty} \frac{1}{t} \int_0^t \sup_{u \in [t_0, t_0 + \delta]} | Y_{s+u+ih}^{\theta} -Y_{s+t_0+ih}^{\theta} |^r d s \\
& \leq \sum_{i=0}^q \limsup_{t\rightarrow  \infty} \frac{t+ih}{t} \frac{1}{t+ih} \int_0^{t+ih} \sup_{u \in [t_0, t_0 + \delta]} | Y_{s+u}^{\theta} -Y_{s+t_0}^{\theta} |^r d s \\
& \leq C_q \limsup_{t\rightarrow  \infty}  \max_{i \in \llbracket 0, q \rrbracket}\frac{1}{t+ih} \int_0^{t+ih} \sup_{u \in [t_0, t_0 + \delta]} | Y_{s+u}^{\theta} -Y_{s+t_0}^{\theta} |^r d s \\
& \leq C_q \limsup_{t\rightarrow  \infty}  \frac{1}{t} \int_0^{t} \sup_{u \in [t_0, t_0 + \delta]} | Y_{s+u}^{\theta} -Y_{s+t_0}^{\theta} |^r d s .
\end{align*}
By \cite[Eq A.19]{panloup2020general}, we can further bound the right-hand side above by $C \delta^{r-1}+C_r \delta^{Hr}$. Choosing $\delta < 1$ and $r > \max(2, \frac{1}{\min(\mathcal{H})})$, we get \eqref{eq:tightness}.

Hence, let $(t_n)_{n \ge 1}$ be an increasing sequence going to $+\infty$ such that $\{\frac{1}{t_n}  \int_0^{t_n}  \delta_{X_{s+.}^{\theta}} ds \}_{n \ge 1}$ converges (pathwise) to a probability measure $\gamma$. 
As in Appendix A.2 of \cite[Proposition 3.3]{panloup2020general}, we get that $\gamma$ is stationary.
 Let us now prove that $\gamma$ is the law of $\bar{X}^{\theta}$. A process ${x_t=(y_t,z^1_t,\dots,z^{q}_t)}$ has the law of $X^{\theta}$ if 
\begin{align*}
  &  y_\cdot  - y_0 - \int_0^. b_\xi(y_u) du \text{ has the law of $\sigma B$ where B has Hurst parameter $H$};\\
 &   z^i_\cdot - \ell^i \left( \int_0^\cdot b_\xi(y_u) du , \dots, \int_0^{.+ih} b_\xi(y_u) du \right) \text{ has the law of}~ \sigma \ell^i(B_\cdot,\dots,B_{\cdot+ih}) \\ & \quad \text{for all $i \in \llbracket 1,q \rrbracket$}.
\end{align*}
Let us define 
\begin{align*}
G(x_\cdot)  = \begin{pmatrix}
     y_\cdot  - y_0 - \int_0^. b_\xi(y_u) du  \\
     z^1_\cdot - \ell^1 \left( \int_0^\cdot  b_\xi(y_u) du, \int_0^{.+h} b_\xi(y_u) du \right) \\
     \vdots \\
     z^{q}_\cdot - \ell^q\left( \int_0^\cdot b_\xi(y_u) du , \dots, \int_0^{.+qh} b_\xi(y_u) du \right)
\end{pmatrix} 
\end{align*}
and
\begin{align*}
\mathbf{B}_\cdot & = \left( \sigma  B_{\cdot}, \dots, \sigma \ell^q(B_{\cdot},\dots,B_{qh+\cdot}) \right). 
\end{align*}
Hence we have to prove that $\gamma \circ G^{-1} ~\text{is the law of}~ \mathbf{B}_\cdot$. This follows the same lines as the end of Appendix A.2 of \cite{panloup2020general} and we omit the details. 

The same analysis presented in this Section still holds if we replace $\frac{1}{t} \int_0^t |X_s^\theta|^p ds $ by $\frac{1}{n} \sum_{k=0}^{n-1} | X_{kh}^\theta|^p $.  This is mostly due to the fact that in Proposition \ref{prop:uniform-Lq-bounds}, we also proved that the moments $\frac{1}{n} \sum_{k=0}^{n-1} | X_{kh}^\theta|^p $ are finite uniformly in $n$, and therefore the right-hand side in \eqref{eq:comp-X-Y} is finite even when the integral is replaced by a discrete sum.

\subsection{Proof of Theorem \ref{thm:strong_consist}}\label{subsec:consist}

Let $d$ be a distance that belongs to $\mathcal{D}_p$. We want to apply Proposition \ref{thm:convergence} to $v\equiv n$ and  
\begin{equation*}
L_n(\theta)=d\left( \frac{1}{n}\sum_{k=0}^{n-1} \delta_{X^{\theta_0}_{kh}}, \mu_{\theta} \right).
\end{equation*}
In view of Lemma \ref{lem:contrast-conv}, we know that for each $\theta$, $L_n(\theta)$ converges a.s. to $L(\theta)=d(\mu_{\theta_0}, \mu_{\theta})$. Besides, the continuity of $L(\theta)$ comes from Proposition~\ref{prop:continuity-distance-measures}. If we prove the uniform convergence, then we can finally apply Proposition \ref{thm:convergence} to get that the limit points of $\hat{\theta}_n$ are included in the set $\argmin\{L(\theta), \theta \in \Theta \}$, which under assumption \ref{asmp:identif} is reduced to $\{\theta_0\}$.

Now to prove the uniform convergence, it is sufficient to show that the family 
\begin{equation*}
\left\{\theta \mapsto d\left(\frac{1}{n}\sum_{k=0}^{n-1} \delta_{X^{\theta_0}_{kh}},\mu_{\theta}\right), n \ge 1, \ \theta \in \Theta\right\} 
\end{equation*}
is equicontinuous. Actually, for any $\theta_1$ and $\theta_2$ in $\Theta$, we have

\begin{align*}
    \left| d\left(\frac{1}{n}\sum_{k=0}^{n-1} \delta_{X^{\theta_0}_{kh}},\mu_{\theta_1} \right) - d\left( \frac{1}{n}\sum_{k=0}^{n-1} \delta_{X^{\theta_0}_{kh}},\mu_{\theta_2} \right) \right| \leq  d(\mu_{\theta_1},\mu_{\theta_2}) .
\end{align*}
In view of Proposition~\ref{prop:continuity-distance-measures}, the term on the right-hand side goes to $0$ as $| \theta_1 - \theta_2 | \rightarrow 0$. This proves the equicontinuity and thus the uniform convergence.

\subsection{Proof of Theorem \ref{thm:rate}}\label{sec:rateofconv}
Since \ref{asmp:strong_identif} implies \ref{asmp:identif} and $d_{CF,p} \in \mathcal{D}_1 \subset \mathcal{D}_2$, we can apply Theorem \ref{thm:strong_consist} to obtain the strong consistency. For the rest of this section, $d$ always refer to the distance $d_{CF,p}$. We recall that $X = X^{\theta_0}$ denotes the observed process with the true parameter $\theta_0$.
In view of the strong identifiability assumption \ref{asmp:strong_identif}, it suffices to bound $ \mathbb{E}d(\mu_{\hat{\theta}_n},\mu_{\theta_0})^\alpha$ 
to obtain a rate of convergence on $ \mathbb{E}|\hat{\theta}_n - \theta_0 |^2$.

Our strategy is in line with the Section 5 of \cite{panloup2020general}, with adaptations due to the estimation of $\sigma$ and $H$. It is based on the following decomposition: since $\hat{\theta}_n$ minimizes the function $\theta \mapsto d(\frac{1}{n} \sum_{k=0}^{n-1} \delta_{X_{kh}},\mu_{\theta_0})$, we have
\begin{align*}%
    d(\mu_{\hat{\theta}_n},\mu_{\theta_0}) & \leq d\left(\frac{1}{n} \sum_{k=0}^{n-1} \delta_{X_{kh}},\mu_{\theta_0}\right) + d\left(\frac{1}{n} \sum_{k=0}^{n-1} \delta_{X_{kh}},\mu_{\hat{\theta}_n}\right)\\
    & \leq 2 d\left(\frac{1}{n} \sum_{k=0}^{n-1} \delta_{X_{kh}},\mu_{\theta_0}\right)  \\
   & =: 2 D_n^{(1)}  .
\end{align*}

\paragraph{$L^\alpha(\Omega)$ bound on $D_n^{(1)}$.}
Following the proof of \cite[Section 5.1]{panloup2020general}, we obtain a bound on $D_n^{(1)}$.

\begin{lemma}\label{lem:bound-D_n^1}
Assume that \ref{asmp:strong_identif} holds with $p$ and $\alpha$ satisfying $p>\frac{\alpha+d(q+1)}{2}$. There exists a positive constant $C_{\alpha,q}$ such that for any $n \in \mathbb{N}$,
\begin{align*}%
    \mathbb{E}[ | D_n^{(1)} |^\alpha] \leq  C_{\alpha,q} \left(n^{-\alpha} +  n^{-\frac{\alpha}{2}(2 - 2 \max(\mathcal{H}) \vee 1)} \right)  ,
\end{align*}
where we recall that $q$ is the number of linear transformations added to construct the augmented process $X^{\theta_0}$.
\end{lemma}

\begin{proof}
Decompose $D_n^{(1)}$ as $D_n^{(1)} \leq D_n^{(11)} + D_n^{(12)}$ where
\begin{align*}%
  D_n^{(11)} & := d\left(\mu_{\theta_0}, \frac{1}{n} \sum_{k=0}^{n-1} \mathbb{E}[\delta_{X_{kh}}]\right), \\
    D_n^{(12)} & := d\left(\frac{1}{n} \sum_{k=0}^{n-1} \mathbb{E}[\delta_{X_{kh}}],\frac{1}{n} \sum_{k=0}^{n-1} \delta_{X_{kh}}\right) .
\end{align*}
The expectation of the random measure $\mathbb{E}[\delta_{X_t}]$ is understood as a deterministic measure given by $\mathbb{E}[\delta_{X_t}](f) = \mathbb{E}[f(X_t)]$ for any bounded measurable $f$. 

As in the proof of \cite[Lemma 5.3]{panloup2020general}, the bound on $D_n^{(11)}$ relies on the pathwise exponential convergence of $Y_{t}-\bar{Y}_{t}$ towards $0$, which implies the same exponential convergence of $X_{t}-\bar{X}_{t}$. For the Lipschitz function $f_\chi(x) = e^{i\langle \chi, x \rangle}$, following the aforementioned proof leads to
\begin{align}\label{eq:Dn11-comparison}
\left|\frac{1}{n} \sum_{k=0}^{n-1} \mathbb{E} f_{\chi}\left(X_{kh}\right) - \mu_{\theta_0}(f_{\chi})\right| 
&\leq \frac{1}{n}\|f_{\chi}\|_{Lip} \sum_{k=0}^{n-1} \mathbb{E}\left[\left|X_{kh} -\bar{X}_{kh} \right|\right] \leq \frac{C_{q}}{n}\|f_{\chi}\|_{Lip}
\end{align}
and then
\begin{align}
\mathbb{E} [| D_n^{(11)} |^\alpha ] \leq \frac{C_{\alpha,q}}{n^\alpha}  .\label{eq:bound-D_n^11}
\end{align}

Let us now bound $D_n^{(12)}$. This is a concentration result inspired by \cite[Theorem 2.3]{varvenne2019concentration}. Namely, apply \cite[Lemma 5.5]{panloup2020general} to $\frac{1}{n} \sum_{k=0}^{n-1} f_\chi(X_{kh})$ to get
\begin{align*}
 \mathbb{E} \left[ \left| \frac{1}{n}\sum_{k=0}^{n-1} f_\chi (X_{kh}) - \mathbb{E}(f_\chi (X_{kh})) \right|^\alpha \right] \leq C_\alpha\,  \left\| f_{\chi} \right\|_{Lip}^\alpha n^{-\frac{\alpha}{2}{(2-\max(2H,1))}},
\end{align*}
for some positive constant $C_\alpha$ that depends on $\alpha$.

Using the definition of $d_{CF,p}$ and Jensen's inequality, one gets
\begin{align*}
\mathbb{E}[| D_n^{(12)} |^\alpha] \leq \int_{\R^{d(q+1)}} \mathbb{E} \left[ \left| \frac{1}{n}\sum_{k=0}^{n-1} f_\chi (X_{kh}) - \mathbb{E}(f_\chi (X_{kh})) \right|^\alpha \right] g_p(\chi) d\chi .
\end{align*}
As $\left\| f_{\chi} \right\|_{Lip} \leq | \chi |$, there is $\mathbb{E}[| D_n^{(12)} |^\alpha] \leq C_\alpha n^{-\frac{\alpha}{2}(2-\max(2H,1))} \int_{\R^{d(q+1)}} | \chi |^\alpha g_p(\chi) d\chi$, the integral being finite since $p>\frac{\alpha+d(q+1)}{2}$. This bound with \eqref{eq:bound-D_n^11} yield the result. 
\end{proof}

\section{Strong consistency and rate of convergence of the estimator $\hat{\theta}_{n,N,\gamma}$}\label{sec:practical}

In this section, the results on the estimator $\hat{\theta}_{n,N,\gamma}$ are proven: first, $\hat{\theta}_{n,N,\gamma}$ is shown to be strongly consistent (Theorem~\ref{thm:strong_consist2}), then the rate presented in Theorem~\ref{thm:rate-2} is established under the strong identifiability assumption \ref{asmp:strong_identif}.

\subsection{Proof of Theorem \ref{thm:strong_consist2} and Theorem \ref{thm:rate-2}}\label{subsec:inv_results}

%
First, using intermediary results that we shall prove in Section \ref{subsec:contrast_conv2} and Section \ref{sec:bounds-rate}, we provide a proof of the strong consistency and a rate of convergence for the estimator $\hat{\theta}_{n,N,\gamma}$ defined in \eqref{eq:thetan-02}.
\begin{proof}[Proof of Theorem \ref{thm:strong_consist2}]
We use again Proposition \ref{thm:convergence} with 
$$L_v(\theta) = d\left(\frac{1}{n} \sum_{k=0}^{n-1} \delta_{X_{k h}^{\theta_{0}}}, \frac{1}{N} \sum_{k=0}^{N-1} \delta_{X^{\theta,\gamma}_{k \gamma}}\right),$$
this time with $v=(n,N,\gamma)$. We will prove in Section \ref{subsec:contrast_conv2} that the contrast $L_v(\theta)$ converges uniformly as $(n,N,\gamma) \rightarrow (\infty, \infty, 0)$ to $L(\theta) = d(\mu_{\theta},\mu_{\theta_0})$, by first proving pointwise convergence and then using an equicontinuity argument. Since $L(\theta)$ is the same as in Section \ref{sec:consist}, we have by Proposition~\ref{prop:continuity-distance-measures} that $L(\theta)$ is continuous. Then we apply Proposition \ref{thm:convergence} to conclude. 
\end{proof}

\begin{proof}[Proof of Theorem \ref{thm:rate-2}]
Since \ref{asmp:strong_identif} implies \ref{asmp:identif} and $d_{CF,p} \in \mathcal{D}_1 \subset \mathcal{D}_2$, we can apply Theorem \ref{thm:strong_consist2} to obtain the strong consistency. To prove the convergence above, we proceed similarly to Section \ref{sec:rateofconv}. We decompose the term $\mathbb{E}d(\mu_{\hat{\theta}_{n,N, \gamma}},\mu_{\theta_0})^\alpha$ slightly differently. First we use the triangle inequality to get 
\begin{align*}
    d\left(\mu_{\theta_{0}}, \mu_{\hat{\theta}_{n,N, \gamma}}\right) & \leq d\left(\mu_{\theta_{0}}, \frac{1}{n} \sum_{k=0}^{n-1} \delta_{X^{\theta_0}_{kh}}\right)+d\left(\frac{1}{n} \sum_{k=0}^{n-1} \delta_{X^{\theta_0}_{kh}}, \frac{1}{N} \sum_{k=0}^{N-1} \delta_{X_{k \gamma}^{\hat{\theta}_{n,N,\gamma}, \gamma}}\right) \\ 
    & \quad + d\left(\frac{1}{N} \sum_{k=0}^{N-1} \delta_{X_{k \gamma}^{\hat{\theta}_{n,N,\gamma}, \gamma}}\, ,\, \mu_{\hat{\theta}_{n,N,\gamma}}\right) .
\end{align*}
Now, since $\hat{\theta}_{n,N,\gamma}$ minimizes the function  $\theta \mapsto d\left(\frac{1}{n} \sum_{k=0}^{n-1} \delta_{X^{\theta_0}_{k h}}, \frac{1}{N} \sum_{k=0}^{N-1} \delta_{X^{\theta,\gamma}_{k \gamma}}\right)$, we can further bound $d(\mu_{\theta_{0}}, \mu_{\hat{\theta}_{n,N, \gamma}})$ as
\begin{align}\label{eq:bound-d1}
    d \left(\mu_{\theta_{0}}, \mu_{\hat{\theta}_{n,N, \gamma}}\right) & \leq d\left(\mu_{\theta_{0}}, \frac{1}{n} \sum_{k=0}^{n-1} \delta_{X^{\theta_0}_{kh}}\right)+d\left(\frac{1}{n} \sum_{k=0}^{n-1} \delta_{X^{\theta_0}_{kh}}, \frac{1}{N} \sum_{k=0}^{N-1} \delta_{X_{k \gamma}^{\theta_0, \gamma}}\right) \nonumber \\ & \quad + \sup_{\theta \in \Theta} d\left(\frac{1}{N} \sum_{k=0}^{N-1} \delta_{X^{\theta,\gamma}_{k \gamma}}, \mu_{\theta}\right) .
\end{align}
To allow pathwise comparison, let us define the following processes. For any $\theta \in \Theta$, define $Y^{\theta,\gamma}_\cdot(B)$, an Euler scheme of $Y^\theta$ defined with the same fBm $B$. Namely, $Y^{\theta,\gamma}_\cdot(B)$ is defined by \eqref{eq:euler-scheme} where $\hatB$ is replaced by $B$. As in Section \ref{subsec:not-assump}, define $X^{\theta,\gamma}(B)$ by
\begin{align*}
   X^{\theta,\gamma}(B) = \left(Y^{\theta,\gamma}_\cdot(B), \ell^{1} (Y^{\theta,\gamma}_\cdot(B), Y^{\theta,\gamma}_{\cdot+h}(B)),\dots,\ell^q( Y^{\theta,\gamma}_\cdot(B),\dots, Y^{\theta,\gamma}_{\cdot+qh}(B))\right) .
\end{align*}
We also define $Y(\hatB)$ which is the solution to \eqref{eq:fsde} with the fBm $\hatB$, and similarly we define $X(\hatB)$. Now, we can do pathwise comparison between $X^\theta$ and $X^{\theta,\gamma}(B)$, and between $X^{\theta,\gamma}$ and $X^{\theta}(\hatB)$.
\paragraph{Bounding the second term in \eqref{eq:bound-d1}.}
Split the second term in the right-hand side of \eqref{eq:bound-d1} as follows
\begin{align}\label{eq:bound-d2}
\begin{split}
 d\left(\frac{1}{n} \sum_{k=0}^{n-1} \delta_{X^{\theta_0}_{kh}}, \frac{1}{N} \sum_{k=0}^{N-1} \delta_{X_{k \gamma}^{\theta_{0}, \gamma}}\right) & \leq d\left(\frac{1}{n} \sum_{k=0}^{n-1} \delta_{X^{\theta_0}_{kh}}, \mu_{\theta_{0}}\right) +d\left(\mu_{\theta_{0}}, \frac{1}{N} \sum_{k=0}^{N-1} \delta_{X_{k \gamma}^{\theta_{0}, \gamma}}\right)   \\
&\quad +d\left(\frac{1}{N} \sum_{k=0}^{N-1} \delta_{X_{k \gamma}^{\theta_{0}}}, \frac{1}{N} \sum_{k=0}^{N-1} \delta_{X_{k \gamma}^{\theta_{0}, \gamma}}\right) .
\end{split}
\end{align}
Furthermore, split the last term above as
\begin{align*}
d\left(\frac{1}{N} \sum_{k=0}^{N-1} \delta_{X_{k \gamma}^{\theta_{0}}}, \frac{1}{N} \sum_{k=0}^{N-1} \delta_{X_{k \gamma}^{\theta_{0}, \gamma}}\right)
& \leq d\left(\frac{1}{N} \sum_{k=0}^{N-1} \delta_{X_{k \gamma}^{\theta_{0}}}, \frac{1}{N} \sum_{k=0}^{N-1} \delta_{X_{k \gamma}^{\theta_0,\gamma}(B)} \right) \\ 
& \quad + d \left( \frac{1}{N} \sum_{k=0}^{N-1} \delta_{X_{k \gamma}^{\theta_0,\gamma}(B)}, \mu_{\theta_0} \right) + d \left( \frac{1}{N} \sum_{k=0}^{N-1} \delta_{X^{\theta_0, \gamma}_{k \gamma}} , \mu_{\theta_0}\right) . 
\end{align*}
Moreover,
\begin{align*}
& d\left(\frac{1}{N} \sum_{k=0}^{N-1} \delta_{X_{k \gamma}^{\theta_{0}}}, \frac{1}{N} \sum_{k=0}^{N-1} \delta_{X_{k \gamma}^{\theta_{0}, \gamma}}\right) \\ & \leq  2d\left(\frac{1}{N} \sum_{k=0}^{N-1} \delta_{X_{k \gamma}^{\theta_{0}}}, \frac{1}{N} \sum_{k=0}^{N-1} \delta_{X_{k \gamma}^{\theta_0,\gamma}(B)} \right)  + d \left( \frac{1}{N} \sum_{k=0}^{N-1} \delta_{X_{k \gamma}^{\theta_{0}}}, \mu_{\theta_0} \right) \\ & \quad  + d \left( \frac{1}{N} \sum_{k=0}^{N-1} \delta_{X^{\theta_0,\gamma}_{k \gamma}} , \frac{1}{N} \sum_{k=0}^{N-1} \delta_{X_{k \gamma}^{\theta_0}(\hatB)} \right)  + d \left(\frac{1}{N} \sum_{k=0}^{N-1} \delta_{X_{k \gamma}^{\theta_0}(\hatB)}, \mu_{\theta_0}\right) .
\end{align*}
Injecting the above bound into \eqref{eq:bound-d2}, we get
\begin{align}\label{eq:bound-d22}
\begin{split}
 & d\left(\frac{1}{n} \sum_{k=0}^{n-1} \delta_{X^{\theta_0}_{kh}}, \frac{1}{N} \sum_{k=0}^{N-1} \delta_{L_{k \gamma}^{\theta_{0}}}\right) \\ & \leq d\left(\frac{1}{n} \sum_{k=0}^{n-1} \delta_{X^{\theta_0}_{kh}}, \mu_{\theta_{0}}\right) + 2d\left(\mu_{\theta_{0}}, \frac{1}{N} \sum_{k=0}^{N-1} \delta_{X_{k \gamma}^{\theta_{0}}}\right)  +  d \left(\frac{1}{N} \sum_{k=0}^{N-1} \delta_{X_{k \gamma}^{\theta_0}(\hatB)}, \mu_{\theta_0}\right)    \\
&+ 2d\left(\frac{1}{N} \sum_{k=0}^{N-1} \delta_{X_{k \gamma}^{\theta_{0}}}, \frac{1}{N} \sum_{k=0}^{N-1} \delta_{X_{k \gamma}^{\theta_0,\gamma}(B)} \right)+d \left( \frac{1}{N} \sum_{k=0}^{N-1} \delta_{X^{\theta_0,\gamma}_{k \gamma}} , \frac{1}{N} \sum_{k=0}^{N-1} \delta_{X_{k \gamma}^{\theta_0}(\hatB)} \right)    .
\end{split}
\end{align}
\paragraph{Bounding the third term in \eqref{eq:bound-d1}.}
Split the third term in \eqref{eq:bound-d1} as follows
\begin{align}\label{eq:bound-d3}
\begin{split}
& \sup_{\theta \in \Theta} d\left(\frac{1}{N} \sum_{k=0}^{N-1} \delta_{X^{\theta,\gamma}_{k \gamma}}, \mu_{\theta}\right) \\ & \leq  \sup_{\theta \in \Theta} d\left(\frac{1}{N} \sum_{k=0}^{N-1} \delta_{X^{\theta,\gamma}_{k \gamma}}, \frac{1}{N} \sum_{k=0}^{N-1} \delta_{X_{k \gamma}^{\theta_0}(\hatB)} \right)  +  \sup_{\theta \in \Theta} d\left(\frac{1}{N} \sum_{k=0}^{N-1} \delta_{X_{k \gamma}^{\theta_0}(\hatB)}, \mu_{\theta}\right) .
\end{split} 
\end{align} 

\paragraph{Final bound on $d \left(\mu_{\theta_{0}}, \mu_{\hat{\theta}_{n,N, \gamma}}\right) $.} 
Using \eqref{eq:bound-d22} and \eqref{eq:bound-d3} in \eqref{eq:bound-d1}, we get
\begin{align}\label{eq:decomp-d}
& d \left(\mu_{\theta_{0}}, \mu_{\hat{\theta}_{n,N, \gamma}}\right)  \nonumber \\ & \leq 2 d\left(\frac{1}{n} \sum_{k=0}^{n-1} \delta_{X^{\theta_0}_{kh}}, \mu_{\theta_{0}}\right)   + 2d\left( \frac{1}{N} \sum_{k=0}^{N-1} \delta_{X_{k \gamma}^{\theta_{0}}}, \mu_{\theta_{0}} \right)  +  d \left(\frac{1}{N} \sum_{k=0}^{N-1} \delta_{X_{k \gamma}^{\theta_0}(\hatB)}, \mu_{\theta_0}\right) \nonumber  \\
& \quad + \sup_{\theta \in \Theta} d\left(\frac{1}{N} \sum_{k=0}^{N-1} \delta_{X_{k \gamma}^{\theta}(\hatB)}, \mu_{\theta}\right)  \nonumber \\
& \quad +2 \sup_{\theta \in \Theta}  d\left(\frac{1}{N} \sum_{k=0}^{N-1} \delta_{X_{k \gamma}^{\theta}}, \frac{1}{N} \sum_{k=0}^{N-1} \delta_{X_{k \gamma}^{\theta_0,\gamma}(B)} \right)+ 2 \sup_{\theta \in \Theta} d \left( \frac{1}{N} \sum_{k=0}^{N-1} \delta_{X^{\theta,\gamma}_{k \gamma}} , \frac{1}{N} \sum_{k=0}^{N-1} \delta_{X_{k \gamma}^{\theta}(\hatB)} \right)   .
\end{align}
The first three terms on the right-hand side can be bounded exactly as the term $D_n^{(11)}$ in the proof of Lemma \ref{lem:bound-D_n^1}, one thus gets
\begin{align}
d\left(\frac{1}{n} \sum_{k=0}^{n-1} \delta_{X^{\theta_0}_{kh}}, \mu_{\theta_{0}}\right) & \leq  C_{\alpha,q} \left(n^{-\alpha} + n^{-\frac{\alpha}{2}(2 - 2 \max(\mathcal{H}) \vee 1)}\right)   \label{eq:1st-term}  \\
d\left( \frac{1}{N} \sum_{k=0}^{N-1} \delta_{X_{k \gamma}^{\theta_{0}}}, \mu_{\theta_{0}} \right)  & \leq C_{\alpha,q} \left(N^{-\alpha} +  N^{-\frac{\alpha}{2}(2 - 2 \max(\mathcal{H}) \vee 1)}\right)  \label{eq:2nd-term} \\ 
 d \left(\frac{1}{N} \sum_{k=0}^{N-1} \delta_{X_{k \gamma}^{\theta_0}(\hatB)}, \mu_{\theta}\right) & \leq  C_{\alpha,q} \left(N^{-\alpha} +  N^{-\frac{\alpha}{2}(2 - 2 \max(\mathcal{H}) \vee 1)}\right) \label{eq:3d-term}.
\end{align}
\begin{remark}
For the term $d \left(\frac{1}{N} \sum_{k=0}^{N-1} \delta_{X_{k \gamma}^{\theta_0}(\hatB)}, \mu_{\theta_0}\right)$, notice that $\mu_{\theta_0}$ is also the law of $\bar{X}^{\theta_0, \hat{B}}$, the stationary augmented process associated to \eqref{eq:fsde0} with the fBm $\hat{B}$ instead of $B$, so \eqref{eq:Dn11-comparison} in the proof of Lemma \ref{lem:bound-D_n^1} still holds since we compare two solutions with the same noise, and therefore we know that they converge exponentially to each other as $t\rightarrow \infty$ by Proposition \ref{prop:general-results}.
\end{remark}
Let us define 
\begin{equation}\label{eq:decom-D1-2-3}
\begin{split}
D_{N, \gamma}^{(21)}(\theta) & := d\left(\frac{1}{N} \sum_{k=0}^{N-1} \delta_{X_{k \gamma}^{\theta}}, \frac{1}{N} \sum_{k=0}^{N-1} \delta_{X_{k \gamma}^{\theta,\gamma}(B)} \right)   \\
D_{N, \gamma}^{(22)}(\theta) & := d \left( \frac{1}{N} \sum_{k=0}^{N-1} \delta_{X^{\theta,\gamma}_{k \gamma}} , \frac{1}{N} \sum_{k=0}^{N-1} \delta_{X_{k \gamma}^{\theta}(\hatB)} \right)    \\
D_{N, \gamma}^{(3)}(\theta) & := d\left(\frac{1}{N} \sum_{k=0}^{N-1} \delta_{X_{k \gamma}^{\theta}(\hatB)}, \mu_{\theta}\right)  .
\end{split}
\end{equation}
In Section \ref{sec:bounds-rate}, we show how to bound the moments of  $\sup_{\theta \in \Theta} D_{N,\gamma}^{(21)}(\theta),\ \sup_{\theta \in \Theta} D_{N,\gamma}^{(22)}(\theta)$ and $\sup_{\theta \in \Theta}  D_{N, \gamma}^{(3)}(\theta)$. Namely, we prove that for any $\varepsilon < \alpha\min(\mathcal{H})$ and any $\varpi\in (0,1)$, there exist constants $C_{\alpha,\varepsilon}$ and $C_{\alpha,\varepsilon,\varpi}$ such that for any $N \ge 1$ and $\gamma\leq\gamma_{0}$ with $N \gamma \ge 1$, the following bounds hold:
\begin{align}
    \mathbb{E}\sup_{\theta \in \Theta} \left|D_{N,\gamma}^{(21)}(\theta)\right|^{\alpha} & \leq C_{\alpha, \varepsilon}  \gamma^{\alpha \min(\mathcal{H})-\varepsilon}  \label{eq:bounds-rate1}\\
    \mathbb{E}\sup _{\theta \in \Theta}\left|D_{N, \gamma}^{(22)}(\theta)\right|^{\alpha} & \leq C_{\alpha, \varepsilon}  \gamma^{\alpha \min(\mathcal{H})-\varepsilon} \label{eq:bounds-rate2} \\
    \mathbb{E}\sup _{\theta \in \Theta} \left| D_{N, \gamma}^{(3)}(\theta) \right|^{\alpha} & \leq C_{\alpha, \varepsilon, \varpi} \left(\gamma^{\alpha \min(\mathcal{H}) -\varepsilon}+ (N \gamma)^{-\bar{\eta}}\right)  \label{eq:bounds-rate3},
\end{align}
with $\bar{\eta}=\frac{\varpi \alpha^{2}}{2(\alpha \varpi+2d)}(2-(2 \max(\mathcal{H}) \vee 1))$. Injecting the bounds \eqref{eq:1st-term}, \eqref{eq:2nd-term}, \eqref{eq:3d-term}, \eqref{eq:bounds-rate1}, \eqref{eq:bounds-rate2} and \eqref{eq:bounds-rate3} into the decomposition \eqref{eq:decomp-d} concludes the proof.
\end{proof}

\subsection{Proof of the uniform convergence of the contrast}\label{subsec:contrast_conv2}

In this section, we obtain the uniform convergence of the contrast 
$$(n,N,\gamma) \mapsto d\left(\frac{1}{n} \sum_{k=0}^{n-1} \delta_{X_{k h}^{\theta_{0}}}, \frac{1}{N} \sum_{k=0}^{N-1} \delta_{X^{\theta,\gamma}_{k \gamma}}\right)$$
 towards $d(\mu_{\theta_{0}},\mu_{\theta})$, that is used in the proof of Theorem~\ref{thm:strong_consist2}. First we prove that almost surely, there is convergence as $(n,N,\gamma) \rightarrow (\infty, \infty, 0)$ for each fixed $\theta$. We have already proven in Section \ref{subsec:contrast_conv} that $d(\frac{1}{n} \sum_{k=0}^{n-1} \delta_{X_{k h}^{\theta_{0}}}, \mu_{\theta})$ converges to $d(\mu_{\theta_0}, \mu_{\theta})$ as $n$ goes to infinity. By Proposition \ref{prop:dconv-euler}, $d(\frac{1}{N} \sum_{k=0}^{N-1} \delta_{X^{\theta,\gamma}_{k \gamma}}, \mu_{\theta}^\gamma)$ converges to $0$ as $N \rightarrow \infty$. Finally, we prove in Proposition \ref{prop:conv_contrast}$(i)$ that $d(\mu_{\theta}, \mu_{\theta}^\gamma)$ converges to $0$ as $\gamma \rightarrow 0$. Therefore 
\begin{align*}
d\left(\frac{1}{n} \sum_{k=0}^{n-1} \delta_{X_{k h}^{\theta_{0}}}, \frac{1}{N} \sum_{k=0}^{N-1} \delta_{X^{\theta,\gamma}_{k \gamma}}\right) \underset{\substack{(n,N,\gamma) \rightarrow (\infty, \infty, 0)}}{\longrightarrow}   d(\mu_{\theta_{0}},\mu_{\theta}) .
\end{align*}

We extend the convergence result to a uniform convergence in $\theta$ in the following Proposition.

\begin{prop}\label{prop:conv_contrast}
Let $0 < p \leq 2$ and $d \in \mathcal{D}_p$. Under the assumptions \ref{asmp:compact}, \ref{asmp:drift} and \ref{asmp:identif}, there exists $\gamma_0>0$ such that for all $\gamma \in (0,\gamma_0]$, the following assertions hold true.
\begin{itemize}
    \item[(i)] $\displaystyle\lim_{\gamma \rightarrow 0} \sup_{\theta \in \Theta} d(\mu_{\theta},\mu_{\theta}^\gamma) =0$.
    \item[(ii)] $\displaystyle\lim_{N \rightarrow \infty} \sup_{\theta \in \Theta} d\left(\frac{1}{N} \sum_{k=0}^{N-1}  \delta_{X_{k\gamma}^{\theta,\gamma}},  \mu_\theta^\gamma \right)=0$.
    \item[(iii)] $\displaystyle\lim_{\gamma \rightarrow 0} \lim_{n,N \rightarrow \infty} \sup_{\theta \in \Theta} \left| d\left(\frac{1}{n} \sum_{k=0}^{n-1}  \delta_{X_{kh}^{\theta_0}}, \frac{1}{N} \sum_{k=0}^{N-1}  \delta_{X_{k\gamma}^{\theta,\gamma}}\right) - d(\mu_{\theta_0}, \mu_{\theta}) \right| = 0$. 
\end{itemize}
\end{prop}

\begin{proof}
Notice that $(iii)$ is a simple consequence of the previous statements $(i)$ and $(ii)$.
\paragraph{Proof of $(i)$.}
By the triangle inequality,
\begin{align*}
d(\mu_\theta,\mu_\theta^\gamma) \leq d(\mu_\theta,\mathcal{L}(X_{N \gamma}^\theta)) + d(\mu_\theta^\gamma,\mathcal{L}(X_{N \gamma}^{\theta,\gamma} ))+ d(\mathcal{L}(X_{N \gamma}^{\theta,\gamma}),\mathcal{L}(X_{N \gamma}^\theta)) .
\end{align*}
Since $d$ is bounded by the 2-Wasserstein distance, for all $N\ge 1$ there is
\begin{align}\label{eq:defW123}
d(\mu_\theta,\mu_\theta^\gamma) &\leq \mathcal{W}_2(\mu_\theta,\mathcal{L}(X_{N \gamma}^\theta)) + \mathcal{W}_2(\mu_\theta^\gamma,\mathcal{L}(X_{N \gamma}^{\theta,\gamma} ))+ \mathcal{W}_2(\mathcal{L}(X_{N \gamma}^{\theta,\gamma}),\mathcal{L}(X_{N \gamma}^\theta)) \nonumber \\
&=: W^{(1)} + W^{(2)} + W^{(3)}.
\end{align}
As for $W^{(1)}$, we have
\begin{align*}
W^{(1)} = \mathcal{W}_2(\mu_\theta,\mathcal{L}(X_{N \gamma}^\theta))  \leq
\left(\mathbb{E}|   X_{N \gamma}^\theta  - \bar{X}_{N \gamma}^\theta |^2\right)^\frac{1}{2} .
\end{align*}
By Lemma~\ref{prop:general-results}, the right-hand side term converges to $0$ as $N\to \infty$ uniformly in $\theta$. We now look at the second term:
\begin{align}\label{eq:bound-1}
W^{(2)} = \mathcal{W}_2(\mu_\theta^\gamma, \mathcal{L}(X_{N \gamma}^{\theta,\gamma})) & \leq \left(\mathbb{E}|\bar{X}_{N  \gamma}^{\theta,\gamma} - X_{N \gamma}^{\theta,\gamma}|^2 \right)^\frac{1}{2} \nonumber \\
& \leq C_q  \sum_{i=0}^{q} \left(\mathbb{E}|\bar{Y}_{N  \gamma+ih}^{\theta,\gamma} - Y_{N \gamma+ih}^{\theta,\gamma}|^2 \right)^\frac{1}{2} .
\end{align}
By \cite[Equation (4.2)]{panloup2020general}, we have for any $k \in \mathbb{N}$,
\begin{align}\label{eq:bound-1bis}
\left|\bar{Y}_{k \gamma}^{\theta,\gamma}-Y_{k \gamma}^{\theta,\gamma}\right|^2 \leq \left(1-2 \gamma \beta+\gamma^2 K^2\right)^k\left|\bar{Y}_0^{\theta,\gamma}- Y_0^{\theta,\gamma}\right|^2.
\end{align}
Furthermore, for any $i \in \llbracket 0,q \rrbracket$, there exists $j \in \mathbb{N}$ such that $Y_{N \gamma+ih}^{\theta,\gamma}=Y_{j \gamma}^{\theta,\gamma}$ and $\bar{Y}_{N  \gamma+ih}^{\theta,\gamma} = \bar{Y}_{j  \gamma}^{\theta,\gamma}$. Therefore, the bound \eqref{eq:bound-1bis} holds for all the terms in \eqref{eq:bound-1}. We conclude that there exists $\gamma_0>0$, such that for $\gamma \leq \gamma_0$, the second term goes to $0$ uniformly in $\theta$ when $N\to \infty$. Now for the last term in \eqref{eq:defW123}, by definition of the Wasserstein distance, we have 
\begin{align*}
W^{(3)} = \mathcal{W}_2( \mathcal{L}(X_{N \gamma}^\theta), \mathcal{L}(X_{N \gamma}^{\theta,\gamma})) & \leq \left( \mathbb{E} |X_{N  \gamma}^\theta - X_{N \gamma}^{\theta,\gamma}(B) |^2 \right)^\frac{1}{2} \\
& \leq C_q  \sum_{i=0}^{q} \left( \mathbb{E} | Y_{N  \gamma+ih}^\theta - Y_{N \gamma+ih}^{\theta,\gamma}(B) |^2 \right)^\frac{1}{2} .
\end{align*}
In \cite[Proposition 3.7 (i)]{panloup2020general}, it was proved that there exists positive constants $C$ and $\rho$ that depend only the Lipschitz constant $K$ from \ref{asmp:drift}  such that for any $m \in \mathbb{N}$,
\begin{align*}
| Y_{m \gamma}^{\theta}-Y_{m \gamma}^{\theta,\gamma}(B) |^2 \leq C \sum_{j=0}^{m-1} \phi_{j}( Y_{j \gamma}^{\theta,\gamma}(B) ) e^{-\rho \gamma (m-j+1)} ,
\end{align*}
where $\phi_{j} (z) = \gamma^{3} |b_\xi(z) |^2 + \int_0^{\gamma} |B_{j \gamma + t} - B_{j \gamma} |^{2} dt$.
This pathwise comparison is possible because the two processes are defined with the same noise $B$.
Since $b_\xi$ is uniformly sub-linear, it follows that
\begin{align}\label{eq:bound-3}
| Y_{m \gamma}^{\theta}-Y_{m \gamma}^{\theta,\gamma}(B) |^2 \leq C  \sum_{j=0}^{m-1} \left( \gamma^{3} (1+ |Y_{j \gamma}^{\theta,\gamma}(B)|^{2r})+  \int_0^\gamma |B_{j \gamma + t} - B_{j \gamma} |^{2} d t \right) e^{-\rho \gamma (m-j+1)}  .
\end{align}
Now for $i \in \llbracket 0,q \rrbracket$ and $k \in \mathbb{N}$, since the process $Y^{\theta,\gamma}$ is constant over intervals of size $\gamma$, recalling the notation $t_\gamma = \gamma \lfloor \frac{t}{\gamma} \rfloor$, we can always write 
\begin{align}\label{eq:interpol-decom}
Y_{k \gamma+ih}^{\theta}-Y_{k \gamma+ih}^{\theta,\gamma}(B) =\left( Y_{(k \gamma + ih)_{\gamma}+\varepsilon_{k,i}}^\theta- Y_{(k \gamma + ih)_{\gamma} }^{\theta} \right) +  \left(Y_{(k \gamma + ih)_{\gamma} }^{\theta}-Y_{(k \gamma + ih)_{\gamma}}^{\theta,\gamma}(B)\right),
\end{align}
where 
$$\varepsilon_{k,i} =k \gamma + i h - (k \gamma + i h)_\gamma  < \gamma. $$
For the first term in \eqref{eq:interpol-decom}, using the sub-linear growth of $b$, we write
\begin{align*}
& | Y_{(k \gamma + i h)_\gamma  +\varepsilon_{k,i}}^\theta- Y_{(k \gamma + i h)_\gamma  }^{\theta} | \\ & \leq C \left( \int_{(k \gamma + ih)_{\gamma}}^{(k \gamma + ih)_{\gamma}+\varepsilon} (1+| Y_s^\theta|^{r} ) ds + | B_{(k \gamma + ih)_{\gamma}+\varepsilon_{k,i}}-B_{(k \gamma + ih)_{\gamma}} | \right) .
\end{align*}
It follows from Jensen's inequality that
\begin{align} \label{eq:epsilon-bound} 
\begin{split}
&  | Y_{j(k \gamma + ih)_{\gamma}+\varepsilon_{k,i}}^\theta- Y_{(k \gamma + i h)_\gamma }^{\theta} |^2 \\ & \leq C \left( \varepsilon_{k,i} \int_{(k \gamma + ih)_{\gamma}}^{(k \gamma + ih)_{\gamma}+\varepsilon_{k,i}} (1+| Y_s^\theta|^{2r} ) ds +  | B_{(k \gamma + ih)_{\gamma}+\varepsilon_{k,i}}-B_{(k \gamma + ih)_{\gamma}} |^2 \right).
\end{split}
 \end{align}
The second term in \eqref{eq:interpol-decom} can be bounded using \eqref{eq:bound-3} with $m \equiv \frac{(k \gamma + ih)_{\gamma}}{\gamma}$. Combining this and \eqref{eq:epsilon-bound} in \eqref{eq:interpol-decom}, we get that for any $k \in \mathbb{N}$,
\begin{equation}\label{eq:bound-3-interpol}
\begin{split}
 & \sum_{i=0}^q |Y_{k \gamma+ih}^\theta - Y_{k \gamma+ih}^{\theta,\gamma}(B)|^2   \\ &
 \leq C  \sum_{i=0}^q  \gamma \int_{(k \gamma + ih)_{\gamma}}^{(k \gamma + ih)_{\gamma}+\varepsilon_{k,i}} (1+| Y_s^\theta|^{2r} ) ds +| B_{(k \gamma + ih)_{\gamma}+\varepsilon_{k,i}}-B_{(k \gamma + ih)_{\gamma}} |^2 \\ 
 & \quad + C \sum_{i=0}^q \sum_{j=0}^{\lfloor \frac{ k \gamma + ih}{\gamma} \rfloor -1} \Big( \gamma^{2} (1+ |Y_{j \gamma}^{\theta,\gamma}(B)|^{2r}) + \gamma^{-1} \int_0^\gamma |B_{j \gamma + t} - B_{j \gamma} |^{2} d t \Big) \gamma e^{-\rho(\lfloor \frac{ k \gamma + ih}{\gamma} \rfloor -j+1)}.
\end{split}
\end{equation}
Taking the expectation, using $\displaystyle\limsup_{\substack{n \rightarrow \infty\\ \gamma \rightarrow 0} } \gamma \sum_{j=0}^{n} e^{-\rho(n-j+1)} < + \infty$ and $r \leq 1$, we get
\begin{align*}
& \sup_{\theta \in \Theta,  \gamma \in (0,\gamma_0)} \limsup_{k \rightarrow \infty}  \gamma^{-2\max(\mathcal{H})} \sum_{i=0}^q \mathbb{E}\left|Y_{N \gamma+ih}^\theta - Y_{N \gamma+ih}^{\theta,\gamma}(B) \right|^2 \\ 
& \quad \leq C \left(  1 +\sup_{\theta \in \Theta,  \gamma \in (0,\gamma_0)} \limsup_{k \rightarrow \infty} \mathbb{E}| Y_{k \gamma}^{\theta,\gamma}(B) |^{2r} +\sup_{\theta \in \Theta} \limsup_{t \rightarrow \infty} \mathbb{E}| Y_{t}^{\theta} |^{2r}  \right)  . 
\end{align*}
Using Proposition \ref{prop:uniform-Lq-bounds-L}$(i)$ and Proposition \ref{prop:uniform-Lq-bounds}$(i)$, it follows that there exists $\gamma_0$ such that for $\gamma \leq \gamma_0$, the right-hand side is finite. This concludes the proof of $(i)$.

\paragraph{Proof of $(ii)$.} We already know that the convergence is true for fixed $\theta$.  In order to extend the result to uniform convergence, we show that the family $\{ \theta \mapsto d(\frac{1}{N} \sum_{k=0}^N \delta_{X_{k\gamma}^{\theta,\gamma}},  \mu_\theta^\gamma); N \ge 1 ; \theta \in  \Theta \}$ is equicontinuous for a fixed $\gamma \in (0,\gamma_0]$. For some $\theta_1$ and $\theta_2$  in $\Theta$, there is
\begin{align*}
& \left|  d\left(\frac{1}{N} \sum_{k=0}^{N-1} \delta_{X_{k\gamma}^{\theta_1,\gamma}},  \mu_{\theta_1}^\gamma\right)  -  d\left(\frac{1}{N} \sum_{k=0}^{N-1} \delta_{X_{k\gamma}^{\theta_2,\gamma}},  \mu_{\theta_2}^\gamma\right) \right| \\ 
&\quad \leq d(\mu_{\theta_1}^\gamma,\mu_{\theta_2}^\gamma) +  d\left(\frac{1}{N} \sum_{k=0}^{N-1} \delta_{X_{k\gamma}^{\theta_1,\gamma}},\frac{1}{N} \sum_{k=0}^{N-1} \delta_{X_{k\gamma}^{\theta_2,\gamma}}\right) .
\end{align*}
Decompose the second term to get
\begin{align*}
d\left(\frac{1}{N} \sum_{k=0}^{N-1} \delta_{X_{k\gamma}^{\theta_1,\gamma}},\frac{1}{N} \sum_{k=0}^{N-1}  \delta_{X_{k\gamma}^{\theta_2,\gamma}}\right)^2 
& \leq C \mathcal{W}_2\left(\frac{1}{N} \sum_{k=0}^{N-1}  \delta_{X_{k\gamma}^{\theta_1,\gamma}},\frac{1}{N} \sum_{k=0}^{N-1}  \delta_{X_{k\gamma}^{\theta_2,\gamma}}\right)^2 \\
&  \leq C \frac{1}{N} \sum_{k=0}^{N-1} |  X_{k\gamma}^{\theta_1,\gamma}-X_{k\gamma}^{\theta_2,\gamma} |^2 \\
& \leq C_q   \sum_{i=0}^{q} \frac{1}{N} \sum_{k=0}^{N-1}   |  Y_{k\gamma+ih}^{\theta_1,\gamma}-Y_{k\gamma+ih}^{\theta_2,\gamma} |^2 .
\end{align*}
Let $\varpi \in (0,1)$ and $p \ge 1$. By Proposition \ref{prop:regularity-result},  there exists a random variable $\mathbf{C}$ with finite moments of order $p$ such that for all $\theta_1, \theta_2 \in \Theta$,
\begin{align*}
\frac{1}{N} \sum_{k=0}^{N-1}   |  Y_{k\gamma}^{\theta_1,\gamma}-Y_{k\gamma}^{\theta_2,\gamma} |^2  & \leq \mathbf{C} \left( 1 \wedge | \theta_1 -\theta_2 |^{\varpi} \right) .
\end{align*}

These results still hold when replacing $Y_{k\gamma}^{\theta,\gamma}$ by $Y_{k\gamma+ih}^{\theta,\gamma}$, since we compare two piecewise constant processes. Thus $d(\frac{1}{N} \sum_{k=0}^{N-1} \delta_{X_{k\gamma}^{\theta_1,\gamma}},\frac{1}{N} \sum_{k=0}^{N-1}  \delta_{X_{k\gamma}^{\theta_2,\gamma}})$ goes to $0$ as $| \theta_1 - \theta_2 |  \rightarrow 0$ uniformly in $N$.  The same goes for  $d(\mu_{\theta_1}^\gamma,\mu_{\theta_2}^\gamma) $  by taking the limit $N  \rightarrow \infty$. This concludes the proof of the equicontinuity and therefore the proof of $(ii)$.
\end{proof}

\subsection{Proof of the bounds \eqref{eq:bounds-rate1}, \eqref{eq:bounds-rate2} and \eqref{eq:bounds-rate3}}\label{sec:bounds-rate}

We prove here the bounds \eqref{eq:bounds-rate1}, \eqref{eq:bounds-rate2} and \eqref{eq:bounds-rate3} on $D_{N, \gamma}^{(21)}$, $D_{N,\gamma}^{(22)}$ and $D_{N, \gamma}^{(3)}$ that were defined in \eqref{eq:decom-D1-2-3}. In this section, $d$ always refer to the distance $d_{CF,p}$.

\begin{prop}\label{prop:boud-rateofconv}
Recall that $\alpha$ is the exponent in the strong identifiability assumption \ref{asmp:strong_identif}. Assume that the exponent $r$ in the sub-linear growth of $b_\xi$ in \eqref{eq:drift-growth} satisfies $r \leq 1$.
For any $\varepsilon\in (0, \alpha\min(\mathcal{H}))$ and any $\varpi\in (0,1)$, there exist constants $C_{\alpha,\varepsilon}>0$ and $C_{\alpha,\varepsilon,\varpi}>0$  such that for all $\gamma \in (0, \gamma_0]$ and $N \ge 1$ satisfying $N \gamma \ge 1$, the inequalities \eqref{eq:bounds-rate1}, \eqref{eq:bounds-rate2} and \eqref{eq:bounds-rate3} hold.
\end{prop}

\begin{proof}
First, observe that for both $D_{N,\gamma}^{(21)}$ and $D_{N, \gamma}^{(22)}$ we compare a solution of an SDE with its respective Euler scheme, with both processes defined with the same noise $B$. This allows to do a pathwise comparison. We only detail the bound on $D_{N,\gamma}^{(21)}$, the bound on $D_{N,\gamma}^{(22)}$ can be obtained similarly.
Since $d_{CF,p}$ is an element of $\mathcal{D}_1$, there is
\begin{align*}
\sup_{\theta \in \Theta} D_{N, \gamma}^{(21)}(\theta) & \leq \frac{1}{N} \sum_{k=0}^{N-1} \sup_{\theta \in \Theta} | X_{k \gamma}^{\theta} - X_{k \gamma}^{\theta,\gamma}(B) | \\
& \leq C_q \sum_{i=0}^{q} \frac{1}{N} \sum_{k=0}^{N-1} \sup_{\theta \in \Theta} |Y_{k \gamma+ih}^{\theta} -Y_{k \gamma+ih}^{\theta,\gamma}(B) |.
\end{align*}
Recall that $\alpha \ge 2$ in \ref{asmp:strong_identif}. Hence, an application of Jensen's inequality gives
\begin{align*}
\mathbb{E} \sup_{\theta \in \Theta} D_{N, \gamma}^{(21)}(\theta)^\alpha  \leq C_{q,\alpha} \sum_{i=0}^{q} \mathbb{E} \Biggl[ \sup_{\theta \in \Theta} \frac{1}{N}  \sum_{k=0}^{N-1}  |  Y_{k \gamma + ih}^{\theta}-Y_{k \gamma + ih}^{\theta,\gamma}(B) |^{\alpha} \Biggr] 
\end{align*}
Define 
$$ \mathcal{I} := \sum_{i=0}^{q}  \sup_{\theta \in \Theta} \frac{1}{N}  \sum_{k=0}^{N-1}  |  Y_{k \gamma + ih}^{\theta}-Y_{k \gamma + ih}^{\theta,\gamma}(B) |^{2}  .$$
We will first provide a bound on $\mathcal{I}$. Using \eqref{eq:bound-3-interpol},
\begin{align}
& \sum_{i=0}^q  |Y_{k \gamma+ih}^\theta - Y_{k \gamma+ih}^{\theta,\gamma}(B)|^2  \nonumber \\ 
& \quad \leq C  \sum_{i=0}^q \Bigg( \gamma \int_{(k \gamma + ih)_{\gamma}}^{(k \gamma + ih)_{\gamma}+\varepsilon_{k,i}} (1+| Y_s^\theta|^{2 r} ) ds +| B_{(k \gamma + ih)_{\gamma}+\varepsilon_{k,i}}-B_{(k \gamma + ih)_{\gamma}} |^2 \Bigg) \nonumber \\ 
& \quad\quad + C \sum_{i=0}^q \Bigg( \sum_{j=0}^{\lfloor \frac{ k \gamma + ih}{\gamma} \rfloor -1} \Big( \gamma^{2} (1+ |Y_{j \gamma}^{\theta,\gamma}(B)|^{2r})  + \gamma^{-1} \int_0^\gamma |B_{j \gamma + t} - B_{j \gamma} |^{2} d t \Big) \gamma e^{-\rho(\lfloor \frac{ k \gamma + ih}{\gamma} \rfloor -j+1)} \Bigg)  \nonumber \\
&\quad =: C \sum_{i=0}^q  \left( \mathcal{I}_{1,k}(i) + \mathcal{I}_{2,k}(i)+\mathcal{I}_{3,k}(i)+\mathcal{I}_{4,k}(i)\right).
 \label{eq:defdesI}
\end{align}
Hence there is $\displaystyle \mathcal{I} \leq C_q \sup_{\theta \in \Theta} \sum_{i = 0}^{q} \frac{1}{N} \sum_{k=0}^{N-1} \left( \mathcal{I}_{1,k}(i) + \mathcal{I}_{2,k}(i)+\mathcal{I}_{3,k}(i)+\mathcal{I}_{4,k}(i) \right)$.
 Let us provide uniform bounds in $\theta$ on the sum over $k$ of the terms $\mathcal{I}_{1,k}(i),\,  \mathcal{I}_{2,k}(i),\, \mathcal{I}_{3,k}(i),\, \mathcal{I}_{4,k}(i)$. First we have
\begin{align}
\frac{1}{N} \sum_{k=0}^{N-1} \mathcal{I}_{1,k}(i) &  \leq \frac{1}{N} \int_{(ih)_\gamma}^{(N \gamma+ih)_\gamma} \gamma (1+| Y_s^\theta|^{2r} ) ds \nonumber \\ 
& \leq    \frac{\gamma}{N} \int_{ih}^{N \gamma+ih} (1+| Y_s^\theta|^{2r} ) ds + \frac{\gamma}{N} \int_{(ih)_\gamma}^{ih} (1+| Y_s^\theta|^{2r} ) ds  \nonumber \\
& \leq  \frac{\gamma^2}{N\gamma} \int_{0}^{N \gamma} (1+| Y_{s+ih}^\theta|^{2r} ) ds + 2 \gamma^2 \mathds{1}_{i \neq 0} \left(\frac{1}{ih} \int_{0}^{ih}  (1+| Y_s^\theta|^{2r} ) ds \right) \label{eq:I2k}.
\end{align}
For $ \mathcal{I}_{3,k}(i) $, write
\begin{align}
& \frac{1}{N} \sum_{k=0}^{N-1} \mathcal{I}_{3,k}(i)   \leq C  \frac{\gamma^2}{N} \sup_{\theta \in \Theta}  \sum_{k=0}^{\lfloor \frac{ N \gamma + ih}{\gamma} \rfloor -1}  (1+|Y_{k \gamma}^{\theta,\gamma}(B)|^{2r} ) \nonumber \\ 
 &\hspace{2.15cm} \leq  C \frac{\gamma}{N} \sup_{\theta \in \Theta} \int_0^{(N \gamma +ih)_\gamma - \gamma} (1+ |Y_{t_{\gamma}}^{\theta,\gamma}(B)|^{2r}) dt \nonumber \\
 &  \leq  C \frac{\gamma}{N} \sup_{\theta \in \Theta} \int_0^{N \gamma - \gamma} (1+ |Y_{t_{\gamma + ih}}^{\theta,\gamma}(B)|^{2r}) dt  + C  \mathds{1}_{i \neq 0} \frac{\gamma}{N}  \sup_{\theta \in \Theta}  \int_0^{ih} (1+ |Y_{t_{\gamma}}^{\theta,\gamma}(B)|^{2r}) dt \nonumber \\
 & \leq C \frac{\gamma^2}{N\gamma} \sup_{\theta \in \Theta} \int_0^{N \gamma } (1+ |Y_{t_{\gamma + ih}}^{\theta,\gamma}(B)|^{2r}) dt + C \frac{\gamma^2}{N \gamma}  \mathds{1}_{i \neq 0} \sup_{\theta \in \Theta}  \frac{1}{ih} \int_0^{ih} (1+ |Y_{t_\gamma}^{\theta,\gamma}(B)|^{2r}) dt . \label{eq:I3k}
\end{align}
For $\mathcal{I}_{4,k}(i)$ we have
\begin{align}
& \frac{1}{N} \sum_{k=0}^{N-1} \mathcal{I}_{4,k}(i)  \leq \frac{1}{N} \sum_{k=0}^{\lfloor \frac{ N \gamma + ih}{\gamma} \rfloor -1} \gamma^{-1} \int_0^\gamma | B_{k \gamma + t} - B_{k \gamma} |^{2} d t \nonumber  \\
&\hspace{2.15cm} \leq \frac{\gamma^{-2} }{N} \int_0^{N \gamma + ih - \gamma}  \left( \int_0^\gamma  | B_{s_\gamma + t} - B_{s_\gamma} |^{2} dt \right) ds \nonumber \\ 
& \leq \frac{1 }{N \gamma} \int_0^{ih} \left( \gamma^{-1}  \int_0^\gamma  | B_{s_\gamma + t} - B_{s_\gamma} |^{2} dt \right) ds +  \frac{1}{N\gamma} \int_0^{N \gamma} \left( \gamma^{-1}  \int_0^\gamma  | B_{s_\gamma +ih+ t} - B_{s_\gamma+ih} |^{2} dt \right) ds .\label{eq:I4k}
\end{align}
Therefore, using \eqref{eq:I2k}, \eqref{eq:I3k}, \eqref{eq:I4k} in \eqref{eq:defdesI}, it comes
\begin{align*}
\mathcal{I} & \leq C  \sum_{i=0}^q \Bigg( \sup_{\theta \in \Theta}   \frac{\gamma^2}{N\gamma} \int_{0}^{N \gamma} (1+| Y_{s+ih}^\theta|^{2r} ) ds + 2\mathds{1}_{i \neq 0} \sup_{\theta \in \Theta}   \gamma^2\frac{1}{ih} \int_{0}^{ih}  (1+| Y_s^\theta|^{2r} ) ds   \\ 
& \quad +  C \frac{\gamma^2}{N\gamma} \sup_{\theta \in \Theta} \int_0^{N \gamma } (1+ |Y_{t_{\gamma + ih}}^{\theta,\gamma}(B)|^{2r}) dt  + C \frac{\gamma^2}{N \gamma} \mathds{1}_{i \neq 0} \sup_{\theta \in \Theta}  \frac{1}{ih} \int_0^{ih} (1+ |Y_{t_\gamma}^{\theta,\gamma}(B)|^{2r}) dt \\
& \quad + \frac{1 }{N \gamma} \int_0^{ih} \left( \gamma^{-1}  \int_0^\gamma  \sup_{\theta \in \Theta}   | B_{s_\gamma + t} - B_{s_\gamma} |^{2} dt \right) ds  +  \frac{1}{N\gamma} \int_0^{N \gamma} \left( \gamma^{-1}  \int_0^\gamma  \sup_{\theta \in \Theta}   | B_{s_\gamma +ih+ t} - B_{s_\gamma+ih} |^{2} dt \right) ds \\
& \quad + \frac{1}{N} \sum_{k=0}^{N-1} \sup_{\theta \in \Theta}  | B_{(k \gamma + ih)_{\gamma}+\varepsilon_{k,i}}-B_{(k \gamma + ih)_{\gamma}} |^2 \Bigg)  .
\end{align*}
Since $N \gamma \ge 1$ and $\varepsilon_{k,i} < \gamma$, using Jensen's inequality ($\alpha/2 \ge 1$) and taking the expectation, we get by applying \cite[Proposition 3.5]{HRarxiv} that for $\varepsilon \in (0,\alpha\min(\mathcal{H}))$,
\begin{align*}
\mathbb{E} [ \mathcal{I}^{\alpha/2} ] & \leq C_q \sum_{i=0}^q \Bigg( \mathbb{E} \left[ \sup_{\theta \in \Theta} \frac{1}{N\gamma} \int_{0}^{N \gamma} (1+| Y_{s+ih}^\theta|^{2r} ) ds \right]^{\alpha/2} \\ & \quad + 2 \gamma^2 \mathds{1}_{i \neq 0} \mathbb{E} \left[\sup_{\theta \in \Theta} \frac{1}{ih} \int_{0}^{ih}  (1+| Y_s^\theta|^{2r} ) ds  \right]^{\alpha/2}  \\ 
& \quad +  C \gamma^2 \mathbb{E} \left[ \frac{1}{N\gamma} \sup_{\theta \in \Theta} \int_0^{N \gamma } (1+ |Y_{t_{\gamma + ih}}^{\theta,\gamma}(B)|^{2r}) dt \right]^{\alpha/2} \\ & \quad + C \gamma^2 \mathds{1}_{i \neq 0} \mathbb{E} \left[  \sup_{\theta \in \Theta}  \frac{1}{ih} \int_0^{ih} (1+ |Y_{t_\gamma}^{\theta,\gamma}(B)|^{2r}) dt \right]^{\alpha/2} \\
& \quad + \gamma^{\alpha \min(\mathcal{H})-\varepsilon} +  \frac{1}{N\gamma} \int_0^{N \gamma} \gamma^{\alpha \min(\mathcal{H})-\varepsilon} \, ds +  \frac{1}{N} \sum_{k=0}^{N-1} \gamma^{\alpha \min(\mathcal{H})-\varepsilon} \Bigg).
\end{align*}
It follows that
\begin{align*}
\mathbb{E} [ \mathcal{I}^{\alpha/2} ] &  \leq C_q \sum_{i=0}^q \Bigg( \mathbb{E} \left[ \sup_{\theta \in \Theta} \frac{1}{N\gamma} \int_{0}^{N \gamma} | Y_{s+ih}^\theta|^{2r}  ds \right]^{\alpha/2} \\ & \quad + 2 \gamma^2 \mathds{1}_{i \neq 0} \mathbb{E} \left[\sup_{\theta \in \Theta} \frac{1}{ih} \int_{0}^{ih}  | Y_s^\theta|^{2r}  ds  \right]^{\alpha/2}  \\ 
& \quad +  C \gamma^2 \mathbb{E} \left[ \frac{1}{N\gamma} \sup_{\theta \in \Theta} \int_0^{N \gamma } |Y_{t_{\gamma + ih}}^{\theta,\gamma}(B)|^{2r} dt \right]^{\alpha/2} \\ & \quad + C \gamma^2 \mathds{1}_{i \neq 0} \mathbb{E} \left[  \sup_{\theta \in \Theta}  \frac{1}{ih} \int_0^{ih}  |Y_{t_\gamma}^{\theta,\gamma}(B)|^{2r} dt \right]^{\alpha/2} + \gamma^\alpha +  \gamma^{\alpha \min(\mathcal{H})-\varepsilon} \Bigg).
\end{align*} 
By Proposition~\ref{prop:uniform-Lq-bounds}$(iii)$, Proposition~\ref{prop:uniform-Lq-bounds-L}$(ii)$ and since $r \leq 1$, we have that
\begin{align*}
\mathbb{E} \left[ \sup_{\theta \in \Theta} \frac{1}{N\gamma} \int_{0}^{N \gamma} | Y_{s}^\theta|^{2r}  ds \right]^{\alpha/2} ~\text{ and } \quad
\mathbb{E} \left[ \frac{1}{N\gamma} \sup_{\theta \in \Theta} \int_0^{N \gamma } |Y_{t_{\gamma }}^{\theta,\gamma}(B)|^{2r} dt \right]^{\alpha/2} 
\end{align*}
are bounded uniformly in $N$ and $\gamma$. One can check that the result still holds when the process is shifted by $ih$ since the shifted process is still solution of an SDE that satisfies the necessary assumptions. 
Therefore, for $\varepsilon \in (0,\alpha\max(\mathcal{H}))$,
\begin{align*}
\mathbb{E} [ \mathcal{I}^{\alpha/2} ] & \leq C \gamma^{\alpha \min(\mathcal{H})-\varepsilon}  .
\end{align*}
We conclude by observing that by Jensen's inequality
\begin{align*}
 \mathcal{I}   & \ge  \sum_{i=0}^{q} \left( \sup_{\theta \in \Theta} \frac{1}{N}  \sum_{k=0}^{N-1}  |  Y_{k \gamma + ih}^{\theta}-Y_{k \gamma + ih}^{\theta,\gamma}(B) | \right)^2  \\
& \ge C \left( \sum_{i=0}^{q}  \sup_{\theta \in \Theta} \frac{1}{N}  \sum_{k=0}^{N-1}  |  Y_{k \gamma + ih}^{\theta}-Y_{k \gamma + ih}^{\theta,\gamma}(B) |  \right)^2 .
\end{align*}
Hence
\begin{align}\label{eq:proof-bound-D2}
\mathbb{E}  \sup_{\theta \in \Theta} D_{N, \gamma}^{(21)}(\theta)^\alpha  \leq C \mathbb{E} [ \mathcal{I}^{\alpha/2}]  \leq C \gamma^{\alpha \min(\mathcal{H})-\varepsilon} .
\end{align}

\vspace{0.1cm}

Consider now $D_{N, \gamma}^{(3)}(\theta)$, which was defined in \eqref{eq:decom-D1-2-3}. 
Since $\mu_\theta$ is also the stationary law of the process $X^{\theta}(\hat{B})$, we drop the dependence on $\hat{B}$ for the rest of the proof. We first start with the following decomposition:
\begin{align*}
\sup_{\theta \in \Theta} D_{N, \gamma}^{(3)}(\theta) \leq \sup_{\theta \in \Theta} D_{N, \gamma}^{(31)}(\theta) + \sup_{\theta \in \Theta} D_{N, \gamma}^{(32)}(\theta)  ,
\end{align*}
where, noticing that $\frac{1}{N} \sum_{k=0}^{N-1} \delta_{X_{k \gamma}^{\theta}} = \frac{1}{T} \int_{0}^T \delta_{X_{t_{\gamma}}^\theta} dt$ for $T=N\gamma$,
\begin{align*}
D_{N, \gamma}^{(31)}(\theta) & = d\left(\mu_{\theta}, \frac{1}{T} \int_{0}^T \delta_{X_t^\theta} dt \right)\\
D_{N, \gamma}^{(32)}(\theta) & = d\left(\frac{1}{T} \int_{0}^T \delta_{X_t^\theta} dt, \frac{1}{T} \int_{0}^T \delta_{X_{t_{\gamma}}^\theta} dt\right) .
\end{align*}

\paragraph{Bound on $\sup_{\theta \in \Theta} D_{N, \gamma}^{(32)}(\theta)$.} Similar arguments as before lead to
\begin{align*}
\mathbb{E}  \sup_{\theta \in \Theta} D_{N, \gamma}^{(32)}(\theta)^\alpha \leq C_{q,\alpha} \sum_{i=0}^{q} \frac{1}{T} \int_0^T \sup_{\theta \in \Theta}   | Y_{t + ih}^{\theta}-Y_{t_{\gamma} + ih}^{\theta} |^\alpha\, dt  .
\end{align*}
We will show how to bound the quantity above for $i=0$. The same arguments can be used for any value of $i$. Since $Y^{\theta}$ is a solution of \eqref{eq:fsde}, 
it follows from using Jensen's inequality and integrating over $t$ that
\begin{align*}
\frac{1}{T} \int_0^T | Y_t^{\theta} - Y_{t_\gamma}^{\theta} |^\alpha\, dt 
\leq 2^{\alpha-1} \gamma^{\alpha-1} \frac{1}{T} \int_0^T \int_{t_\gamma}^t |b_\xi(Y_s^\theta)|^\alpha \, ds\, dt + 2^{\alpha-1} | \sigma |^\alpha \frac{1}{T} \int_0^T | B_t - B_{t_\gamma} |^\alpha  dt .
\end{align*}
By Fubini's theorem, we get that
\begin{align*}
\frac{1}{T} \int_0^T | Y_t^{\theta} - Y_{t_\gamma}^{\theta} |^\alpha \, dt
\leq 2^{\alpha-1} \gamma^{\alpha} \frac{1}{T} \int_0^T  |b_\xi(Y_s^\theta) |^\alpha d s +2^{\alpha-1} | \sigma |^\alpha \frac{1}{T} \int_0^T | B_t - B_{t_\gamma} |^\alpha  d t .
\end{align*}
The drift term above is bounded thanks to the sublinear growth of $b_\xi$ given by \eqref{eq:drift-growth} and the uniform bounds on the $L^q$ moments of $Y_t^{\theta}$ given in Proposition \ref{prop:uniform-Lq-bounds} (ii). As for the term $|B_t - B_{t_\gamma} |$, we have thanks to \cite[Proposition 3.5]{HRarxiv} that for all $\varepsilon > 0$,
\begin{align*}
\mathbb{E} \left( \sup_{H \in \mathcal{H}} | B_t - B_{t_\gamma} |^\alpha \right) \leq C \gamma^{ \alpha \min(\mathcal{H}) -\varepsilon} .
\end{align*}
Hence, the two previous inequalities yield
\begin{align*}
\mathbb{E} \left( \sup_{\theta \in \Theta} \frac{1}{T} \int_0^T | Y_t^{\theta} - Y_{t_\gamma}^{\theta} |^{\alpha}  dt \right) \leq C \gamma^{ \alpha \min(\mathcal{H})-\varepsilon} .
\end{align*}

\paragraph{Bound on $D_{N, \gamma}^{(31)}(\theta)$.} The quantity $D_{N, \gamma}^{(31)}(\theta)$ can be handled the same way as $D_n^{(1)}$ in the proof of Lemma \ref{lem:bound-D_n^1}. Namely, we get that
\begin{align}\label{eq:D_n31-simple}
\mathbb{E} D_{N, \gamma}^{(31)}(\theta)^\alpha \leq C_\alpha \left(T^{-\alpha} + T^{-\frac{\alpha}{2}(2-\max(2\max(\mathcal{H}),1) )} \right) .
\end{align}

\paragraph{Bound on $ \sup_{\theta \in \Theta} D_{N, \gamma}^{(31)}(\theta)$.}
Let $\varphi(\theta) = d(\mu_{\theta}, \frac{1}{T} \int_0^T \delta_{X_t^\theta} dt )$.  Let $\varepsilon>0$ and $\Theta^{(\varepsilon)}:=\left\{\theta_i^{(\varepsilon)} \mid 1 \leq i \leq M_{\varepsilon}\right\}$ such that $\Theta \subset \bigcup_{i=1}^{M_{\varepsilon}} B\left(\theta_i^{(\varepsilon)}, \varepsilon\right)$ for some points $\theta_{i}^{(\varepsilon)}$ in $\Theta$. Then, as in \cite[Eq. (5.27)-(5.28)]{panloup2020general} and using \eqref{eq:D_n31-simple}, one gets
\begin{align*}
\mathbb{E} \sup_{\theta \in \Theta} \varphi (\theta)^\alpha \leq C_\alpha \left( \mathbb{E} \sup_{\theta \in \Theta} | \varphi(\theta) - \varphi(\theta_\varepsilon) |^\alpha + M_\varepsilon \Bigl( T^{-\alpha} + T^{-\frac{\alpha}{2}(2-\max(2\max(\mathcal{H}),1)} \Bigr)\right),
\end{align*}
where $\theta_{\varepsilon}:=\underset{\theta^{\prime} \in \{\theta_i^{(\varepsilon)}\}}{\operatorname{argmin}}\left|\theta^{\prime}-\theta\right|$.
Now ${| \varphi(\theta) - \varphi(\theta_\varepsilon) | \leq d(\mu_{\theta},\mu_{\theta_\varepsilon}) + d\left(\frac{1}{T} \int_0^T \delta_{X_t^\theta} dt, \frac{1}{T} \int_0^T \delta_{X_t^{\theta_\varepsilon}} dt\right)}$. 
Since $d$ belongs to $\mathcal{D}_2$, the second term in the right-hand side yields
\begin{align*}
d\left(\frac{1}{T} \int_0^T \delta_{X_t^\theta} dt, \frac{1}{T} \int_0^T \delta_{X_t^{\theta_\varepsilon}} dt\right) & \leq C_q \sum_{i=0}^{q}  \frac{1}{T} \int_0^T |Y_{t+ih}^{\theta}-Y_{t+ih}^{\theta_\varepsilon} |^2 dt .
\end{align*}
For $\varpi \in (0, 1)$, Proposition \ref{prop:regularity-result} gives the existence of a random variable $\mathbf{C}$ with finite moments such that
\begin{align*}
 \frac{1}{T} \int_0^T |Y_{t}^{\theta}-Y_{t}^{\theta_\varepsilon} |^2 dt \leq \mathbf{C} | \theta-\theta_\varepsilon |^{\frac{\varpi}{2}} .
\end{align*}
This bound still holds if $Y_t$ is replaced by $Y_{t+ih}$ since 
\begin{align*}
\frac{1}{T} \int_0^T |Y_{t+ih}^{\theta}-Y_{t+ih}^{\theta_\varepsilon} |^2 dt \leq \frac{T+ih}{T}\frac{1}{T+ih} \int_0^{T+ih} |Y_{t}^{\theta}-Y_{t}^{\theta_\varepsilon} |^2 dt .
\end{align*}
Overall, we get
\begin{align*}
 d\left(\frac{1}{T} \int_0^T \delta_{X_t^\theta} dt, \frac{1}{T} \int_0^T \delta_{X_t^{\theta_\varepsilon}} dt\right)  \leq C_{q}\, \mathbf{C} | \theta - \theta_\varepsilon |^{\frac{\varpi}{2}} .
\end{align*}
Hence we have obtained that
\begin{align*}
\mathbb{E} \sup_{\theta \in \Theta} \varphi (\theta)^\alpha \leq C_{\alpha,q, \varpi} \left(  \varepsilon^{\frac{\alpha \varpi}{2}}+ M_\varepsilon \Bigl( T^{-\alpha} + T^{-\frac{\alpha}{2}(2-\max(2\max(\mathcal{H}),1))} \Bigr)\right) .
\end{align*}
Choosing $M_\varepsilon \leq \frac{C}{\varepsilon^d}$ and $\varepsilon=T^{-\chi}$ for some $\chi>0$, it comes that
\begin{align*}
\mathbb{E} \sup_{\theta \in \Theta} \varphi (\theta)^\alpha \leq C_{\alpha,q,\varpi} \left(  T^{- \chi \alpha  \frac{\varpi}{2}}+ T^{-\frac{\alpha}{2}(2-\max(2\max(\mathcal{H}),1)) + \chi d}\right) .
\end{align*}
Finally optimize over $\chi$ to get $\mathbb{E} \sup_{\theta \in \Theta} \varphi (\theta)^\alpha \leq C_{\alpha,q,\varpi}  T^{-\bar{\eta}}$, 
for $\bar{\eta}=\frac{\varpi \alpha^{2}}{2(\alpha \varpi+2d)}(2-(2 \max(\mathcal{H}) \vee 1))$.
\end{proof}

\section{Application to fractional Ornstein-Uhlenbeck processes}\label{sec:appfOU}

In this section, we prove the results of Section \ref{sec:app-OU} and provide numerical experiments to illustrate the convergence of the estimators in the case of fractional Ornstein-Uhlenbeck processes. We first prove the identifiability assumption for the fractional Ornstein-Uhlenbeck (OU) process in Section \ref{subsec:simplif} and Section \ref{sec:Hknown-or-xiknown}, then for a family of small perturbations of the fractional OU process in Section \ref{sec:OU-perturbation} and finally, in Section \ref{sec:numerical}, we provide numerical results.

\subsection{Identifiability assumption: proof of Proposition \ref{prop:identifOU}}\label{subsec:simplif}

The proof of Proposition \ref{prop:identifOU} is based on the injectivity of a specific function, as stated in the following Lemma (the proof is given in Section \ref{sec:Hknown-or-xiknown}).
\begin{lemma}\label{lem:injectivity}
 Assume one of the three cases $\theta= (\xi,H)$, $\theta=(\xi,\sigma)$ or $\theta=(\sigma,H)$, then there exists $h_0 > 0$ such that for $h \in (0, h_0)$, the function $f$ defined by
\begin{align}\label{assump_o-u_two}
   f: \theta \mapsto \begin{pmatrix}
  \sigma^2  H \Gamma(2H) \xi^{-2H} \\
   \sigma^2  \Gamma(2H+1)\frac{\sin(\pi H)}{\pi} \int_0^{\infty} \cos(h x) \frac{x^{1-2H}}{\xi^2 + x^2} dx\\
   \end{pmatrix} 
   \end{align}
is one-to-one.
\end{lemma}

For the fractional OU process, recall that the stationary measure follows the Gaussian distribution given by \eqref{eq:fOUinvariant}.
Furthermore, the processes $\bar{U}_{.+ih}^\theta$ are also Gaussian with the same law. The correlation between these processes is given by (see \cite[Eq (2.2)]{cheridito2003fractional}):
\begin{align}\label{eq:fOUcov}
    \mathbb{E}(\bar{U}^\theta_t \bar{U}^\theta_{t+ih}) = \sigma^2 \frac{\Gamma(2H+1) \sin(\pi H)}{\pi} \int_{0}^{\infty} \cos(ih x) \frac{x^{1-2H}}{\xi^2 + x^2} dx .
\end{align}
Now for $\theta_1,\theta_2$ in $\Theta$, there is
\begin{align*}
   d_{CF,p}(\mu_{\theta_1},\mu_{\theta_2})^2 
 & = \int_{\R^2} \left( \EE e^{i \langle \chi, (\bar{U}_t^{\theta_1},\bar{U}_{t+h}^{\theta_1}-\bar{U}_t^{\theta_1})\rangle}  - \EE e^{i \langle \chi, (\bar{U}_t^{\theta_2},\bar{U}_{t+h}^{\theta_2}-\bar{U}_t^{\theta_2})\rangle}\right)^2 g_p(\chi) d \chi .
\end{align*}
Since the process $(\bar{U}_{\cdot}^\theta, \bar{U}^\theta_{\cdot+h}-\bar{U}^\theta_{\cdot})$ is Gaussian and stationary, it comes:
\begin{align*}
   d_{CF,p}(\mu_{\theta_1},\mu_{\theta_2}) = 0  ~~\text{ iff }~~
   \begin{array}{cc}
       \mathbb{E}(\bar{U}_0^{\theta_1})^2 = \mathbb{E}(\bar{U}_0^{\theta_2})^2 \\
        \mathbb{E}\left(\bar{U}_0^{\theta_1}(\bar{U}_{h}^{\theta_1}-\bar{U}_0^{\theta_1})\right) = \mathbb{E}\left(\bar{U}_0^{\theta_2}(\bar{U}_{h}^{\theta_2}-\bar{U}_0^{\theta_2})\right) %
   \end{array}
\end{align*}
which thus reads
\begin{align*}
   d_{CF,p}(\mu_{\theta_1},\mu_{\theta_2}) = 0 ~\text{iff}~
   \begin{array}{cc}
       \mathbb{E}(\bar{U}_0^{\theta_1})^2 = \mathbb{E}(\bar{U}_0^{\theta_2})^2 \\
        \mathbb{E}\left(\bar{U}_0^{\theta_1}\bar{U}_{h}^{\theta_1}\right) = \mathbb{E}\left(\bar{U}_0^{\theta_2}\bar{U}_{h}^{\theta_2}\right) %
   \end{array}
   .
\end{align*}
In view of \eqref{eq:fOUinvariant} and \eqref{eq:fOUcov}, assumption \ref{asmp:identif} becomes equivalent to the injectivity of the function $f$ defined in \eqref{assump_o-u_two}, which is therefore given by Lemma \ref{lem:injectivity}.

\begin{remark}\label{rmk:injectivity-allparameters}
In \cite{haress2020estimation}, the authors studied fractional OU processes and proposed a similar estimator for $(\xi,\sigma,H)$ simultaneously. Similarly to our case, for a consistency argument to hold, they are left to study the injectivity of
\begin{align}\label{eq:f-3D}
f:(\xi,\sigma,H) \mapsto
 \begin{pmatrix}
   \sigma^2  H \Gamma(2H) \xi^{-2H} \\
    \sigma^2 \Gamma(2H+1)\frac{\sin(\pi H)}{\pi} \int_0^{\infty} \cos(h x) \frac{x^{1-2H}}{\xi^2 + x^2} dx\\
   \sigma^2   \Gamma(2H+1)\frac{\sin(\pi H)}{\pi} \int_0^{\infty} \cos(2h x) \frac{x^{1-2H}}{\xi^2 + x^2} dx\\
   \end{pmatrix} .
\end{align}
The injectivity was not proven but numerical arguments were provided to support this claim.
\end{remark}

\subsection{Injectivity of $f$:  proof of Lemma \ref{lem:injectivity}}\label{sec:Hknown-or-xiknown}

\paragraph{The case $\theta=(\sigma,H)$.} Let $(a,b)$ be in the range of $f$. We will show that the equation
\begin{equation}\label{eq:injectivity-eq-for-2}
\begin{split}
a & = \sigma^2 H \Gamma(2H) \xi^{-2H}  \\
b & = \sigma^2 \Gamma(2H+1) \frac{\sin(\pi H)}{\pi} \xi^{-2H} \int_0^{\infty} \cos(\xi hx) \frac{x^{1-2H}}{1+x^2} dx   ,
\end{split}
\end{equation}
has a unique solution in $\sigma,H$. First, thanks to the first equation, notice that we can write $\sigma^2 = \frac{a \xi^{2H}}{H \Gamma(2H)}$. Injecting this in the second equation we get
\begin{align*}
b \pi & = a \sin(\pi H) \int_0^\infty \cos(\xi h x) \frac{x^{1-2H}}{1+x^2}\, dx:= a g(H)  .
\end{align*}
We will show that the function $g$ is injective. Since $g$ is continuously differentiable, it suffices to show that $g'(H) >0$ for all $H \in \mathcal{H}$.
We have
\begin{align}\label{eq:g'(H)}
g'(H) & = \pi \cos(\pi H) \int_0^\infty \cos(\xi h x) \frac{x^{1-2H}}{1+x^2} dx  - 2 \sin(\pi H) \int_0^{\infty} \cos(\xi h x) \log(x) \frac{x^{1-2H}}{1+x^2} dx  .
\end{align}
By \eqref{eq:fOUinvariant} and \eqref{eq:fOUcov}, we have
\begin{align*}%
  \mathbb{E}(\bar{U}^\theta_t \bar{U}^\theta_{t+0}) = \sigma^2 \frac{\Gamma(2H+1) \sin(\pi H)}{\pi} \int_{0}^{\infty}  \frac{x^{1-2H}}{\xi^2 + x^2} dx = \mathbb{E} (\bar{U}^\theta_t)^2  = \sigma^2 H \Gamma(2H) \xi^{-2H} .
\end{align*}
Using $\Gamma(2H+1) = 2H \Gamma(2H)$ and the change of variables $y=x/\xi$  yields
\begin{align*}
 \sin(\pi H) \int_{0}^{\infty}  \frac{x^{1-2H}}{1 + x^2} dx = \frac{\pi}{2}.
\end{align*}
By differentiating with respect to $H$, we get
\begin{align*}
\pi \cos(\pi H) \int_0^\infty \frac{x^{1-2H}}{1+x^2} dx  - 2 \sin(\pi H) \int_0^{\infty} \log(x) \frac{x^{1-2H}}{1+x^2} dx  = 0 .
\end{align*}
Subtracting this term to $g'(H)$ in \eqref{eq:g'(H)}, we get
\begin{align*}
g'(H) =  \int_0^\infty (1- \cos(\xi h x)) (2 \sin(\pi H) \log(x) - \pi \cos(\pi H) ) \frac{x^{1-2H}}{1+x^2} dx   .
\end{align*}
Let $\beta_H = e^{\frac{\pi \cos(\pi H)}{2 \sin(\pi H)}}$. Then
\begin{align*}
g'(H) &  =  \int_0^{\beta_H} (1- \cos(\xi h x)) (2 \sin(\pi H) \log(x) - \pi \cos(\pi H) ) \frac{x^{1-2H}}{1+x^2} dx   \\ & \quad + \int_{\beta_H}^{\infty} (1- \cos(\xi h x)) (2 \sin(\pi H) \log(x) - \pi \cos(\pi H) ) \frac{x^{1-2H}}{1+x^2} dx .
\end{align*}
For $x\in (0,\beta_{H}]$, $2 \sin(\pi H) \log(x) - \pi \cos(\pi H)\leq 0$. In addition, using $1- \cos(x)  \leq \frac{x^2}{2}$ in the first integral and the change of variables $y = hx$ in the second integral, we get
\begin{align*}
g'(H) & \ge \frac{\xi^2 h^2}{2} \int_0^{\beta_H} x^2  (2 \sin(\pi H) \log(x) - \pi \cos(\pi H) ) \frac{x^{1-2H}}{1+x^2} dx  \\ & \quad + h^{2H} \int_{\beta_H h}^{\infty}  (1- \cos(\xi x))  (2 \sin(\pi H ) \log(x) -2 \sin(\pi H) \log(h) \\ & \quad - \pi \cos(\pi H) ) \frac{x^{1-2H }}{h^2+x^2} dx   .
\end{align*}
For $x \ge \beta_H h$, we have $2 \sin(\pi H ) \log(x) -2 \sin(\pi H) \log(h) - \pi \cos(\pi H)  \ge 0$. Assuming $h \leq 1$, it thus follows that
\begin{align*}
g'(H) & \ge \frac{\xi^2 h^2}{2} \int_0^{\beta_H} x^2  (2 \sin(\pi H) \log(x) - \pi \cos(\pi H) ) \frac{x^{1-2H}}{1+x^2} dx  \\ & \quad + h^{2H} \int_{\beta_H}^{\infty}  (1- \cos(\xi x))  (2 \sin(\pi H ) \log(x) -2 \sin(\pi H) \log(h) \\ & \quad - \pi \cos(\pi H) ) \frac{x^{1-2H }}{h^2+x^2} dx  \\
& \ge \frac{\xi^2 h^2}{2} \int_0^{\beta_H} x^2  (2 \sin(\pi H) \log(x) - \pi \cos(\pi H) ) \frac{x^{1-2H}}{1+x^2} dx  \\ & \quad + h^{2H} \int_{\beta_H}^{\infty}  (1- \cos(\xi x))  (2 \sin(\pi H ) \log(x) - \pi \cos(\pi H) ) \frac{x^{1-2H}}{1+x^2} dx   \\
& \quad + 2 h^{2H} | \log(h) | \sin(\pi H) \int_{\beta_H}^{\infty} (1-\cos(\xi x) ) \frac{x^{1-2H}}{h^2+x^2} dx .
\end{align*}
Since $H \in \mathcal{H}$ and $\xi \in \Xi$, we deduce that there exists $C_1, C_2,C_3 >0$ such that 
\begin{align*}
g'(H) & \ge C_1 h^{2H} | \log(h) | + C_2 h^{2H} - C_3 h^2 .
\end{align*}
Therefore, there exists $C >0$ and $h_0 > 0$ such that for $h \in (0, h_0)$, we have 
\begin{align}\label{eq:g'(H)>h2H}
g'(H) \ge C h^{2H} |\log(h) | > 0.
\end{align}
We have thus proved that $f$ is one-to-one. 

\paragraph{The case $\theta=(\xi,H)$.}Let $(a,b)$ be in the range of $f$. We prove that the following equation has a unique solution in $(\xi, H)$:
\begin{align*}
   a &=  H \Gamma(2H) \xi^{-2H}  \\
   b & =  2H \Gamma(2H) \xi^{-2H} \frac{\sin(\pi H)}{\pi} \int_0^{\infty} \cos( \xi h x) \frac{x^{1-2H}}{1 + x^2} d x ,
\end{align*}
 which is equivalent to solving
 \begin{align*}
     \xi & = \left(\frac{a}{H \Gamma(2H)}\right)^{-\frac{1}{2H}} \\
  b  & =   2 a \frac{\sin(\pi H)}{\pi} \int_0^{\infty} \cos\left(\left(\frac{a}{H \Gamma(2H)}\right)^{-\frac{1}{2H}} h x\right) \frac{x^{1-2H}}{1 + x^2} dx  .
\end{align*}
For the rest of this section, we will focus on the function 
\begin{align*}
    g_a(H) & = \sin(\pi H) \int_0^{\infty} \cos\left(\left(\frac{a}{H \Gamma(2H)}\right)^{-\frac{1}{2H}} h x\right) \frac{x^{1-2H}}{1 + x^2} dx .
\end{align*}
We will show that for all possible values of $a$, $g_a$ is a bijection and therefore there exists a unique $H$ such that $g_a(H) = \frac{\pi b}{2a}$. For this $H$, $\xi$ is then uniquely determined by the equality $\xi = (\frac{a}{H \Gamma(2H)})^{-\frac{1}{2H}}$.
 
We plan to differentiate $g_a$. For $H > 1/2$, the derivative in the $H$ variable of the function $x\mapsto \cos((\frac{a}{H \Gamma(2H)})^{-\frac{1}{2H}} h x) \frac{x^{1-2H}}{1 + x^2}$ is integrable and we get
\begin{align}\label{eq:lower-bound-g}
\begin{split}
g_a'(H)& =  \pi \cos(\pi H) \int_0^{\infty} \cos\left(\left(\frac{a}{H \Gamma(2H)}\right)^{-\frac{1}{2H}} h x\right) \frac{x^{1-2H}}{1 + x^2} dx \\ 
& \quad  - \sin(\pi H) \left[H \mapsto (\frac{a}{H \Gamma(2H)})^{-\frac{1}{2H}} \right]'(H) \ h \int_0^{\infty} x \sin\left(\left(\frac{a}{H \Gamma(2H)}\right)^{-\frac{1}{2H}} h x\right) \frac{x^{1-2H}}{1 + x^2} dx \\
& \quad - 2 \sin(\pi H) \int_0^{\infty} \cos\left(\left(\frac{a}{H \Gamma(2H)}\right)^{-\frac{1}{2H}} h x\right) \log(x) \frac{x^{1-2H}}{1 + x^2} dx \\
&=:  g'_{a,1}(H) + g'_{a,2}(H) + g'_{a,3}(H) .
\end{split}
\end{align}
Unfortunately when $H\leq 1/2$, the integral that appears in $g'_{a,2}(H)$ is not defined in Lebesgue's sense.
However, we have for any $A>1$ that
\begin{align*}
\int_{1}^A \sin(Cx) \frac{x^{2-2H}}{1+x^2} dx = \int_{1}^A \sin(Cx) x^{-2H} dx - \int_{1}^A \sin(Cx) \frac{x^{-2H}}{(1+x^2)} dx. 
\end{align*}
The first integral in the right-hand side converges in Riemann's sense as $A\to\infty$ and the second one converges as a classical Lebesgue's integral. Thus we get that $g'_{a,1}(H) + g'_{a,2}(H) + g'_{a,3}(H)$ is well-defined even for $H\in(0,1/2]$, and then that the equality \eqref{eq:lower-bound-g} also holds for $H\leq1/2$.

Now notice that $g_{a,1}'(H)+g_{a,3}'(H)$ is exactly the term $g'(H)$ handled in the previous case $\theta= (\sigma,H)$ with $\xi \equiv (\frac{a}{H \Gamma(2H)})^{-\frac{1}{2H}}$. We have shown in \eqref{eq:g'(H)>h2H} that there exists $C_1> 0$ and $h_0>0$ such that for $h \in (0, h_0)$,
\begin{align}\label{eq:ga1+ga3}
g_{a,1}'(H)+g_{a,3}'(H) > C_1 h^{2H} | \log(h) | .
\end{align}
We now prove an upper bound on the absolute value of $g_{a,2}'(H)$. Using the change of variable $y = h x$, we have
\begin{align}\label{eq:ga2}
& |g_{a,2}'(H) | \\ & \leq \sin(\pi H) \left| \left[H \mapsto (\frac{a}{H \Gamma(2H)})^{-\frac{1}{2H}} \right]'(H) \right|  h^{2H} \left| \int_0^{\infty}  \sin\left(\left(\frac{a}{H \Gamma(2H)}\right)^{-\frac{1}{2H}} x\right) \frac{x^{2-2H}}{h^2 + x^2} dx \right| .
\end{align}
Let us show that the integral $\mathcal{J} = | \int_0^{\infty}  \sin((\frac{a}{H \Gamma(2H)})^{-\frac{1}{2H}} x) \frac{x^{2-2H}}{h^2 + x^2} dx | $ is bounded uniformly in $h \in (0,h_0)$. Using that $| \sin(x) | \leq x$ for $x\geq 0$ we have for $\alpha_{H} = 2\pi (\frac{a}{H \Gamma(2H)})^{\frac{1}{2H}}$ that
\begin{align*}
\mathcal{J}  & \leq 
 \int_0^{\alpha_{H}}  \left(\frac{a}{H \Gamma(2H)}\right)^{-\frac{1}{2H}}  \frac{x^{3-2H}}{h^2 + x^2} dx  + \left| \int_{\alpha_{H}}^{\infty} \sin\left(\left(\frac{a}{H \Gamma(2H)}\right)^{-\frac{1}{2H}} x\right)  \frac{x^{2-2H}}{h^2 + x^2} dx \right| \\
& \leq  \int_0^{\alpha_{H}}  \left(\frac{a}{H \Gamma(2H)}\right)^{-\frac{1}{2H}}  x^{1-2H} dx \\ & \quad + \left| \int_{\alpha_{H}}^{\infty} \sin\left(\left(\frac{a}{H \Gamma(2H)}\right)^{-\frac{1}{2H}} x\right)  \left( \frac{x^{2-2H}}{h^2 + x^2} -\frac{x^{2-2H}}{x^2}  \right)  dx \right| \\
& \quad + \left| \int_{\alpha_{H}}^{\infty}   \sin\left(\left(\frac{a}{H \Gamma(2H)}\right)^{-\frac{1}{2H}} x\right)  x^{-2H} dx \right|  .
\end{align*}
Hence bounding the sine function by $1$ in the second integral and using the change of variables $y=(\frac{a}{H \Gamma(2H)})^{-\frac{1}{2H}}x$ in the third, we get
\begin{align*}
\mathcal{J}   & \leq \left(\frac{a}{H \Gamma(2H)}\right)^{-\frac{1}{2H}} \frac{\alpha_{H}^{2-2H}}{2-2H}  + h^2 \left| \int_{\alpha_{H}}^\infty \sin\left(\left(\frac{a}{H \Gamma(2H)}\right)^{-\frac{1}{2H}} x\right) \frac{x^{-2H}}{h^2+x^2} dx \right| \\ 
& \quad+ \left(\frac{a}{H \Gamma(2H)}\right)^{-\frac{2H-1}{2H}}  \left| \int_{2\pi}^\infty \frac{\sin(x)}{x^{2H}} dx \right| \\
& \leq \left(\frac{a}{H \Gamma(2H)}\right)^{-\frac{1}{2H}} \frac{\alpha_{H}^{2-2H}}{2-2H}  + h^2  \int_{\alpha_{H}}^\infty  x^{-2-2H} dx  +  \left(\frac{a}{H \Gamma(2H)}\right)^{\frac{1-2H}{2H}}  \left| \int_{2\pi}^\infty \frac{\sin(x)}{x^{2H}} dx \right|  \\
& \leq  \left(\frac{a}{H \Gamma(2H)}\right)^{-\frac{1}{2H}} \frac{1}{2-2H}  + h^2 \frac{\alpha_{H}^{-1-2H}}{1+2H} +  \left(\frac{a}{H \Gamma(2H)}\right)^{\frac{1-2H}{2H}}   \left| \int_{2\pi}^\infty \frac{\sin(x)}{x^{2H}} dx \right|.
\end{align*}
Writing $ \int_{2\pi}^\infty \frac{\sin(x)}{x^{2H}} dx $ as the sum of positive terms
\begin{align}
\int_{2\pi}^\infty \sin(x) x^{-2H} dx & = \sum_{k=0}^{\infty} \int_{2k\pi}^{(2k+1)\pi} \frac{\sin(x)}{x^{2H}} dx+ \int_{(2k+1) \pi}^{(2k+2)\pi}  \frac{\sin(x)}{x^{2H}} dx \nonumber  \\
& = \sum_{k=0}^{\infty} \int_{2k \pi}^{(2k+1)\pi} \sin(x) ( \frac{1}{x^{2H}}-\frac{1}{(x+\pi)^{2H}} ) dx\label{eq:decomposcint}  \\ 
& \leq \sum_{k=1}^\infty \frac{\pi}{(2k\pi)^{2H}} \left(1-\frac{1}{(1+k^{-1})^{2H}}\right), \nonumber 
\end{align}
we get that the last sum can be bounded uniformly for $H\in\mathcal{H}$. Thus, $\mathcal{J}$ can be bounded uniformly for $H\in\mathcal{H}$ by a constant $C_{2}>0$. From \eqref{eq:ga2}, we thus get
\begin{align*}
 |g_{a,2}'(H) | & \leq C_{2}\, h^{2H} \sin(\pi H) \left| [H \mapsto (\frac{a}{H \Gamma(2H)})^{-\frac{1}{2H}} ]'(H) \right|. 
\end{align*}
Since the mapping $H \mapsto (\frac{a}{H \Gamma(2H)})^{-\frac{1}{2H}}$ is smooth on $(0,+\infty)$ and $\mathcal{H}$ is  a compact subset of $(0,1)$, we deduce that there exists a constant $\tilde{C}_2>0$ such that
\begin{align*}
 |g_{a,2}'(H) |  \leq \tilde{C}_2 h^{2H} . 
\end{align*}
Combining this with \eqref{eq:ga1+ga3}, we conclude that for any $h \in (0,h_0)$,
\begin{align*}
g_{a,1}'(H)+g_{a,2}'(H)+ g_{a,3}'(H)  \ge C_1 h^{2H} | \log(h)| - \tilde{C}_2 h^{2H} .
\end{align*}
Hence, there exists $C, h_1 >0$ such that for any $h \in (0,h_1)$, we have
\begin{align*}
g_{a,1}'(H)+g_{a,2}'(H)+ g_{a,3}'(H)  \ge C h^{2H} | \log(h) |  > 0 .
\end{align*}
This proves that $g_a$ is a bijection. 

\paragraph{The case $\theta=(\xi,\sigma)$.} As before, for $(a,b)$ in the range of $f$, we need to show that 
\eqref{eq:injectivity-eq-for-2} has a unique solution in $\xi, \sigma$, for a given $H$. Notice that \eqref{eq:injectivity-eq-for-2} is equivalent to
\begin{align*}
a & =\sigma^2 H \Gamma(2H) \xi^{-2H}  \\
b & =a \frac{\sin(\pi H)}{\pi} \int_0^{\infty} \cos(\xi h x) \frac{x^{1-2H}}{1+x^2}  dx =: a \frac{\sin(\pi H)}{\pi}  \tilde{g}(\xi)  .
\end{align*}
Thus, it is enough to show that $\tilde{g}'(\xi) <0$ for all $\xi$.  We have
\begin{align*}
\tilde{g}'(\xi) = -h \int_0^{\infty} \sin(\xi h x) \frac{x^{2-2H}}{1+x^2} dx .
\end{align*}
Let $C$ be a constant that may depend only on $\Theta$ and may change from line to line. We decompose $\tilde{g}'(\xi)$ as
\begin{align*}
\tilde{g}'(\xi) & = -h \int_0^{1} \sin(\xi h x) \frac{x^{2-2H}}{1+x^2} dx  -h \int_1^{\infty} \sin(\xi h x) \frac{x^{2-2H}}{1+x^2} dx.
\end{align*}
Using $|\sin(x)| \leq x$ in the first integral, we get
\begin{align*}
\tilde{g}'(\xi)  & \leq C h^2 - h \int_1^{\infty} \sin(\xi h x) \frac{x^{2-2H}}{1+x^2} dx \\
& = C h^2 + h \int_1^{\infty} \sin(\xi h x) \frac{x^{-2H}}{1+x^2} dx - h \int_1^\infty \sin(\xi h x) x^{-2H} dx .
\end{align*}
Since $\xi$ is in a compact, we use in the first integral that $|\sin(\xi h x)|\leq C h x$. As for the second integral, we use the change of variables $y=\xi h x$ to get
\begin{align*}
\tilde{g}'(\xi)  & \leq C h^2 + C h^{2} \int_1^{\infty} \frac{x^{1-2H}}{1+x^2} dx- \xi^{2H-1} h^{2H} \int_{\xi h}^\infty \sin(x) x^{-2H} dx  \\
& \leq C h^2 + \xi^{2H-1} h^{2H} \int_0^{\xi h} \sin(x) x^{-2H} dx - \xi^{2H-1} h^{2H} \int_{0}^\infty \sin(x) x^{-2H} dx .
\end{align*} 
Using the inequality $|\sin(x)| \leq x$ and the fact that $\xi$ is in a compact, we have $|\xi^{2H-1} h^{2H} \int_0^{\xi h} \sin(x) x^{-2H} dx| \leq C h^2$. As for the last term, we write $\int_{0}^\infty \sin(x) x^{-2H} dx =  \int_{0}^{2\pi} \sin(x) x^{-2H} dx + \int_{2 \pi}^{\infty} \sin(x) x^{-2H} dx$. 
The second term is positive by \eqref{eq:decomposcint}, therefore
\begin{align*}
\int_{0}^\infty \sin(x) x^{-2H} dx & \ge \int_{0}^{2\pi} \sin(x) x^{-2H} dx  = \int_0^\pi  \sin(x) x^{-2H} dx  + \int_{\pi}^{2 \pi} \sin(x) x^{-2H} dx \\
& = \int_0^\pi  \sin(x) ( \frac{1}{x^{2H}} - \frac{1}{(x+\pi)^{2H}}) dx > 0.
\end{align*}
Since the last integral is continuous in $H$, it follows that $\inf_{\theta \in \Theta} \xi^{2H-1}\int_{0}^\infty \sin(x) x^{-2H} dx \geq c>0$ and we get
\begin{align*}
\tilde{g}'(\xi)  & \leq  C h^2 - c\,  h^{2H} .
\end{align*}
It follows that there exists $h_0>0$ such that for any $h<h_0$, we have $g'(\xi) < 0$.

\subsection{Strong identifiability assumption: proof of Lemma \ref{lem:strongidentif-O-U} and Lemma \ref{lem:O-U-perturbation}}\label{sec:OU-perturbation}
In this section, we prove Lemma \ref{lem:strongidentif-O-U} and Lemma \ref{lem:O-U-perturbation}, that is the strong identifiability assumption for the fractional Ornstein-Uhlenbeck process and small perturbations of this process.

\begin{proof}[Proof of Lemma \ref{lem:strongidentif-O-U}]
The condition $p \ge 1$ ensures that $d_{CF,p}$ is well-defined in dimension $d=1$. When $\theta= \xi$, this lemma was proved in \cite[Lemma 6.2]{panloup2020general}.

Let us deal with the case $\theta=H$. We have already seen that $\mu_{\theta}=\mathcal{N}\left( 0,\sigma^2 H \Gamma(2H) \xi^{-2H} \right)$. Taking into account the expression of $d_{CF,p}$ in \eqref{eq:d_cf} yields
\begin{align*}
d_{CF,p} ( \mu_{\theta_1}, \mu_{\theta_2})^2 = \int_\mathbb{R} \left( \exp\left(-\frac{\sigma^2 H_1 \Gamma(2H_1)}{2 \xi^{2H_1}} \eta^2\right) -  \exp\left(-\frac{\sigma^2 H_2 \Gamma(2H_2)}{2 \xi^{2H_2}} \eta^2\right) \right)^2 g_p(\eta) d \eta  .
\end{align*}
Let $g(H,\eta)=\exp(-\frac{\sigma^2 H \Gamma(2H)}{2 \xi^{2H}} \eta^2)$, then we have
\begin{align}\label{eq:lowerbound-d}
d_{CF,p} ( \mu_{\theta_1}, \mu_{\theta_2})^2 = \int_\mathbb{R} \left( \int_{H_1}^{H_2} \partial_H g(H, \eta)\, dH \right)^2 g_p(\eta) d \eta  .
\end{align}
We will show that $\partial_H g(H,\eta)$ is bounded away from $0$. We have
\begin{align*}
\partial_H g(H,\eta) =\frac{ \sigma^2 \eta^2 }{2\xi^{2H}} \exp\left(-\frac{\sigma^2 H \Gamma(2H)}{2 \xi^{2H}} \eta^2\right) \left( \Gamma(2H)+2H\Gamma'(2H) - 2H\Gamma(2H) \log(\xi) \right) .
\end{align*} 
Under \eqref{eq:xiawayfrom0}, we have $| \Gamma(2H)+2H\Gamma'(2H) - 2H\Gamma(2H) \log(\xi) | > 0$ for all $H \in \mathcal{H}$. Hence, there exists two positive constants $c,C$ that depend only on $\Theta$ such that, we have 
$$|\partial_H g(H,\eta)| \ge C \eta^2 \exp(-c \eta^2) ~ \text{for all}~ H \in [m_{\mathcal{H}},M_{\mathcal{H}}].$$
Using this in \eqref{eq:lowerbound-d}, it follows that
\begin{align*}
d_{CF,p} ( \mu_{\theta_1}, \mu_{\theta_2})^2 \ge C^2 |H_1-H_2|^2  \int_\mathbb{R} \eta^2 \exp(-2c \eta^2)  g_p(\eta) d \eta  .
\end{align*}

A similar analysis can be done when $\theta= \sigma$. In this case, one needs to show that the derivative of $g(\sigma)=\exp(-\frac{\sigma^2 H \Gamma(2H)}{2 \xi^{2H}} \eta^2)$ is bounded away from $0$. Since there is
\begin{align*}
g'(\sigma) = - \frac{2 \sigma H \Gamma(2H)}{2 \xi^{2H}} \eta^2 \exp\left(-\frac{\sigma^2 H \Gamma(2H)}{2 \xi^{2H}} \eta^2\right) ,
\end{align*}
and all the parameters live in compact sets that do not contain $0$, there exists positive constants $\tilde{C},\tilde{c}$ such that all $\sigma \in [m_\Sigma, M_\Sigma]$, we have $g'(\sigma)< -\tilde{C} \eta^2 \exp(-\tilde{c} \eta^2)$. Hence we can conclude as in the previous case.
\end{proof}
%

\begin{proof}[Proof of Lemma \ref{lem:O-U-perturbation}]
The case $\theta=\xi$ was considered in \cite[Proposition 6.4]{panloup2020general} under the same assumptions. Our proof for $\theta=H$ or $\theta=\sigma$ will be very similar. More specifically, we decompose $d_{CF,p}(\mu_{\theta_1}^\lambda, \mu_{\theta_2}^\lambda)$ as
\begin{align}\label{eq:dcf-decomp-lower-b}
d_{CF,p}(\mu_{\theta_1}^\lambda, \mu_{\theta_2}^\lambda) \ge I_3^{1/2} - \left(  I_2^{1/2} + I_{11}^{1/2} + I_{12}^{1/2} \right) ,
\end{align}
where 
\begin{align*}
I_{1j} & = \int_{\mathbb{R}} \left( \mathbb{E}[\exp(i \eta \bar{U}^{\lambda, \theta_j}_t) ] - \mathbb{E}[\exp(i \eta U^{\lambda, \theta_j}_t) ] \right)^2 g_p(\eta) d\eta, \quad j=1,2, \\
I_2 & =  \int_{\mathbb{R}} \left( \mathbb{E}[\exp(i \eta U^{\lambda, \theta_1}_t) ] - \mathbb{E}[\exp(i \eta U^{0, \theta_1}_t)  - \mathbb{E}[\exp(i \eta U^{\lambda, \theta_2}_t) ] + \mathbb{E}[\exp(i \eta U^{0, \theta_2}_t)   ] \right)^2 g_p(\eta) d\eta, \\
I_3 & = \int_{\mathbb{R}} \left( \mathbb{E}[\exp(i \eta U^{0, \theta_1}_t) ] - \mathbb{E}[\exp(i \eta U^{0, \theta_2}_t) ] \right)^2 g_p(\eta) d\eta  .
\end{align*}
In the above definition of $I_{1j},\, I_{2}$ and $I_{3}$, $t$ is an arbitrary large time to be determined later. Our goal is to bound $I_3$ from below and bound $I_2$ and $I_{1j}$ from above.

\paragraph{Lower bound for $I_3$.} 
We bound $I_3$ from below as follows:
\begin{align*}
I_3 & \ge  \frac{1}{3}\int_{\mathbb{R}} \left( \mathbb{E}[\exp(i \eta \bar{U}^{0, \theta_1}_t) ] - \mathbb{E}[\exp(i \eta \bar{U}^{0, \theta_2}_t) ] \right)^2 g_p(\eta) d\eta \\ 
& \quad - \bigg(  \int_{\mathbb{R}} \left( \mathbb{E}[\exp(i \eta \bar{U}^{0, \theta_1}_t) ] - \mathbb{E}[\exp(i \eta U^{0, \theta_1}_t) ] \right)^2 g_p(\eta) d\eta \\ 
& \quad\quad +  \int_{\mathbb{R}} \left( \mathbb{E}[\exp(i \eta \bar{U}^{0, \theta_2}_t) ] - \mathbb{E}[\exp(i \eta U^{0, \theta_2}_t) ] \right)^2 g_p(\eta) d\eta \bigg) \,
\end{align*}
Now, by Lemma \ref{lem:strongidentif-O-U}, there exists a constant $c_1$ such that the first term is bounded from below by $c_1 | \theta_1-\theta_2 |^2$.  In view of Proposition \ref{prop:general-results}, the other terms are bounded by $C e^{-c t}$. Choosing $t$ large enough, we can thus bound $I_3$ from below by
\begin{align*}
I_3 \ge \frac{c_1}{6} |\theta_1- \theta_2 |^2 .
\end{align*}

\paragraph{Upper bound for $I_{1j}$.}
The term $I_{1j}$ also represents a distance between the solution of \eqref{eq:O-U-perturbation} and its stationary version. Under the assumption that $b_\xi$, $\partial_\xi b_\xi$ and $\partial_y b_\xi$ are bounded and $\lambda$ is small enough, the drift $-\xi . + \lambda b_\xi(.)$ satisfies assumption \ref{asmp:drift}. Theorefore by Proposition~\ref{prop:general-results} we have $I_{1j} \leq C e^{-ct}$. Setting $t$ large enough we get that
\begin{align*}
I_{1j} \leq \frac{c_1}{16} | \theta_1 - \theta_2 |^2 .
\end{align*}

\paragraph{Upper bound for $I_2$.} It was shown in \cite[Equation (6.17)]{panloup2020general} under $p>3/2$ that
\begin{align*}
I_2 & \leq C \mathbb{E} \Biggl[ \left( |U^{0, \theta_2}_t - U^{0, \theta_1}_t | + | \Delta_R (U_t) | \right)   \left(  |U^{\lambda, \theta_2}_t - U^{0, \theta_2}_t | + | \Delta_R (U_t) |  \right) + | \Delta_R(U_t) | \Biggr] ,
\end{align*}
where $\Delta_R(U_t)$ are the rectangular increments defined by
\begin{align*}
\Delta_R(Y_t) = U_t^{\lambda, \theta_1} - U_t^{0, \theta_1} - U_t^{\lambda, \theta_2} + U_t^{0, \theta_2} .
\end{align*}
But $\Delta_R(U_t)=0$ when $\theta= H$ or $\theta=\sigma$. So 
$I_2 \leq \lambda^2 C \mathbb{E}( | U_t^{0, \theta_2}-U_t^{0, \theta_1} |^2 \| \partial_\lambda U^{\lambda, \theta} \|^2_{\infty})$.
It was also proved in \cite[equation (6.18) and thereafter]{panloup2020general} that when $b_\xi$ and $\partial_y b_\xi$ are both bounded and $\lambda \leq m_{\Xi}(1-\epsilon)$, there is $\| \partial_\lambda U^{\lambda, \theta} \|^2_{\infty} \leq c_{m_{\Xi},M_{\Xi},\epsilon}$. Hence we deduce that
\begin{align*}
I_2 \leq C_{m_{\Xi},M_{\Xi},\epsilon} \lambda^2 \mathbb{E} | U_t^{0, \theta_2}-U_t^{0, \theta_1} |^2 .
\end{align*}
Now if $\theta=H$, we get from the same computation as from the stationary case (see \cite[Lemma A.1]{HRarxiv}) that
\begin{align*}
\mathbb{E} \left| U_t^{0, \theta_2}-U_t^{0, \theta_1} \right|^2  \leq C |\theta_2-\theta_1 |^2 ,
\end{align*}
where $C$  does not depend on $t$. When $\theta=\sigma$, 
\begin{align*}
\mathbb{E}\left| U_t^{0, \theta_2}-U_t^{0, \theta_1} \right|^2   & = \mathbb{E} \left( \int_0^t (\theta_2-\theta_1) e^{-t+u} d B_u \right)^2 \\
& \leq C (\theta_2-\theta_1)^2  .
\end{align*}
Thus our bound on $I_2$ becomes
$I_2 \leq C_{m_{\Xi},M_{\Xi},\epsilon} \lambda^2  | \theta_2  - \theta_1 |^{2}$.
Finally, choose $\lambda$ small enough so that
\begin{align*}
I_2 \leq \frac{c_1}{16}  |\theta_1-\theta_2|^2 .
\end{align*}
To finish the proof, combine the bounds obtained for $I_{1j},\, I_{2}$ and $I_{3}$ into \eqref{eq:dcf-decomp-lower-b}.
\end{proof}

\subsection{Numerical results}\label{sec:numerical}

In this section, we provide numerical examples to illustrate the main results of this paper. We only deal with the $1d$ OU model defined in \eqref{eq:fO-U} that starts from $0$,  as it already raises numerous questions about the numerical implementation. We explain at the end how one might extend our approach to more general SDEs of the form \eqref{eq:fsde0}. 
\paragraph{Simulated data.} 
The fractional OU process cannot be simulated exactly. Therefore, we have chosen to approximate it by the Euler scheme with very small time-step $\underline{h}$ (namely $\underline{h} = 10^{-3} $).
Recall that the $L^2$-distance between the true SDE and the Euler scheme is of order $\underline{h}^H$ when both are defined with the same fBm. This result holds independently of the time horizon when the drift is contractive, see e.g. \cite[Proposition 3.7 (i)]{panloup2020general}. 
Recall also that the fBm can be simulated through the Davies-Harte method. Therefore, up to the approximation of the true SDE, we now assume that we are given a sequence $(U_{k h})_{k \ge 0}$, where $(U_t)_{t \ge 0}$ is a solution to \eqref{eq:fO-U} with a given $\theta_0$. Then we create from this path a subsequence of augmented observations $(X_{t_k} )_{k=1,\dots,n}$ as defined in \eqref{eq:defX}. Here we consider the linear transformation to be the simple increments as in \eqref{eq:simpleinc}.
Furthermore, we consider the time-steps $t_k$ to be of the form $t_k = k h$, which means in particular that we assume $h$ to be of the form $k_0 \underline{h}$ with $k_0 \in \mathbb{N}_*$ (namely $k_0= 100$).
Moreover, to compare the estimators \eqref{eq:thetan-0} and \eqref{eq:thetan-02}, for $\gamma = 0.1$, we simulate $N$ of the Euler approximation $U^{\theta,\gamma}$ defined in \eqref{eq:euler-scheme} with $b_{\theta}(x)=-\theta x$. Then we create from this path a subsequence of augmented Euler approximations $(X^{\theta,\gamma}_{k \gamma})_{k=1,\dots, N}$ as defined in \eqref{eq:Xthetagamma} taking again the linear transformations as simple increments.
For the rest of this section, we will use the following terminology:
\begin{itemize}
\item One-dimensional case: This is when we only use the first component of $X^{\theta_0}$ (i.e. $U^{\theta_0}$) as observations. This means that we are only interested in estimating one parameter (either the drift, the diffusion or the Hurst parameter) and we assume the other two are known. There are thus three choices to consider.
\item Two-dimensional case: This is when we want to estimate two parameters and therefore take the first two components of $X^{\theta_0}$ as observations. There are also three choices to consider.
\item Three-dimensional case: This is when we want to estimate all the parameters and therefore consider all the components included in $X^{\theta_0}$. 
\item In the simulations, we shall refer to the $\widehat{\theta}_{n}$ as the oracle and $\widehat{\theta}_{n,N,\gamma}$ as the estimate.
\end{itemize}
\paragraph{Computation of the distance between the empirical measures.} In practice, to implement the estimator \eqref{eq:thetan-0}, one needs to compute the distance $d \in \mathcal{D}_p$ between the average of Dirac measures and the stationary distribution.
If the observed process is $\R$-valued, and $d$ is given by the Wasserstein distance, an explicit computation is possible. However, as we explained in the introduction, using the observations of $U^{\theta_0}$ only allows us to estimate one parameter. If we want to estimate more, we need to add increments of the process into the observations. Unfortunately, the computation of the Wasserstein distance in higher dimension requires approximation/optimization methods that are highly expensive in terms of complexity and are not discussed in this paper. In this context and as in \cite{panloup2020general}, it is simple to consider an approximation of the distance $d_{CF,p}$ defined in \eqref{eq:d_cf}, which we also worked with to obtain the rate of convergence. More specifically, we want a discretisation technique for the integral that appears in \eqref{eq:d_cf}. 
\paragraph{Minimization of the distance with respect to $\theta$.} 
To implement the estimators, we see the problem of computing the $\argmin$ in \eqref{eq:thetan-0} as an optimization problem. More specifically, in the Ornstein-Uhlenbeck case, we already have an expression of the stationary distribution \eqref{eq:fOUinvariant}. Furthermore, we also know how to express the covariance between the process and its increments \eqref{eq:fOUcov}. Since the stationary distribution is $\mu_\theta \sim N(0, \Sigma_\theta)$, we have all the information that is needed to simulate it. 
In this case, to compute $\widehat{\theta}_n$ defined in \eqref{eq:thetan-0} we want to minimise
\begin{align}\label{eq:functional}
F : \theta \mapsto d\Big(\frac{1}{n} \sum_{k=0}^{n-1} \delta_{X_{t_k}^{\theta_0}}, \mu_{\theta}\Big) .
\end{align}
%
We adapt the technique described in \citep[Equations (7.5)-(7.6)]{panloup2020general}.
Taking $d=d_{CF,p}$, the idea is to write the functional $F$ as
\begin{align}\label{eq:functional-xi}
F(\theta) = d_{CF,p}(\mu,\mu_{\theta}) = \mathbb{E} [ | \mu(f_\Phi) - \mu_\theta(f_\Phi) |^2 ] .
\end{align}
where $\mu= \frac{1}{n} \sum_{k=0}^{n-1} \delta_{X_{t_k}^{\theta_0}} \,$, $f_\phi(x) = e^{i \langle x, \phi \rangle}$ and $\Phi$ is random variable that has $g_p$ as density  (see \eqref{eq:gp}).
 Writing $F$ like this allows to perform a gradient descent algorithm. In fact, an approximation of the gradient $\nabla F$ is formally obtained as
\begin{align*}%
 \widehat{\nabla} F := \frac{1}{R}\sum_{r=1}^R \Lambda (\theta, \Phi^r) ,
\end{align*}
where $(\Phi^r)_{r=1,\dots,R}$ is a sequence of i.i.d random variables with law $g_p$ and
\begin{align*}
\Lambda (\theta, \phi) & = \partial_\theta \left( | \mu (f_\phi) - \mu_\theta(f_\phi) |^2 \right)  \\
& = 2 \left( \frac{1}{n} \sum_{k=0}^{n-1} \cos(\langle \phi, X_{t_k}^{\theta_0} \rangle)  - e^{-\frac{1}{2}\phi^T \Sigma_\theta \phi}  \right) \nabla \left( e^{-\frac{1}{2}\phi^T \Sigma_\theta \phi} \right) \\
& =- \left( \frac{1}{n} \sum_{k=0}^{n-1} \cos(\langle \phi, X_{t_k}^{\theta_0} \rangle)  - e^{-\frac{1}{2}\phi^T \Sigma_\theta \phi}  \right)  e^{-\frac{1}{2}\phi^T \Sigma_\theta \phi} \, \nabla \left( \phi^T \Sigma_\theta	 \phi \right) .
\end{align*}
Hence the gradient algorithm reads
\begin{align}\label{eq:GD}
\forall t \in \{0,\dots,T\}, ~ \theta_{t+1} = \theta_t - \eta_t \frac{1}{R}\sum_{r=1}^R \Lambda(\theta_t, \Phi^r_{t+1}) ,
\end{align}
where $(\eta_t)_t$ is a sequence of positive steps and at each gradient step $t$, $(\Phi_t^r)_{t,r}$ is a sequence of i.i.d random variables with law $g_p$.

Moreover, to compute $\widehat{\theta}_{n,N,\gamma}$ defined in \eqref{eq:thetan-02}, we want to minimise
\begin{align*}
F : \theta \mapsto d\Big(\frac{1}{n} \sum_{k=0}^{n-1} \delta_{X_{t_k}^{\theta_0}}, \frac{1}{N} \sum_{k=0}^{N-1} \delta_{X_{t_k}^{\theta,\gamma}}\Big) .
\end{align*}
In this case, write 
$F(\theta)= \mathbb{E} [ | \mu(f_\Phi) - \frac{1}{N} \sum_{k=0}^{N-1} e^{i\langle X^{\theta,\gamma}_{k\gamma}, \Phi \rangle} |^2 ]$.
The gradient $\Lambda$ can be written as in \cite[Eq (7.6)]{panloup2020general}:
\begin{equation*}
\begin{split}
\Lambda(\theta,\phi) & = 2 ( \mu_\theta- \mu) \, \cos(\langle  \phi, . \rangle) \, \rho_\theta (-\sin\langle \phi, . \rangle)  \\
& \quad + 2 ( \mu_\theta- \mu) \, \sin(\langle  \phi, . \rangle)\,  \rho_\theta (\cos\langle \phi, . \rangle)  ,
\end{split}
\end{equation*}
where for any function $g: \mathbb{R} \rightarrow \mathbb{R}$, each component of $\rho_\theta(g(\langle \phi,. \rangle))$ reads:
\begin{align*}
\rho_\theta(g(\langle \phi,. \rangle))^i = \frac{1}{N} g(\langle \phi, X^{\theta,\gamma}_{k \bar{\gamma}} \rangle) \langle \phi, \partial_{\theta^i} X^{\theta,\gamma}_{k \bar{\gamma}}  \rangle .
\end{align*}

Therefore, the question is how to simulate paths of the process $\partial_{\theta^i} X^{\theta,\gamma}_\cdot$. In \cite{panloup2020general} the authors handle the case when $\theta^i$ is the drift parameter $\xi$ and explain how the process can be simulated recursively as
\begin{align*}
\partial_\xi U_{(k+1)\gamma}^{\theta,\gamma} = \partial_\xi U_{k \gamma}^{\theta,\gamma} + \gamma \left( \partial_\xi b_\xi(U_{k \gamma}^{\theta,\gamma} ) + \nabla b_\xi (U_{k \gamma}^{\theta,\gamma} ) \partial_\xi U_{k \gamma}^{\theta,\gamma} \right) .
\end{align*}
The same technique can be used when $\theta^i$ is the diffusion parameter $\sigma$:
\begin{align*}
\partial_\sigma U_{(k+1)\gamma}^{\theta,\gamma} = \partial_\sigma U_{k \gamma}^{\theta,\gamma} + \gamma \nabla b_\xi (U_{k \gamma}^{\theta,\gamma} ) \partial_\sigma U_{k \gamma}^{\theta,\gamma} +  \left( B_{(k+1)\gamma} - B_{k \gamma} \right) .
\end{align*}
Finally, in order to compute $(\partial_H U_{k \gamma})_{k=0,\dots,N}$ in the same way, one needs to compute $\partial_H B$, which is not an obvious task. For instance, using the 
Mandelbrot-Van Ness representation \eqref{eq:MVN}, one cannot simply differentiate the integrand with respect to $H$ to get $\partial_H B$. In \cite{Neuenkirch}, it is shown that for all $t \ge 0$, $B_t$ is almost surely infinitely differentiable with respect to $H$. But since we consider ergodic increments, we need a result that states: almost surely, for all $t \ge 0$, $B_t$ is infinitely differentiable with respect to $H$. In \cite{HRarxiv}, it is shown that the solution to \eqref{eq:fsde} is $(1-\varepsilon)$-H\"older continuous in $H$. So in this case, one simply approximates the gradient of $\theta \mapsto F(\theta)$ by a finite difference with step $\delta=0.1$. Denoting this approximation by $\widehat{\Lambda}(\theta,\Phi)$, the gradient algorithm reads as in \eqref{eq:GD} with $\widehat{\Lambda}(\theta_t, \Phi_t^r)$ in place of $\Lambda(\theta_t, \Phi_t^r)$.
%
\paragraph{Simulation of the variable $\Phi$.} 
Since $g_p$ has a spherical form, $\Phi$ can be simulated using the spherical coordinates and the inverse transform sampling method in any dimension (see e.g. \cite[Section 7]{panloup2020general} for $d=2$).
\paragraph{Numerical illustrations.}
Recall that we consider the process $U$ given by \eqref{eq:fO-U} and we assume $\theta$ to be in a compact interval. The assumptions \ref{asmp:compact} and \ref{asmp:drift} are clearly satisfied, where \ref{asmp:identif} follows from Proposition \ref{prop:identifOU}. Moreover, Lemma \ref{lem:strongidentif-O-U} proves that \ref{asmp:strong_identif} is satisfied when we are only interested in estimating one parameter. Using the strategy described before, we get a discretely observed path of $X$ and an Euler approximation $X^{\theta,\gamma}$. We set the following parameters:
\begin{align*}
\theta_0 &= (\xi_0, \sigma_0 , H_0)  = (2, 0.5, 0.7) \\
\underline{h} & = 10^{-3},~q = 3,~ h = 10^{-1} \\
n & = 1000,~ N = 10 000,~ \gamma = 10^{-1} \\
R & = 100,~ T = 100,~ p = 2 . 
\end{align*}
Let us start with the one-dimensional case (Figure \ref{fig:1D}). We perform the gradient descents described above over $100$ realisations of the observations $(X^{\theta_0}_{kh})_{k=1,\dots,n}$ and plot a histogram highlighting the empirical mean obtained and the empirical variance. More precisely, we denote the empirical variance when estimating $\xi_0, \sigma_0$ and $H_0$ respectively by $\Var_{\xi}$, $\Var_\sigma$ and $\Var_H$. 
%
We use the same algorithm in the two-dimensional case (Figures \ref{fig:2D-1}, \ref{fig:2D-2} and \ref{fig:2D-3}) and the three-dimensional case (Figure \ref{fig:3D}) and plot histograms for all the parameters we are interested in estimating. Take $\theta^0 = (\xi^0 = 1, \sigma^0 = 1, H^0 =0.5 )$ as the initial point in the gradient descents. 
\begin{figure}[h!]
\centering
\includegraphics[width=0.4\textwidth]{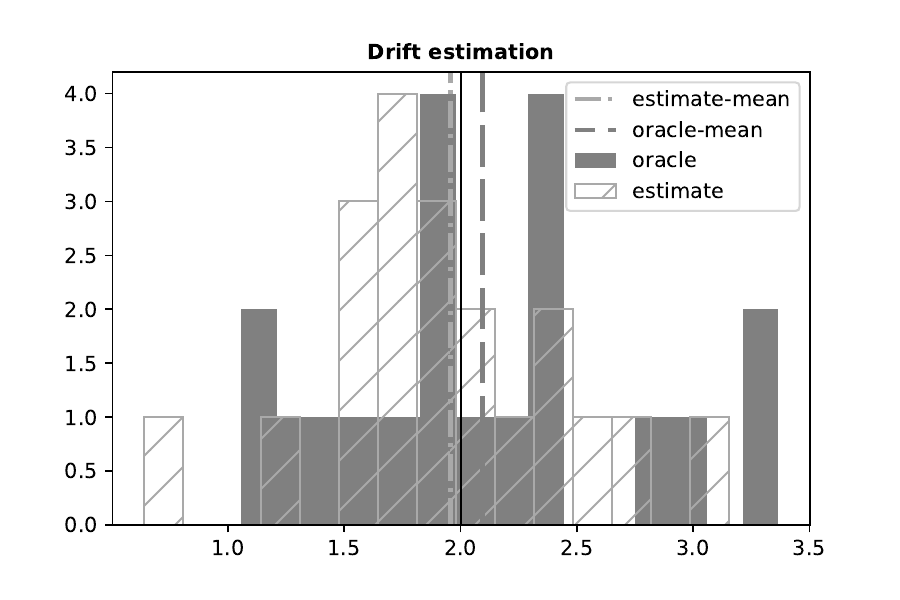}\,\includegraphics[width=0.4\textwidth]{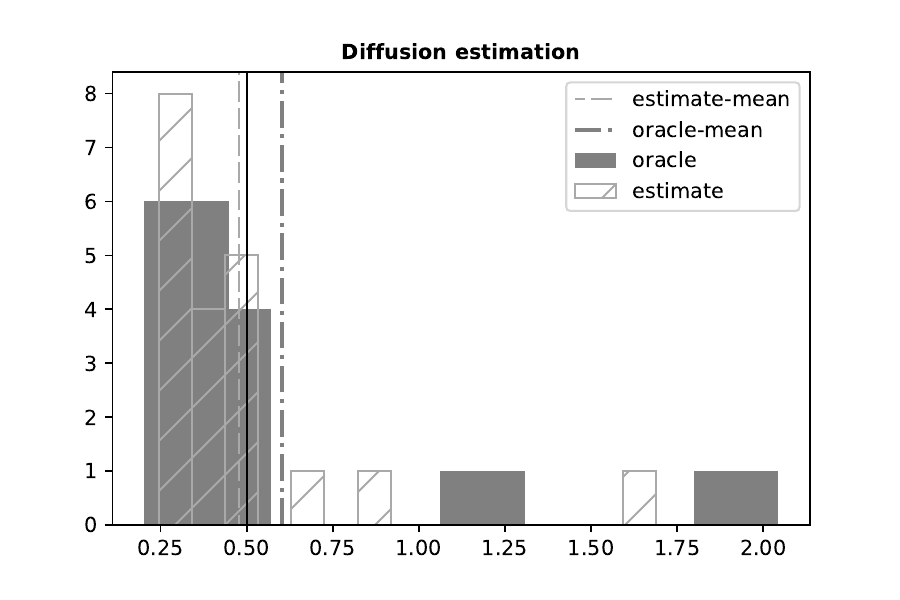}\\\includegraphics[width=0.4\textwidth]{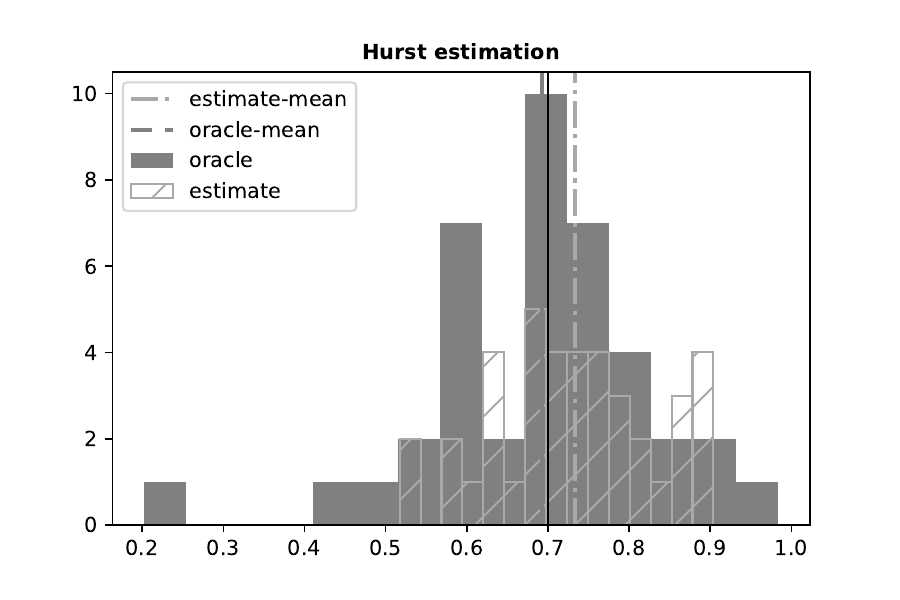}\\
\caption{Histograms for the estimation of each parameter separately. The filled vertical lines are for the true parameters. The dash lines represent the empirical mean of $\widehat{\theta}_n$ and the dash-dotted lines the empirical mean of $\widehat{\theta}_{n,N,\gamma}$. Left: $\Var_\xi (\widehat{\theta}_n) \sim 0.1$, $\Var_\xi(\widehat{\theta}_{n,N,\gamma}) \sim 0.1$. Right: $\Var_\sigma(\widehat{\theta}_n) \sim 0.01$, $\Var_\sigma (\widehat{\theta}_{n,N,\gamma}) \sim 0.01$. Bottom: $\Var_H(\widehat{\theta}_n) \sim 0.01$, $\Var_H(\widehat{\theta}_{n,N,\gamma}) \sim 0.01$.}
\label{fig:1D}
\end{figure}
\begin{figure}[h!]
\centering
\includegraphics[width=0.4\textwidth]{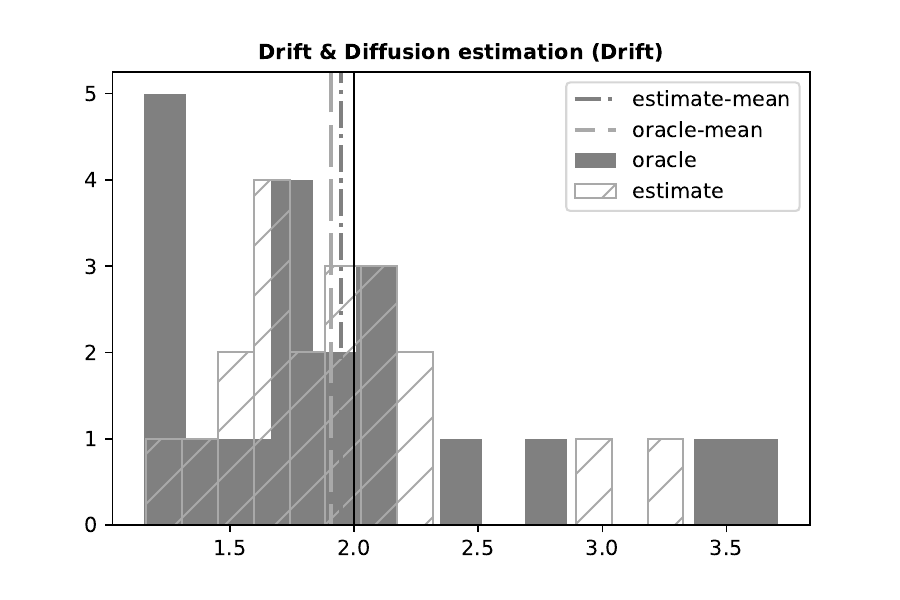}\,\includegraphics[width=0.4\textwidth]{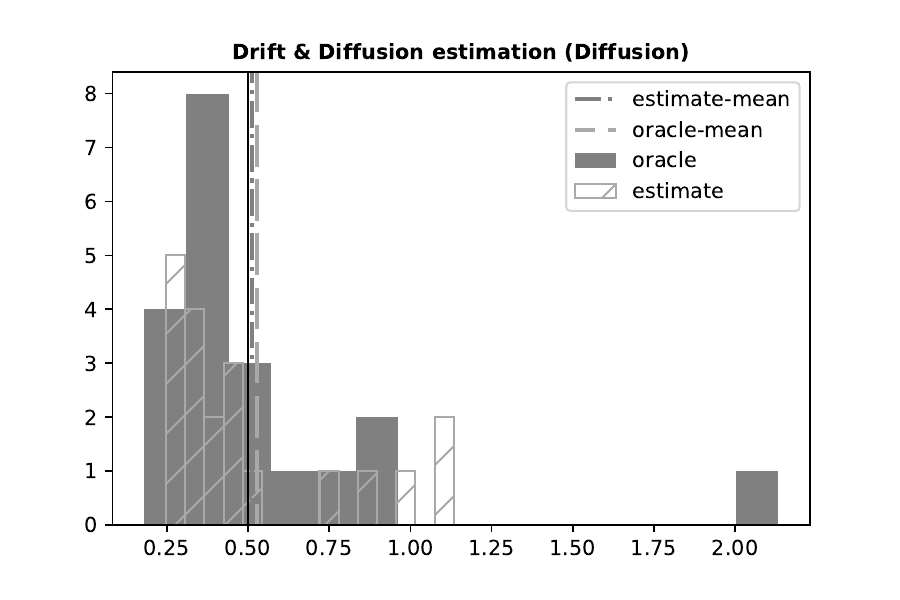}
\caption{Histograms for the estimation of the drift and the diffusion. The filled vertical lines are for the true parameters. The dash lines represent the empirical mean of $\widehat{\theta}_n$ and the dash-dotted lines the empirical mean of $\widehat{\theta}_{n,N,\gamma}$. Left: $\Var_\xi (\widehat{\theta}_n) \sim 0.1$, $\Var_\xi (\widehat{\theta}_{n,N,\gamma}) \sim 0.1$. Right: $\Var_\sigma  (\widehat{\theta}_n)\sim 0.01$, $\Var_\sigma(\widehat{\theta}_{n,N,\gamma}) \sim 0.01$.}
\label{fig:2D-1}
\end{figure}
\begin{figure}[h!]
\centering
\includegraphics[width=0.4\textwidth]{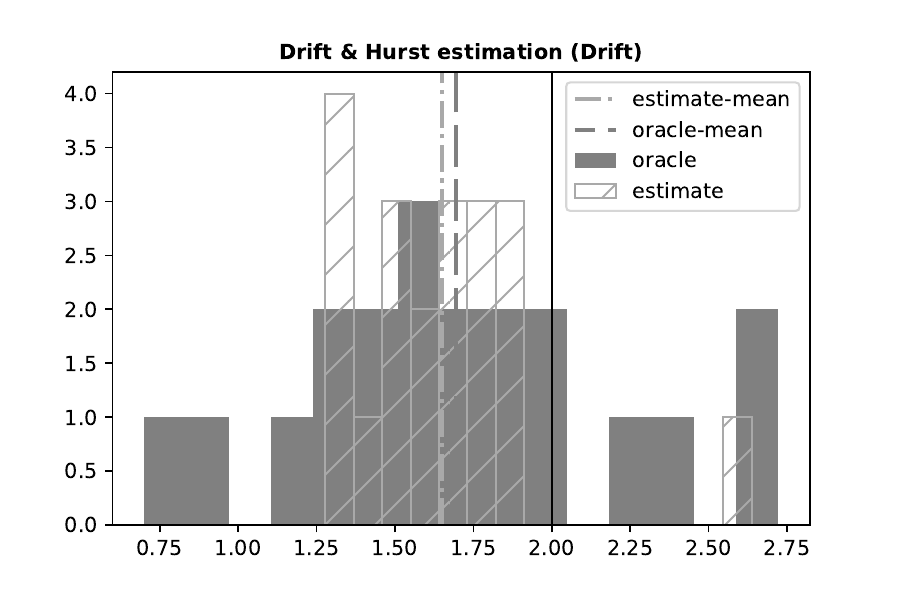}\,\includegraphics[width=0.4\textwidth]{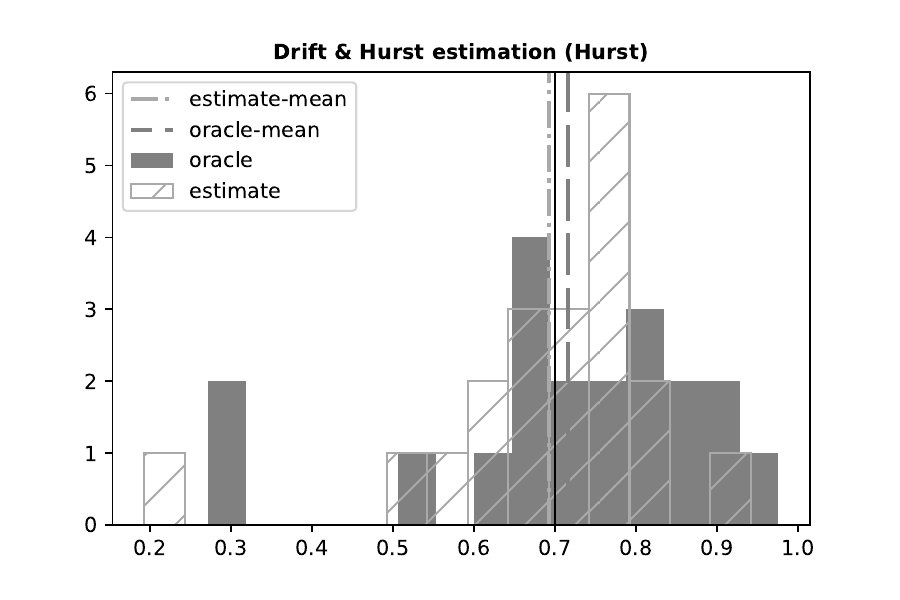}
\caption{Histograms for the estimation of the drift and the Hurst parameter. The filled vertical lines are for the true parameters. The dash lines represent the empirical mean of $\widehat{\theta}_n$ and the dash-dotted lines the empirical mean of $\widehat{\theta}_{n,N,\gamma}$. Left: $\Var_\xi (\widehat{\theta}_n)\sim 0.1$, $\Var_\xi (\widehat{\theta}_{n,N,\gamma}) \sim 0.1$. Right: $\Var_H(\widehat{\theta}_n) \sim 0.01$, $\Var_H(\widehat{\theta}_{n,N,\gamma}) \sim 0.01$.}
\label{fig:2D-2}
\end{figure}
\begin{figure}[h!]
\centering
\includegraphics[width=0.4\textwidth]{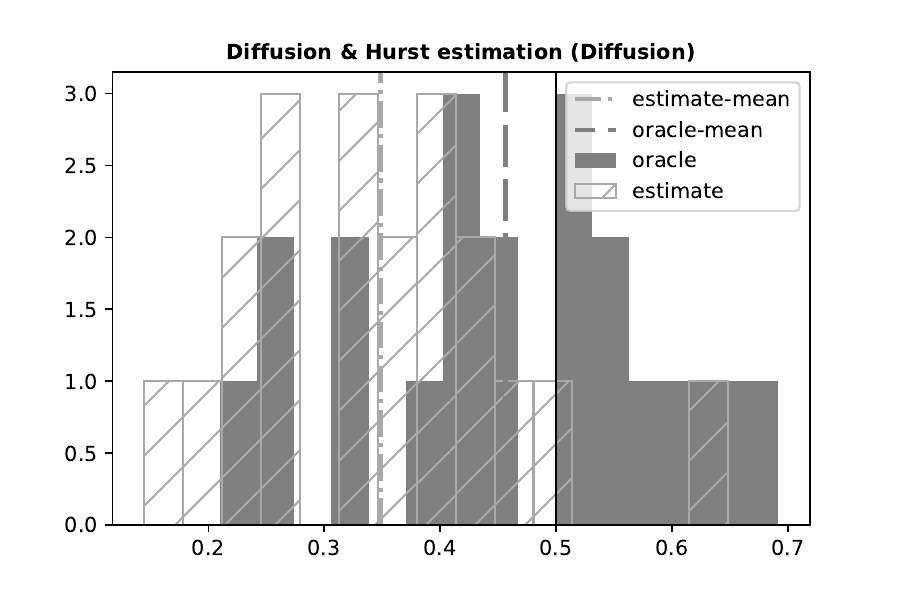}\,\includegraphics[width=0.4\textwidth]{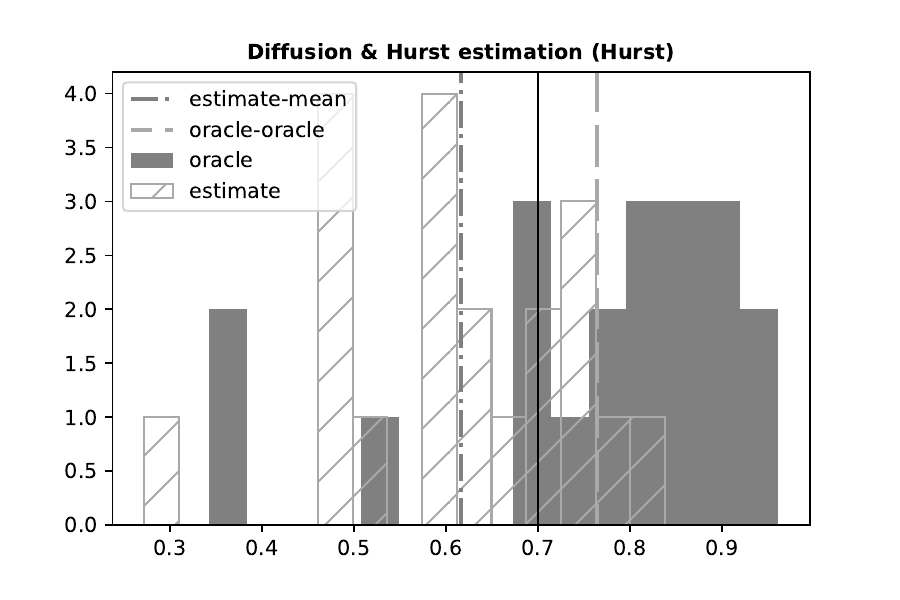}
\caption{Histograms for the estimation of the diffusion and the Hurst parameter. The filled vertical lines are for the true parameters. The dash lines represent the empirical mean of $\widehat{\theta}_n$ and the dash-dotted lines the empirical mean of $\widehat{\theta}_{n,N,\gamma}$. Left: $\Var_\sigma (\widehat{\theta}_n)\sim 0.01$, $\Var_\sigma (\widehat{\theta}_{n,N,\gamma})\sim 0.01$. Right: $\Var_H(\widehat{\theta}_n) \sim 0.01$, $\Var_H(\widehat{\theta}_{n,N,\gamma}) \sim 0.01$.}
\label{fig:2D-3}
\end{figure}
\begin{figure}[h!]
\centering
\includegraphics[width=0.4\textwidth]{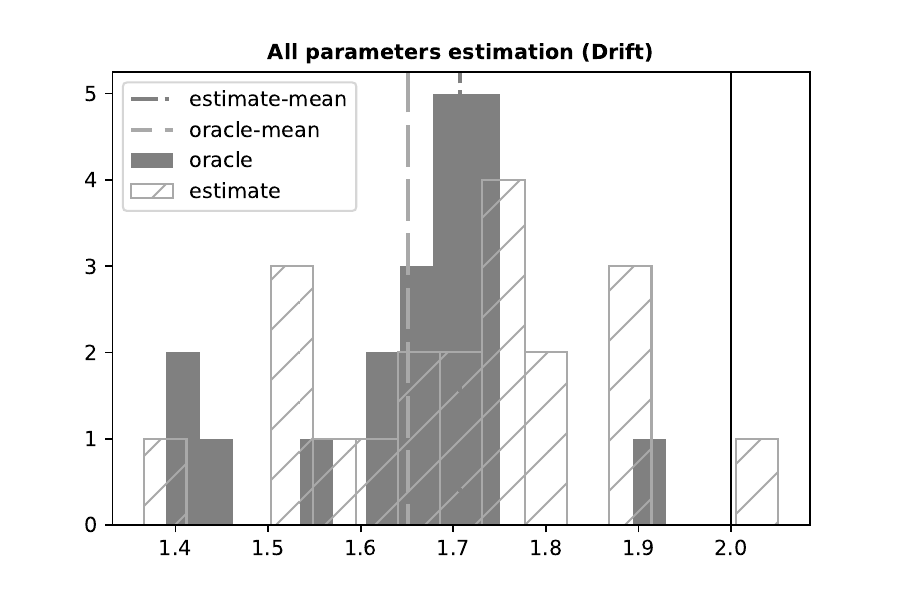}
\includegraphics[width=0.4\textwidth]{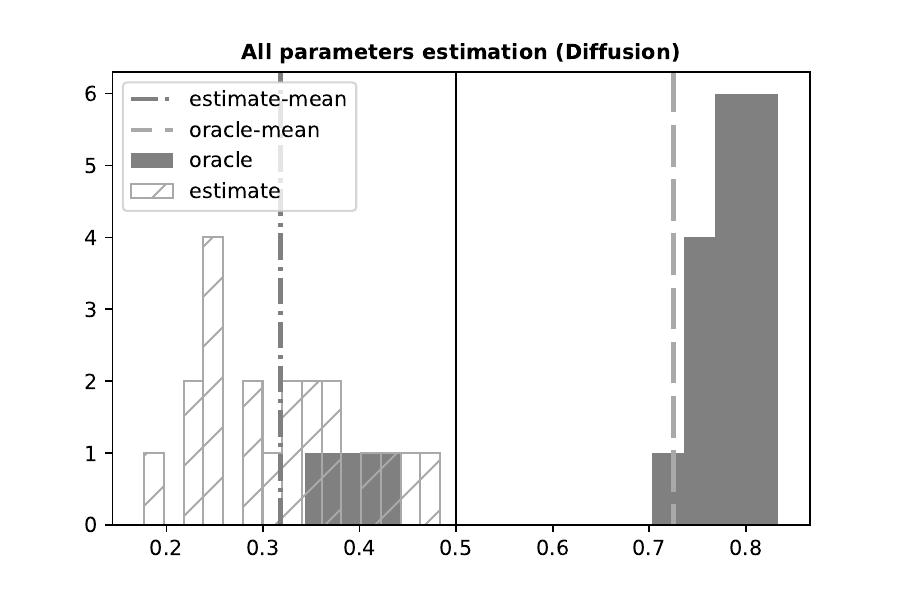}
\includegraphics[width=0.4\textwidth]{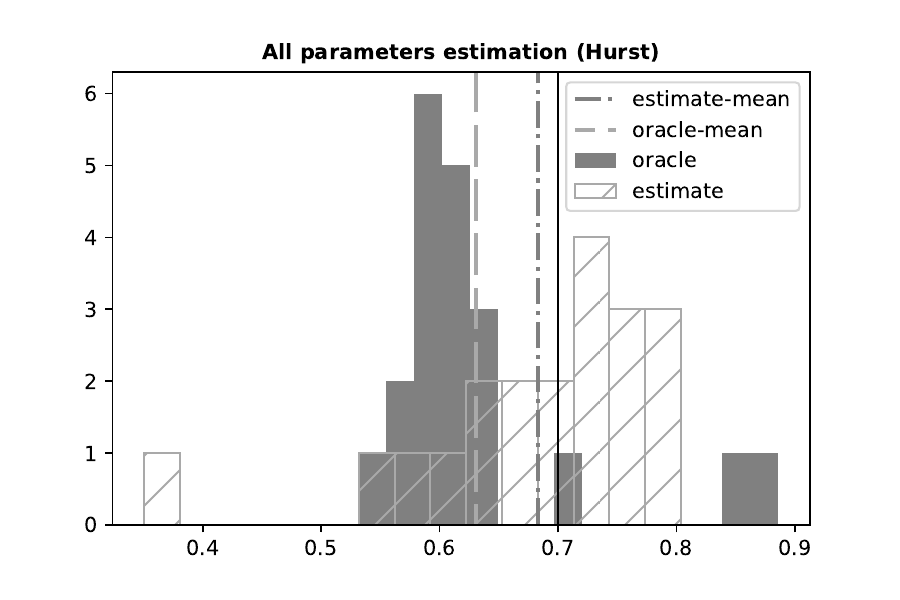}
\caption{Histograms for the estimation of all the parameters. The filled vertical lines are for the true parameters. The dash lines represent the empirical mean of $\widehat{\theta}_n$ and the dash-dotted lines the empirical mean of $\widehat{\theta}_{n,N,\gamma}$. Left: $\Var_\xi(\widehat{\theta}_n) \sim 0.1$, $\Var_\xi (\widehat{\theta}_{n,N,\gamma})\sim 0.1$. Right: $\Var_\sigma(\widehat{\theta}_n)\sim 0.01$, $\Var_\sigma(\widehat{\theta}_{n,N,\gamma})\sim 0.01$. Bottom: $\Var_H(\widehat{\theta}_n)\sim 0.01$, $\Var_H(\widehat{\theta}_{n,N,\gamma})\sim 0.01$.}
\label{fig:3D}
\end{figure}
\paragraph{Discussion.}In the $1d$ case, we get accurate estimators of the parameters  (see Figure \ref{fig:1D}). 
In the $2d$ case, one observes a decrease in the accuracy of the estimates (see the left histogram on Figure \ref{fig:2D-2} and Figure \ref{fig:2D-3} for instance), which is due to a higher bias, as the variances of the estimators stay the same as in the $1d$ case. In the $3d$ case, one observes an increase in both the bias and the variances of the estimators (specially in the estimation of the drift). 
Furthermore, observe that for $N=10000$ and $\gamma=0.1$, the estimator $\widehat{\theta}_{n,N,\gamma}$ behaves qualitatively like $\widehat{\theta}_n$. While $\widehat{\theta}_{n,N,\gamma}$ is slow to compute, this ensures that the error generated by replacing $\theta \mapsto \mu_\theta$ by its approximation is not significant. 
Finally, note that in practice in the $3d$ case, the oracle $\widehat{\theta}_n$ is slow to compute. In fact, computing $\theta \mapsto \mu_\theta$ requires computing the integrals appearing in the right-hand side of \eqref{eq:f-3D} which are slow to converge. Moreover, the error when approximating these integrals can be significant as shown in Figure \ref{fig:2D-3}. Thus, even when the invariant measure is known, it can be more efficient to compute $\widehat{\theta}_{n,N,\gamma}$ than the oracle in higher dimensions.
%
In general, further exploration and in-depth analysis would be necessary to enhance the integration of our statistical procedure with gradient descent algorithms. In particular, going beyond the fractional Ornstein-Uhlenbeck model requires looking into derivative-free methods to approximate $\nabla \Lambda(\theta,\phi)$ (e.g \cite{flaxman2005online}) and faster methods to compute the approximation of $\theta \mapsto \mu_\theta$. This aspect remains open for future investigation.

\begin{appendix}

\section{Regularity in the Hurst parameter}\label{app:regH}
In this section, we recall and adapt some results from our companion paper \cite[Sections 4 and 5]{HRarxiv} that state the regularity in the Hurst parameter of continuous and discrete ergodic means. Recall that the fractional OU process is defined by \eqref{eq:fO-U}, and let us denote by $\bar{U}^{(1,\sigma,H)}$ the stationary fractional OU process with drift $\xi=1$, diffusion matrix $\sigma$ and Hurst parameter $H$.

In the whole Appendix, let $\mathcal{H}$ be a compact subset of $(0,1)$, $\Xi$ be a compact subset of $\R^m$, $\Sigma$ a compact subset of the space of $d\times d$ invertible matrices and denote $\Theta = \Xi\times \Sigma\times \mathcal{H}$.

\begin{lemma}\label{lem:thm4.5}
Let $\varpi \in\left(0, 1 \right)$ and $p \geq 1$. Let $W$ be an $\R^d$-Brownian motion and for any $H\in (0,1)$, denote by $B^H$ the fBm with underlying noise $W$ (i.e. as in \eqref{eq:MVN}). 
There exists a random variable $\mathbf{C}$ with a finite moment of order $p$ such that almost surely, for any $t \geq 0$, any $\theta_1 = (\xi, \sigma, H_1) \in \Theta$ and $\theta_2=(\xi, \sigma, H_2) \in \Theta$,
\begin{equation*}
\frac{1}{t+1} \int_0^{t+1}\left|Y_s^{\theta_1}-Y_s^{\theta_2}\right|^2 d s \leq \mathbf{C} |H_1-H_2|^\varpi ,
\end{equation*}
where $Y^{\theta_{1}}$ (resp. $Y^{\theta_{2}}$) is the solution to \eqref{eq:fsde} with parameter $\theta_{1}$ (resp. $\theta_{2}$), a drift $b_{\xi}$ satisfying \ref{asmp:drift} and driving fBm $B^{H_{1}}$ (resp. $B^{H_{2}}$),  and $Y^{\theta_{1}}$ and $Y^{\theta_{2}}$ start from the same initial condition.
\end{lemma}
\begin{proof}
For $i=1,2$, the process $\sigma^{-1} Y_\cdot^{\theta_{i}}$ is solution to the SDE
\begin{align*}
\sigma^{-1} Y_t^{\theta_{i}} = \sigma^{-1} Y_0 + \int_0^t \tilde{b}_\xi (\sigma^{-1} Y_s^{\theta_{i}}) ds + B^{H_{i}}_t ,
\end{align*}
with $\tilde{b}_\xi (x) = \sigma^{-1} b_\xi (\sigma \cdot)$. We have $\tilde{b}_{\xi} \in \mathcal{C}^{1,1}(\mathbb{R}^d \times \Xi, \mathbb{R}^d)$ and since $\sigma$ lives in the compact set $\Sigma$, $\tilde{b}_{\xi}$ still satisfies \eqref{eq:drift-coerciv} and \eqref{eq:drift-growth}. We choose the stationary fractional OU $\bar{U}^{(1,\Id,H_1)}$ with the same noise $B^{H_{1}}$ as $Y^{\theta_{1}}$ (similarly for $\bar{U}_s^{(1,\Id,H_2)}$). As in the proof of \cite[Theorem 4.5]{HRarxiv}, a comparison between $Y_s^{\theta_1}-Y_s^{\theta_2}$ and $\bar{U}^{(1,\Id,H_1)}-\bar{U}^{(1,\Id,H_2)}$ gives
\begin{align*}
\frac{1}{t+1} \int_0^{t+1}\left| \sigma^{-1}(Y_s^{\theta_1}-Y_s^{\theta_2}) \right|^2 d s & \leq C | \bar{U}_0^{(1,\Id,H_1)} - \bar{U}_0^{(1,\Id,H_2)} |^2 \\ 
& \quad + \frac{1}{t+1} \int_0^{t+1} |\bar{U}_s^{(1,\Id,H_1)} - \bar{U}_s^{(1,\Id,H_2)} |^2 ds .
\end{align*}
We can now apply \cite[Proposition 4.2]{HRarxiv} with $t'=t=0$ and \cite[Proposition 4.4]{HRarxiv} with $H'=K=K'$ and $t'=t$ to get that there exists a random variable $\mathbf{C}_1$ (independent of $\xi$ and $\sigma$) with a finite moment of order $p$ such that
\begin{align*}
\frac{1}{t+1} \int_0^{t+1}\left| \sigma^{-1}(Y_s^{\theta_1}-Y_s^{\theta_2}) \right|^2 d s \leq \mathbf{C}_1 | H_1-H_2 |^{\varpi} .
\end{align*}
Since $\left| \sigma^{-1}(Y_s^{\theta_1}-Y_s^{\theta_2}) \right| \ge | \sigma^{-1} | \left| (Y_s^{\theta_1}-Y_s^{\theta_2}) \right| $, dividing by $| \sigma^{-1}|$ and taking the supremum over $\Sigma$, we get the desired result by setting $\mathbf{C} = | \sigma^{-1} |^{-1} \mathbf{C}_1$.
\end{proof}

\begin{lemma}\label{lem:prop5.1}
Let $\mathcal{H}$ be a compact subset of $\left(0,1 \right)$, $\varpi \in\left(0, 1 \right)$, and $p \geq 1$. There exists $\gamma_0 >0$ such that for $\gamma \in\left(0, \gamma_0\right)$, there exists a random variable $\mathbf{C}_\gamma$ with a finite moment of order $p$ such that almost surely, for all $t, t^{\prime} \geq 0$ and all $H_1,H_2 \in \mathcal{H}$,
\begin{equation*}
 \frac{1}{t+1} \int_0^{t+1} | \bar{U}_{s_\gamma}^{(1,\Id,H_1)} - \bar{U}_{s_\gamma}^{(1,\Id,H_2)} |^2   ds \leq \mathbf{C}_\gamma |H_1-H_2|^\varpi ,
\end{equation*}
where $s_\gamma$ denotes the leftmost point in a time-discretisation of step $\gamma$.
\end{lemma}
\begin{proof}
Apply \cite[Proposition 5.1]{HRarxiv}, with $t'=t$ and $H'=K'=K$ to get that
\begin{align*}
\frac{1}{t+1} \int_0^{t+1} |  \bar{U}_{s_\gamma}^{(1,\Id, H_1)} - \bar{U}_{s_\gamma}^{(1,\Id,H_2)}   |^2 ds \leq \mathbf{C}_\gamma |H_1-H_2|^\varpi  + C \mathbb{E} | \bar{U}_0^{(1,\Id,H_1)} - \bar{U}_0^{(1,\Id,H_2)} |^2 .
\end{align*}
Now apply \cite[Proposition 4.2]{HRarxiv} with $t=t'=0$ to get the desired result.
\end{proof}

\begin{lemma}\label{lem:thm5.2}
Let $\mathcal{H}$ be a compact subset of $\left(0, 1 \right)$. Let $\varpi \in\left(0, 1\right)$ and $p \geq 1$. There exists $\gamma_0>0$ such that for $\gamma \in\left(0, \gamma_0\right]$, there exists a random variable $\mathbf{C}_\gamma$  with a finite moment of order $p$ such that almost surely, for any $N \in \mathbb{N}^*$, any $\theta=(\xi, \sigma, H_1)$ and any $\theta_2=(\xi,\sigma,H_2) \in \Theta$,
$$
\frac{1}{N} \sum_{k=1}^N\left|Y_{k \gamma}^{\theta_1,\gamma}-Y_{k \gamma}^{\theta_2,\gamma}\right|^2 \leq \mathbf{C}_\gamma |H_1-H_2|^\varpi ,
$$
where $Y^{\theta_2,\gamma}$ and $Y^{\theta_2,\gamma}$ are Euler schemes \eqref{eq:euler-scheme} with the same initial condition and driven by fBm with the same underlying noise (see \eqref{eq:MVN}).
\end{lemma}
\begin{proof}
For any $\theta \in \Theta$, the process $\sigma^{-1} Y_\cdot^{\theta,\gamma}$ is solution to the SDE
\begin{align*}
\sigma^{-1} Y_t^{\theta, \gamma} = \sigma^{-1} Y_0 + \int_0^t \tilde{b}_\xi (\sigma^{-1} Y_{s_\gamma}^{\theta,\gamma}) ds + B_t^H ,
\end{align*}
with $\tilde{b}_\xi (x) = \sigma^{-1} b_\xi (\sigma \cdot)$. We have $\tilde{b}_{\xi} \in \mathcal{C}^{1,1}(\mathbb{R}^d \times \Xi, \mathbb{R}^d)$ and since $\sigma$ lives in the compact set $\Sigma$, one can check that $\tilde{b}_{\xi}$ still satisfies \eqref{eq:drift-coerciv} and \eqref{eq:drift-growth}. As in the proof of \cite[Eq. (5.5)]{HRarxiv}, a comparison with the stationary fractional OU process $\bar{U}$ gives
\begin{align*}
& \frac{1}{N} \sum_{k=0}^N\left|\sigma^{-1} (Y_{k \gamma}^{\theta_1, \gamma}-Y_{k \gamma}^{\theta_2,\gamma}) \right|^2 \\ & \leq C\Bigg(\frac{1}{N} \sum_{k=0}^N\left|\bar{U}_{j \gamma}^{(1,\Id,H_1)}-\bar{U}_{j \gamma}^{(1,\Id,H_2)} \right|^2 +\left|\bar{U}_0^{(1,\Id,H_1)}-\bar{U}_0^{(1,\Id,H_2)} \right|^2 \\ & \quad +\frac{1}{N \gamma} \int_0^{N \gamma}\left|U_s^{(1,\Id,H_1)}-U_s^{(1,\Id,H_2)}\right|^2 d s\Bigg).
\end{align*}
The regularity of the second term in the right-hand side is given by \cite[Proposition 4.2]{HRarxiv} and the regularity of the third term is given by \cite[Theorem 4.5]{HRarxiv}. To bound the first term, we apply Lemma \ref{lem:prop5.1}. To conclude the proof, we notice that $|\sigma^{-1} (Y_{k \gamma}^{\theta_1,\gamma}-Y_{k \gamma}^{\theta_2,\gamma}) | \ge | \sigma^{-1} | | (Y_{k \gamma}^{\theta_1}-Y_{k \gamma}^{\theta_2} ) |$, divide by $| \sigma^{-1}|$ and take the supremum over $\Sigma$.
\end{proof}

\section{Continuity and Tightness results} \label{appendix-L}
In Proposition \ref{prop:uniform-Lq-bounds} and Proposition \ref{prop:uniform-Lq-bounds-L}, we prove that the solutions $Y^\theta$ and $Y^{\theta,\gamma}$ to \eqref{eq:fsde} and \eqref{eq:euler-scheme} and their ergodic means have finite moments uniformly in time and $\theta$. Finally, in Proposition \ref{prop:regularity-result}, we state a result on the regularity of the ergodic means in $\theta$.
\begin{prop}\label{prop:uniform-Lq-bounds}
Assume \ref{asmp:compact} and \ref{asmp:drift}. Let $Y^{\theta}$ be the unique solution of \eqref{eq:fsde}. Let $p > 1$. Then the following inequalities hold true:
\begin{itemize}
\item[(i)] $\displaystyle\sup_{t \ge 0} \sup_{\theta \in \Theta} 	\mathbb{E}  |Y_t^{\theta}|^p  < \infty$. 
\item[(ii)] $\displaystyle\mathbb{E} \left( \sup_{t \ge 0} \sup_{\theta \in \Theta}  \frac{1}{t} \int_0^t |Y_s^\theta |^2 ds \right)^{p} < \infty $. 
\item[(iii)] $\displaystyle\mathbb{E} \left( \sup_{n\in \N^*} \sup_{\theta \in \Theta}  \frac{1}{n} \sum_{k=0}^{n-1} |Y_{k h}^\theta |^2 \right)^{p} < \infty $. 
\end{itemize}
\end{prop}

\begin{proof}
Throughout the proof, $C$ will denote a constant that  do not depend on $\theta$ or $t$ and that may change from line to line. Observe that when the supremum is taken only over $\xi$, the proof is already done in \cite[Proposition A.1]{panloup2020general}. The proofs of all three items are based on a comparison with fractional OU processes \eqref{eq:fO-U}. 

For the proof of $(i)$, by \cite[p.725]{Hairer}, a comparison with the stationary fractional OU process $\bar{U}^{(1,  \sigma,  H)} $ yields that there exist constants $c_1,c_2>0$ independent of $\xi$ such that,
\begin{align*}
| Y_t - \bar{U}_t^{(1,  \sigma,  H)}|^p \leq e^{-2c_1 t} | Y_0 |^p + c_2 \int_0^t e^{-2 c_2 (t-s)} (1+ | \bar{U}_s^{(1,  \sigma,  H)} |^p) ds .
\end{align*}
Moreover, since $U^{(1,  \sigma,  H)}$ is a Gaussian process, for any $t \ge 1$, we have $\mathbb{E} | \bar{U}_t^{(1,  \sigma,  H)} |^p  \lesssim  ( \mathbb{E} | \bar{U}_t^{(1,\sigma,H)}  |^{2} )^{p/2}$. By \eqref{eq:fOUinvariant}, we know that 
$\mathbb{E} | \bar{U}_t^{(1,  \sigma,  H)} |^{2}  = \sigma^2 H \Gamma(2H)$. 
Therefore
\begin{align*}
\sup_{t \ge 0} \sup_{\theta \in \Theta} \mathbb{E} |Y_t^{\theta}|^p \leq C (1 + \sup_{t \ge 0} \sup_{\theta \in \Theta} \mathbb{E}   |\bar{U}_t^{(1,\sigma,H)}|^p ) < \infty .
\end{align*}   
For the proof of $(ii)$, we follow the steps of the proof of Proposition A.1 in \cite{panloup2020general} (see equation (A.6) and what follows), to get that for all $t >0$,
\begin{align*}
\sup_{\theta \in \Theta}  \frac{1}{t} \int_0^t |Y_s^\theta |^2 ds & \leq C \sup_{\theta \in \Theta} \frac{1}{t} \int_0^t | U_s^{\theta} |^2 d s  \leq C \sup_{\theta= (1,  \sigma,  H) \in \Theta} \frac{1}{t} \int_0^t  | \sigma | | U_s^{(1, \Id,  H)} |^2 d s  .
\end{align*}
It follows that
\begin{align*}
 \sup_{\theta\in \Theta} \frac{1}{t} \int_0^t |Y_s^\theta |^2 ds &   \leq C\sup_{H \in \mathcal{H}}   \frac{1}{t} \int_0^t  | U_s^{(1, \Id,  H)} |^2 d s \\
& \leq C \Biggl( \sup_{H \in \mathcal{H}} \frac{1}{t } \int_0^t | U_s^{(1, \Id,  H)} - U_s^{(1,\Id,1/2)}|^2 d s + \frac{1}{t } \int_0^t | U_s^{(1,\Id,1/2)}|^2 d s \Biggr)   .
\end{align*}
Moreover, by Lemma \ref{lem:thm4.5} applied to $Y^\theta \equiv U^{(1, \Id,  H)}$ we have that for any $\varpi \in (0, 1)$, there exists a random variable $\mathbf{C}$ with a finite moment of order $p$ such that for any $t \ge 1$,
\begin{align*}
\frac{1}{t} \int_0^t \sup_{\theta\in \Theta} |Y_s^\theta |^2 ds 
& \leq C \Biggl( \mathbf{C} \sup_{H \in \mathcal{H}} | H-\frac{1}{2}|^{\varpi} + \frac{1}{t} \int_0^t | U_s^{(1,\Id,1/2)}|^2 d s  \Biggr) .
\end{align*}
The ergodicity of $U^{(1,\Id,1/2)}$ implies that $ \frac{1}{t } \int_0^t | U_s^{(1,\Id,1/2)}|^2 d s$ converges as $t\to \infty$. It follows that
\begin{align*}
\mathbb{E} \left( \sup_{t >0} \sup_{\theta\in \Theta} \frac{1}{t} \int_0^t |Y_s^\theta |^2 ds \right)^{p} < \infty  .
\end{align*}
The proof of $(iii)$ can be done in the exact same way by transcribing all the integrals to discrete sums and using Lemma \ref{lem:prop5.1}.
\end{proof}

\begin{prop}\label{prop:uniform-Lq-bounds-L}
Assume \ref{asmp:compact} and \ref{asmp:drift}. Let $Y_\cdot^{\theta,\gamma}$ be the Euler scheme \eqref{eq:euler-scheme}. Then there exists $\gamma_0>0$ such that for any $p> 1$, there is
\begin{itemize}
\item[(i)] $\displaystyle \sup_{\theta \in \Theta, \gamma \in (0, \gamma_0)} \limsup_{N \rightarrow \infty}  \mathbb{E} \left|Y^{\theta,\gamma}_{N \gamma} \right|^p  < \infty$.
\item[(ii)] For $\gamma \in (0,\gamma_0]$,  $\displaystyle\mathbb{E} \left( \sup_{\theta \in \Theta} \sup_{N \ge 1}  \frac{1}{N} \sum_{k=0}^{N-1} | Y^{\theta,\gamma}_{k \gamma}|^2 \right)^{p} < \infty $.
\end{itemize}
\end{prop}
\begin{proof}
Note that the same results are proven in \cite[Proposition A.4]{panloup2020general} when $\Theta$ only represents the range of the parameter $\xi$. With this in mind, as in Proposition \ref{prop:uniform-Lq-bounds}, the proof of $(i)$ is based on comparisons with the discrete Ornstein-Uhlenbeck process, which has finite moments uniformly in $\theta$. The proof $(ii)$ is the same as the proof of $(ii)$ in Proposition \ref{prop:uniform-Lq-bounds} and is based on a comparison with the discrete OU process and Lemma \ref{lem:thm5.2}.
\end{proof}

\begin{prop}\label{prop:regularity-result}
Let the assumptions \ref{asmp:compact} and \ref{asmp:drift} hold. Assume also that the exponent $r$ in the sub-linear growth of $b_\xi$ in \eqref{eq:drift-growth} satisfies $r \leq 1$. Let $p \ge 1$ and $\varpi \in (0,1)$, then there exists a positive random variable $\mathbf{C}$ that has a finite moment of order $p$, such that almost surely for all $\theta_1, \theta_2 \in \Theta$ and for all $t \ge 1$,
\begin{align}\label{eq:continuityM-theta}
\frac{1}{t} \int_0^t | Y_s^{\theta_1} - Y_s^{\theta_2} |^2 ds  \leq \mathbf{C}  | \theta_1 - \theta_2 |^{\varpi}   ,
\end{align}
where $Y^{\theta_1}$ and $Y^{\theta_2}$ are solutions to \eqref{eq:fsde} with the same initial condition and driven by an fBm with the same underlying noise (see \eqref{eq:MVN}).
Furthermore, there exists $\gamma_0$ such that for any $\gamma \in (0,\gamma_0]$, there exists a positive random variable $\mathbf{C}_\gamma$ that has a finite moment of order $p$, such that almost surely, for any $\theta_1, \theta_2 \in \Theta$ and any $N \ge 1$,
\begin{align}\label{eq:continuityM-theta-discret}
\frac{1}{N} \sum_{k=0}^N | Y_{k \gamma}^{\theta_1,\gamma} -   Y_{k \gamma}^{\theta_2,\gamma}  |^2ds  \leq \mathbf{C}_\gamma  | \theta_1 - \theta_2 |^{\varpi}   .
\end{align}
\end{prop}
\begin{proof} In the proof, we denote by $C$ a constant independent of time and $\theta$ that may change from line to line. Similarly, $\mathbf{C}$ will denote a positive random variable that has a finite moment of order $p$, that does not depend on $\theta$ and may change from line to line.
Let us first focus on on the proof of \eqref{eq:continuityM-theta}. Up to introducing pivot terms, we can consider three different cases:
\begin{align*}
\text{(1)}~  \theta_1 & = (\xi_1, \sigma, H), \theta_2 = (\xi_2, \sigma, H) \\
 \text{(2)}~ \theta_1 &  = (\xi, \sigma, H_1), ~ \theta_2 = (\xi, \sigma, H_2)\\ 
 \text{(3)}~ \theta_1 & = (\xi, \sigma_1, H), ~ \theta_2 = (\xi, \sigma_2, H) .
\end{align*}

In the first case, we have by \cite[Eq. (5.32)]{panloup2020general} that
\begin{align*}
\frac{1}{t} \int_0^t | Y_s^{\theta_1} - Y_s^{\theta_2} |^2 ds \leq C | \xi_1- \xi_2 |^2 \left( 1+\sup_{\theta \in \Theta}  \frac{1}{t} \int_0^t | Y_s^\theta |^{2r} d s \right) ,
\end{align*}
where $r$ is the exponent in the sub-linear growth assumption on $b_\xi$. Since $r \leq 1$,
\begin{align}\label{eq:first-case}
\frac{1}{t} \int_0^t | Y_s^{\theta_1} - Y_s^{\theta_2} |^2 ds\leq C | \xi_1- \xi_2 |^2 \left( 1+ \sup_{\theta \in \Theta}  \frac{1}{t} \int_0^t | Y_s^\theta |^{2} d s \right) .
\end{align}
It follows from the uniform bound on the moments of $Y_t^\theta$ in Proposition \ref{prop:uniform-Lq-bounds}(ii) that there exists a random variable $\mathbf{C}$ with finite moment of order $p$ such that
\begin{align*}%
\frac{1}{t} \int_0^t | Y_s^{\theta_1} - Y_s^{\theta_2} |^2 ds\leq \mathbf{C} | \xi_1 - \xi_2 |^2 .
\end{align*}

The second case (2) is directly the result of Lemma \ref{lem:thm4.5}.

As for the third case (3), the idea is to compare the process $Y$ with the fractional OU processes $U^{(1,  \sigma_1,  H)}$ and $U^{(1,  \sigma_2,  H)}$ defined by \eqref{eq:fO-U} with the same initial condition and the same driving fBm. For $s \ge 1$, it comes
\begin{align*}
& \frac{\partial}{\partial s} | Y_s^{\theta_1}-Y_s^{\theta_2} - \left( U_s^{(1,  \sigma_1,  H)}-U_s^{(1,  \sigma_2,  H)} \right) |^2  \\ & \quad = 2 \langle  Y_s^{\theta_1}-Y_s^{\theta_2} - \left( U_s^{(1,  \sigma_1,  H)}-U_s^{(1,  \sigma_2,  H)} \right), b(Y_s^{\theta_1})-b(Y_s^{\theta_2}) + \left( U_s^{(1,  \sigma_1,  H)}- U_s^{(1,  \sigma_2,  H)} \right) \rangle \\
& \quad \leq - c_1 |Y_s^{\theta_1}-Y_s^{\theta_2} |^2 - c_2 |U_s^{(1,  \sigma_1,  H)}-U_s^{(1,  \sigma_2,  H)}|^2 + c_3 | Y_s^{\theta_1}-Y_s^{\theta_2}| | U_s^{(1,  \sigma_1,  H)}-U_s^{(1,  \sigma_2,  H)} |,
\end{align*}
where the last inequality follows from \ref{asmp:drift}. Next,  apply Young's inequality to get
\begin{align*}
& \frac{\partial}{\partial s} | Y_s^{\theta_1}-Y_s^{\theta_2} - \left( U_s^{(1,  \sigma_1,  H)}-U_s^{(1,  \sigma_2,  H)} \right) |^2 \\ & \quad \leq - c_1 |Y_s^{\theta_1}-Y_s^{\theta_2} |^2 - c_2 |U_s^{(1,  \sigma_1,  H)}-U_s^{(1,  \sigma_2,  H)}|^2 + c_3 | U_s^{(1,  \sigma_1,  H)}-U_s^{(1,  \sigma_2,  H)} |^2 \\
& \quad \leq -c_1  | Y_s^{\theta_1}-Y_s^{\theta_2} - \left( U_s^{(1,  \sigma_1,  H)}-U_s^{(1,  \sigma_2,  H)} \right) |^2 + c_2 | U_s^{(1,  \sigma_1,  H)}-U_s^{(1,  \sigma_2,  H)} |^2 .
\end{align*}
We can now apply Gr\"onwall's lemma to get
\begin{align*}
| Y_s^{\theta_1}-Y_s^{\theta_2} - \left( U_s^{(1,  \sigma_1,  H)}-U_s^{(1,  \sigma_2,  H)} \right) |^2 \leq C \int_0^s e^{-(s-u)} | U_u^{(1,  \sigma_1,  H)}-U_u^{(1,  \sigma_2,  H)} |^2 d  u .
\end{align*}
Jensen's inequality yields that
\begin{align*}
| Y_s^{\theta_1}-Y_s^{\theta_2} - \left( U_s^{(1,  \sigma_1,  H)}-U_s^{(1,  \sigma_2,  H)} \right) |^{2}  \leq C \int_0^s e^{-(s-u)} | U_u^{(1,  \sigma_1,  H)}-U_u^{(1,  \sigma_2,  H)} |^{2} d  u.
\end{align*}
Then, using Fubini's theorem, it comes that
\begin{align}
\frac{1}{t} \int_0^t | Y_s^{\theta_1}-Y_s^{\theta_2}  |^{2}  d s & \leq \frac{C}{t} \int_0^t | U_u^{(1,  \sigma_1,  H)}-U_u^{(1,  \sigma_2,  H)} |^{2} \int_u^t \mathds{1}_{[0,s]} e^{-(s-u)} d s d u \nonumber \\ & \quad +\frac{C}{t} \int_0^t | U_s^{(1,  \sigma_1,  H)}-U_s^{(1,  \sigma_2,  H)} |^2 d s  \nonumber \\
& \leq \frac{C}{t} \int_0^t | U_u^{(1,  \sigma_1,  H)}-U_u^{(1,  \sigma_2,  H)} |^{2} d u . \label{eq:third-case}
\end{align}
Now, observe that 
$U_s^{(1,  \sigma_1,  H)} - U_s^{(1,  \sigma_2,  H)} = (\sigma_1-\sigma_2) U_s^{(1, \Id, H)} $.
Since $U_s^{(1, \Id, H)}$ has finite moments uniformly in $\theta$ (recall \eqref{eq:fOUinvariant} and that $U^{(1, \Id, H)}$ is a Gaussian process), it follows that
\begin{align*}
\frac{1}{t} \int_0^t | Y_s^{\theta_1}-Y_s^{\theta_2}  |^{2}  d s  \leq \mathbf{C} | \sigma_1 - \sigma_2 |^2  ,
\end{align*}
where $\mathbf{C}$ has a finite moment of order  $p$. This concludes the proof of \eqref{eq:continuityM-theta}. 

\vspace{0.1cm}

The proof of \eqref{eq:continuityM-theta-discret} is obtained using discrete analogues of the 	 previous arguments. More precisely, in the first case (1), similarly to \eqref{eq:first-case}, we have from \cite[Proposition 3.8 (ii)]{panloup2020general} that
\begin{align*}
\frac{1}{N} \sum_{k=0}^N | Y_{k \gamma}^{\theta_1,\gamma} - Y_{k \gamma}^{\theta_2,\gamma} |^2 \leq C | \xi_1- \xi_2 |^2 \left( 1+ \sup_{\theta \in \Theta}  \frac{1}{N} \sum_{k=0}^{N} | Y_{k \gamma}^{\theta,\gamma} |^{2}  \right) .
\end{align*}
While the dependence of the RHS on $\sigma$ and $H$ in \cite{panloup2020general} is not explicit, one can show that the upper bound they obtain in the continuous setting (i.e. \cite[Eq. (5.32)]{panloup2020general}) still holds if the integrals are replaced by discrete sums. Then conclude using the uniform bound on the moments of $Y^{\theta,\gamma}$ in Proposition \ref{prop:uniform-Lq-bounds-L}(ii). 

The second case (2) is directly the result of Lemma \ref{lem:thm5.2}. In the third case (3), similarly to \eqref{eq:third-case}, via a comparison with the discrete fractional OU process $U^{\theta,\gamma}$ (that solves \eqref{eq:euler-scheme} with $b_\xi(\cdot)=-\xi \cdot$), one can show that
\begin{align*}
\frac{1}{N} \sum_{k=0}^N | Y_{k \gamma}^{\theta_1,\gamma}-Y_{k \gamma}^{\theta_2,\gamma}  |^{2}    & \leq \frac{C}{N} \sum_{k=0}^N| U_{k \gamma}^{(1,  \sigma_1,  H), \gamma}-U_{k \gamma}^{(1,  \sigma_2,  H), \gamma} |^{2} . 
\end{align*}
Then we use the linearity of $U^{(1,\sigma,H),\gamma}$ in $\sigma$ to conclude.
\end{proof}

\section{Proof of Proposition \ref{prop:dconv-euler}}\label{subsec:proofof4.1}
The proof follows the same steps as the proof of Lemma \ref{lem:contrast-conv}. Let $\theta = (\xi,\sigma,H) \in \Theta$. We will first prove that almost surely, the random measure $\frac{1}{t} \int_0^t \delta_{X^{\theta,\gamma}_s} ds$ converges in law to $\mu_{\theta}^\gamma$ as $t \rightarrow \infty$. This implies that $\frac{1}{t} \int_0^t \delta_{X^{\theta,\gamma}_s} ds$ converges to $\mu_{\theta}$ in the Prokhorov distance. To extend this result to distances $d$ in $\mathcal{D}_2$ (i.e dominated by the $2$-Wasserstein distance), we use the fact that the $2$-Wasserstein distance is dominated by the Prokorov distance $d_P$ as in \eqref{eq:boundProkhorov}:
\begin{align*}
d\left(\frac{1}{t} \int_0^t \delta_{X^{\theta,\gamma}_s} ds, \mu_{\theta}\right) & \leq C_p \sup_{t \ge 0} \left( \max\left(\frac{1}{t} \int_0^t |X^{\theta,\gamma}_s|^2 ds ,\, \mathbb{E} |\bar{X}^{\theta,\gamma}_t|^2\right) + 1 \right) \\ & \quad \times d_P\left(\frac{1}{t}\int_0^t \delta_{X^{\theta,\gamma}_s} ds, \mu_{\theta}\right). 
\end{align*}
By definition of the process $X^{\theta,\gamma}$, we have that
\begin{align*}%
\max \left( \frac{1}{t} \int_0^t |X^{\theta,\gamma}_s|^2 ds , \, \mathbb{E} |\bar{X}^{\theta,\gamma}_t|^2 \right) \leq C_{q} \sum_{i=0}^{q}  \max\left( \frac{1}{t} \int_0^t |Y^{\theta,\gamma}_{s+ih}|^2 ds , \, \mathbb{E} |\bar{Y}^{\theta,\gamma}_{s+ih}|^2 \right)  .
\end{align*}
Therefore, we conclude thanks to Proposition \ref{prop:uniform-Lq-bounds-L} that in the present case, the convergence in law is equivalent to the convergence for the $2$-Wasserstein distance. Similarly to Section \ref{subsec:contrast_conv}, we consider a family of probability measures on the set of c\`adl\`ag functions for which the identification of the limit will be easier, namely $\{ \pi_N = \frac{1}{N} \sum_{k=0}^{N-1} \delta_{X_{k \gamma+.}^{\theta, \gamma}}  \}_{N \ge 0}$. We first prove that the family is tight and then identify the limit as the stationary law of the augmented process $\bar{X}^{\theta,\gamma}$. 
Tightness in $D(\mathbb{R}_+, \mathbb{R}^d)$, the space of functions that are right-continuous and have limits from the left is equivalent to tightness in $D([0,T], \mathbb{R}^d)$ for every  $T>0$. Thus by \cite[Theorem 13.2]{Billingsley}, tightness is equivalent to the following two points that must hold for any $T > 0$:
\begin{enumerate}
\item[(i)]  $\displaystyle \lim_{a \rightarrow \infty} \limsup_{N \rightarrow \infty}  \pi_N \Big( x: \sup_{t \in [0,T]} | x_t | \ge a \Big) = 0.$

\item[(ii)] Denote $w_T^{\prime}(x, \delta)=\inf _{\left\{t_i\right\}}\left\{\max _{i \leq r} \sup _{s, t \in\left[t_i, t_{i+1}\right)}\left|x_t-x_s\right|\right\}$. Then for any $\eta>0$, almost surely,
\begin{align*}
& \limsup _{\delta \rightarrow 0} \limsup _{N \rightarrow+\infty} \pi_N \left( x: w_T^{\prime} (x, \delta ) \ge \eta \right) 
 = \limsup _{\delta \rightarrow 0} \limsup _{N \rightarrow+\infty} \frac{1}{N} \sum_{k=0}^{N-1} \mathds{1}_{\left\{w_T^{\prime}\left(X^{\theta,\gamma}_{k\gamma+\cdot}, \delta\right) \geq \eta\right\}}=0 ,
\end{align*}
where the infimum runs over finite sets $\left\{t_i\right\}_{i=1,\dots, r}, r\in \N^*$, satisfying
$$
0=t_0<t_1<\cdots<t_r=T \quad \text { and } \quad \inf _{i \leq r}\left(t_i-t_{i-1}\right) \geq \delta .
$$
\end{enumerate}
Since the process has only jumps at times $n \gamma$ with $n \in \mathbb{N}$,  $w_T^{\prime}\left(X^{\theta,\gamma}, \delta\right)=0$ when $\delta<\gamma$, which implies that the second condition (ii) holds. 

The first condition (i) is equivalent to tightness in the space of probability measures on $\R$ of the sequence $(\mu_T^{N})_{N\in \N^*}$ defined by 
\begin{equation*}
\mu_T^{N}=\frac{1}{N} \sum_{k=0}^{N-1} \delta_{\left\{\sup _{t \in[0, T]}\left|X_{k\gamma+t}^{\theta,\gamma} \right|\right\}}.
\end{equation*}
Recall that by definition of $X$ we have $| X^{\theta,\gamma}_{k \gamma} |^2 \leq C_q \sum_{i=0}^q | Y^{\theta,\gamma}_{k \gamma+ih} |^2 \leq C_q \sup_{i \in \llbracket 0,q \rrbracket} | Y^{\theta,\gamma}_{k \gamma+ih} |^2$. Hence, for $V(x)=|x|^2$ and
\begin{align*}
 \tilde{\mu}^N_{T} = \frac{1}{N} \sum_{k=0}^{N-1} \delta_{\left\{ \sup_{t \in [0,T+qh]} \left| Y^{\theta, \gamma}_{k \gamma + t} \right| \right\} },
\end{align*}
we deduce that
$\mu_T^{N}(V)  \leq C_q  \, \tilde{\mu}^N_{T}(V)$. 
From the last equation in the proof of \cite[Proposition 2]{cohen2011approximation}, one has $\sup_{N \ge 1} \tilde{\mu}^N_T (V) <+ \infty$ almost surely, which implies that $(\mu_T^{N})_{N \ge 1}$ is a.s. tight on $\mathbb{R}$ (see \emph{e.g.} \cite[Proposition 2.1.6]{Duflo}).

Now let $(t_n)_{n \ge 1}$ be an increasing sequence going to $+\infty$ and $\{\frac{1}{t_n}  \sum_{k=0}^{t_n-1}  \delta_{X^{\theta,\gamma}_{k \gamma+\cdot}}  \}_{n \ge 1}$ be a (pathwise) sequence with limiting distribution $\rho$. 
As in Appendix A.2 of \cite[Proposition 3.3]{panloup2020general}, we get that $\gamma$ is stationary.
Let us now prove that $\rho$ is the law of $\bar{X}^{\theta,\gamma}$.

A process ${x_t=(y_t,z^1_t,\dots,z^{q}_t)}$ has the law of $\bar{X}^{\theta,\gamma}$ if $x_t = x_{k \gamma}$ for $t \in [k \gamma, (k+1) \gamma]$, and
\begin{align*}
  &  y_\cdot  - y_0 - \int_0^{\cdot_\gamma} b_\xi(y_u) du ~\text{has the law of a $\sigma B_{\cdot_\gamma} $ where B has Hurst parameter $H$};\\
 &   z^i_\cdot  - \ell^i \left( \int_0^{\cdot_\gamma} b_\xi(y_u) du , \dots, \int_0^{(\cdot+ih)_\gamma} b_\xi(y_u) du \right) ~\text{has the law of}~ \sigma \ell^i(B_{\cdot_\gamma},\dots,B_{(\cdot+ih)_\gamma}) \\ & \quad \text{for all $i \in \llbracket 1,q \rrbracket$},
\end{align*}
where for all $t \ge 0$, $t_\gamma = \gamma \lfloor t/\gamma \rfloor$. Now one can proceed as in the end of Section \ref{subsec:contrast_conv} to deduce that $\rho$ is the law of $\bar{X}^{\theta,\gamma}$.

\end{appendix}


\end{document}